\documentclass[10pt,reqno]{amsart}
\usepackage[utf8]{inputenc}
\usepackage[T1]{fontenc}
\usepackage{lmodern}
\usepackage{xspace}
\usepackage{amsfonts,amsmath,amssymb,amsthm}
\usepackage[small]{titlesec}
\usepackage{comment}
\usepackage{fancyhdr}
\usepackage{color}
\usepackage[a4paper,body={160mm,235mm},centering]{geometry}
\newcommand{\rset}{\mathbb{R}}
\newcommand{\R}{\mathbb{R}}

\newcommand{\ep}{\varepsilon}
\newcommand{\ind}{\mathbf{1}}
\newcommand{\fl}{\longrightarrow}
\newcommand{\e}{\mathbb{E}}
\newcommand{\E}{\mathbb{E}}
\renewcommand{\d}{\,d}
\newcommand{\p}{\mathbb{P}}
\newcommand{\lp}{\mathrm{L}}
\newcommand{\m}{\mathcal} 
\newcommand{\bx}{{\bf x}}
\newcommand{\by}{{\bf y}}
\renewcommand{\H}{{\mathcal H}}
\newcommand{\M}{{\mathcal M}}

\newcommand{\ys}[1][2]{\mathcal{S}^{#1}}
\newcommand{\zs}[1][2]{\mathrm{M}^{#1}}
\newcommand{\ld}[1][\mu]{\text{D}_{#1}}
\newcommand{\dotafter}[1]{#1.}
\titleformat{\section}[hang]
{\normalfont\large\bfseries}{\thesection.}{.5em}{\dotafter}[]
\titleformat{\subsection}[runin]
{\normalfont\bfseries}{\thesubsection.}{.4em}{}[.]
\titlespacing*{\subsection}{0pt}{3ex plus 1ex minus .2ex}{1em}
\titleformat{\paragraph}[runin]{\normalfont\bfseries}{\theparagraph.}{.4em}{}[.]
\theoremstyle{plain}
\newtheorem{thm}{Theorem}
\newtheorem{lemme}[thm]{Lemma}
\newtheorem{prop}[thm]{Proposition}
\newtheorem{cor}[thm]{Corollary}

\theoremstyle{definition}
\newtheorem{df}[thm]{Definition}

\theoremstyle{remark}
\newtheorem{rem}[thm]{Remark}

\newcommand \A[1]{{\bf (#1)}}
\def\pe{\color{black}}
\def\pee{\color{black}}
\def\yin{\color{black}}
\def\pierre{\color{black}}
\def\ph{\color{black}}

\newcommand{\TL}{{{\pe \tilde \Lambda}}}
\newcommand{\tl}{{{\pe \tilde \lambda}}}

\newcommand{\phrase}{{\pee \emph{work on the enlarged filtered probability space of Remark \ref{itoprod}}}}

\title{\bf Forward and Backward Stochastic Differential Equations with Normal Constraints in Law and associated $L$-PDEs}
\date{Work in Progress}
\begin{document}
\title[Constrained in law (b)sde]{Forward and Backward Stochastic Differential Equations with Normal Constraints in Law}
\author[Ph. Briand]{Philippe Briand}
\address{Univ. Grenoble Alpes, Univ. Savoie Mont Blanc, CNRS, LAMA, 73000 Chamb\'ery, France}
\email{philippe.briand@univ-smb.fr}

\author[P. Cardaliaguet]{Pierre Cardaliaguet}
\address{Universit\'e Paris-Dauphine, PSL Research University, CNRS, Ceremade, 75016 Paris, France.}
\email{cardaliaguet@ceremade.dauphine.fr}

\author[P.-É. Chaudru de Raynal]{Paul-\'Eric Chaudru de Raynal}
\address{Univ. Grenoble Alpes, Univ. Savoie Mont Blanc, CNRS, LAMA, 73000 Chamb\'ery, France}
\email{pe.deraynal@univ-smb.fr}

\author[Y. Hu]{Ying Hu}
\address{Univ Rennes, CNRS, IRMAR-UMR6625, F-35000, Rennes, France.}
\email{ying.hu@univ-rennes1.fr}

\thanks{{\color{black} The second  and fourth  authors have been partially supported by the ANR project ANR-16-CE40-0015-01.} For the third Author, this work has been partially supported by the ANR project ANR-15-IDEX-02 }

\date{}

\begin{abstract}
In this paper we investigate the well-posedness of backward or forward stochastic differential equations whose law is constrained to live in an a priori given (smooth enough) set and which is reflected along the corresponding ``normal'' vector.  We also study the  associated interacting particle system reflected in mean field  and asymptotically described by such equations. The case of particles submitted to a common noise as well as the asymptotic system is  studied in the forward case. Eventually, we connect the forward and backward stochastic differential equations with normal constraints in law with partial differential equations stated on the Wasserstein space and involving a Neumann condition in the forward case  and a{\yin n o}bstacle in the backward one. 
\end{abstract}

\maketitle
\thispagestyle{fancy}

\bigskip
\section{Introduction}
In this paper, we are concerned with reflected (forward or backward) Stochastic Differential Equations (SDE) in the case where the constraint is on the law of the solution rather than on its paths. This kind of equations have been introduced in their backward form in \cite{BEH1} \textcolor{black}{in the scalar case and when the constraint is {\ph of the form $\int h \d \mu \geq 0$ for some map $h : \R \to \R$ satisfying suitable assumptions and where $\mu$ denotes the law of the considering process}. Such a system being reflected according to the mean of (a functional of) the process, the authors called it a \emph{Mean Reflected Backward Stochastic Differential Equation} (MR BSDE).} In \cite{BCdRGL1}, the authors studied the forward version (hence called MR SDE) \textcolor{black}{in the same setting} as well as \textcolor{black}{its approximation by an appropriate interacting particle system} and numerical schemes. {\yin In \cite{Jab}, weak solution to related forward equations with constraint\footnote{Therein, the author consider SDE with constraint in law but without the Skorokhod condition.} are built}. In the same framework, let us also mention the work \cite{BH17} where the approximation of MR BSDE by an interacting particle system is studied, \cite{HHYHYLPLFL17} where MR BSDE with quadratic generator are investigated {\ph and \cite{BGL1} where MR SDE with jumps are considered}.

The aim of this work consists in enlarging the results of \cite{BCdRGL1, BEH1, BH17} to the multi-dimensional case and for rather general constraint sets on the law in the backward and forward cases. Mean-field interacting particles counterpart of such systems are also investigated as well as the so called \emph{common noise} setting. Eventually, we also aim at introducing the deterministic counterpart of such reflected stochastic system through Partial Differential Equation (PDE) stated on Wasserstein space of Neumann (in the forward case) or {\yin o}bstacle (in the backward case) type.\\

\subsection{Outline of the paper}
\subsubsection*{On SDE with normal constraint in law.} \textcolor{black}{Given coefficients\footnote{which could be assumed to be random, see the results below.} $(b,\sigma) : \R_+\times \R^n \to \R^n \times \R^{n\times d}$,} a MR SDE is an equation of the following form
\begin{equation}\label{eq:MRSDE-scalarCons2}
	\left\{
	\begin{split}
	& X_t  =X_0+\int_0^t b(s,X_s)\, ds + \int_0^t \sigma(s,X_s)\,dB_s + K_t,\quad t\geq 0, \\
	& \e[h(X_t)] \geq 0, \quad \int_0^t \e[h(X_s)] \, dK_s = 0, \quad t\geq 0{\yin ,}
	\end{split}
	\right.
\end{equation}
where $B$ is a Brownian motion, $h$ is a 
{\ph function} from $\rset^n$ to $\rset$, the law of the initial condition $X_0$ is such that $\e[h(X_0)] \geq 0$ {\pee and where, for the time being, $n=d=1$}. 

When focusing on this forward system, there are several ways to understand the mean reflected SDE. One striking example lies into the equation satisfied by the law of the solution $X$, which turns out to be a reflected Fokker-Planck equation. In other words, solving the above system translates into searching for solution to the Skorohod problem stated on Partial Differential Equation of Fokker-Planck type. Indeed, let $(X,K)$ be a solution of the above system (which, according to \cite{BCdRGL1}, {\ph under suitable assumptions}, exists, is unique and where the process $K$ is supposed to be a deterministic increasing process starting from 0). By It\^{o}'s formula, the law $\mu_t=[X_t]$ of $X_t$ satisfies, in the sense of distributions,  
\begin{equation*}
	\left\{
	\begin{split}
& d_t\mu_t(x)= \Big\{ D_x(\mu_t(x) b(t,x))+ \frac12 D^2_x(\mu_t (x)a(t,x))\Big\}dt + D_x\left(\mu_t(x)\right)\,dK_t,\\
& \int h\, d\mu_t \geq 0\quad \int_0^t \Big(\int h\, d\mu_s\Big) \, dK_s = 0, \quad t\geq 0 ,
	\end{split}
	\right.
\end{equation*}
where $a=\sigma^2$. Note that, as this system is deterministic, the deterministic assumption assumed on the reflexion term $K_t$ really makes sense.\\

From this perspective it is natural to look for a solution to the more general problem (allowing $d$ and $n$ to be greater than 1):
\begin{equation}\label{eq.KolmDet}
	\left\{
	\begin{split}
& d_t\mu_t(x)= \Big\{ {\rm div}(\mu_t(x) b(t,x))+ \frac12 \sum_{i,j}D^2_{i,j}(\mu_t (x)a_{ij}(t,x))\Big\}dt + {\rm div}\left(\mu_t(x)\ld H(\mu_t)(x)\right)\,dK_t,\\
& H(\mu_t)\geq 0\quad \int_0^t H(\mu_s) \, dK_s = 0, \quad t\geq 0 ,
	\end{split}
	\right.
\end{equation}
where $H$ is a map from $\m P^2 (\rset^n)$ to $\rset$ (where we denote by $\m P^2 (\rset^n)$  the set of Borel probability measures on $\R^n$ with finite second order moment) and $\ld H$ denotes the Lions' derivative (cf. \cite{CD17-1,LLperso} and the discussion at the end of the Introduction). Of course, when $H(\mu) = \int_{\rset^n} h(x)\, \mu(dx)$, $\ld H(\mu)(x)=\nabla h(x)$. The map $\ld H(\mu)(\cdot)$ can be viewed as a gradient of $H$ in the space ${\mathcal P}(\R^n)$ (see \cite{CD17-1,LLperso} and the discussion below), so that the ``outward normal'' to the set ${\mathcal O}:=\{\mu,\; H(\mu)>0\}$ at a point $\mu\in \partial {\mathcal O}$  is, at least formally,  $\ld H(\mu)(\cdot)$. For this reason \eqref{eq.KolmDet} can be understood as the Fokker-Planck equation with a {\it normal reflexion}. The probabilistic counterpart of \eqref{eq.KolmDet} then writes as the following reflected SDE:
\begin{equation}\label{eq:mains}
	\left\{
	\begin{split}
	& X_t  =X_0+\int_0^t b(s,X_s)\, ds + \int_0^t \sigma(s,X_s)\,dB_s + \int_0^t \ld H([X_s])(X_s)\,dK_s,\quad t\geq 0, \\
	& H(\left[X_t\right]) \geq 0, \quad \int_0^t H(\left[X_s\right]) \, dK_s = 0, \quad t\geq 0.
	\end{split}
	\right.
\end{equation}
Indeed, if one denotes by $(X,K)$ a solution of the above system, by It\^{o}'s formula, the law $\mu_t=[X_t]$ of $X_t$ satisfies system  \eqref{eq.KolmDet} in the sense of distributions. The above problem is actually no longer a ``mean reflected'' SDE but a SDE with \emph{normal constraint in law} since the constraint can now be written in a general form.

In comparison with \cite{BCdRGL1}, we are led to deal with general constraints and possibly multi-dimensional valued processes. Considering this more general setting explains why we now specify the direction of the reflection (which is done along the ``outward normal'') while in  \cite{BCdRGL1} the reflection is oblique along the unitary vector. Indeed, when working in this framework, the explicit formula obtained in the aforementioned work for the process $K$ is not valid anymore, so that we cannot ``just add it'' to ensure that the constraint is satisfied.\\

In \eqref{eq.KolmDet} the Fokker-Planck equation is deterministic, which explains why the reflection term $K$ is naturally deterministic. A further generalization of \eqref{eq.KolmDet} consists in considering instead a {\it  stochastic} Fokker-Planck equation reflected on the set ${\mathcal O}=\{\mu, \; H(\mu)>0\}$. Given the coefficients $(b,\sigma_0,\sigma_1) : \R_+\times \R^n \to \R^n \times \R^{n\times d}\times \R^{n\times d}$ and a $d-$dimensional Brownian motion $W$, one may look for a solution to the system
\begin{equation}\label{eq.KolmSto}
\left\{\begin{array}{l}
\displaystyle d_t\mu_t(x)= \left\{ {\rm div}(\mu_t(x) b(t,x))+ \frac12 \sum_{i,j}D^2_{i,j}(\mu_t (x)a_{ij}(t,x))\right\}dt \\
\displaystyle  \qquad \qquad \qquad \qquad \qquad \qquad \qquad + {\rm div} \left(\mu_t(x) \sigma_1(t,x)dW_t\right)
+ {\rm div}\left(\mu_t(x)\ld H(\mu_t)(x)\right)\,dK_t,\\
\displaystyle  H(\mu_t)\geq 0\quad \int_0^t H(\mu_s) \, dK_s = 0, \quad t\geq 0 {\yin,}
\end{array}\right.
\end{equation}
where $a=(\sigma_0\sigma_0^*+\sigma_1\sigma_1^*)$ and $(K_s)$ is a continuous nondecreasing process adapted to the filtration $\m F^W$ associated with $W$. The way we added the stochastic perturbation in the above reflected stochastic Fokker-Planck equation relies on the probabilistic interpretation of this problem. Indeed, given a Brownian motion $B$ (supposed to be independent of $W$), the above system precisely gives the dynamic of the \emph{conditional} law (conditioned by $\m F^W$) $\mu_t=[X_t|W]$ of the solution of the system 
\begin{equation}\label{eq:mainsCond_intro}
	\left\{
	\begin{split}
	& X_t  =X_0+\int_0^t b(s,X_s)\, ds + \int_0^t \sigma_0(s,X_s)\,dB_s +\int_0^t \sigma_1(s,X_s)dW_s+ \int_0^t \ld H([X_s|W])(X_s)\,dK_s,
	\\
	& H(\left[X_t|W\right]) \geq 0, \quad \int_0^t H(\left[X_s|W\right]) \, dK_s = 0, \quad t\geq 0.
	\end{split}
	\right.
\end{equation}
This is a  reflected SDE with \emph{normal constraint on its conditional law}. For reasons that will be clear in few lines (and relies on the mean field interacting particles counterpart of such a system), we call this case {\it the normal reflexion in law with {\pe a} common noise}.\\

Let us now discuss the interpretation of \eqref{eq:mainsCond_intro} (or \eqref{eq:mains}) at the particles level. Consider the system $(X^i_t)$ of $N$ particles, where, for $i=1, \dots, N$,
\begin{equation}\label{eq:mainps_cond_Intro}	
	\left\{
	\begin{split}
	 & X^i_t  =X^i_0+\int_0^t b(s,X^i_s)\, ds + \int_0^t \sigma_0(s,X^i_s)\,dB^i_s + \int_0^t \sigma_1(s,X^i_s)\,dW_s + \int_0^t  \ld H\left(\mu^N_s\right)(X^i_s)\, dK^N_s, \\
	 & \mu^N_t = \frac{1}{N} \sum_{i=1}^N \delta_{X^i_t}, \quad H\left(\mu^N_t\right) \geq 0, \quad  \quad \int_0^t H\left(\mu^N_s\right) \, dK^N_s = 0, \quad t\geq 0.
	\end{split}
	\right.
\end{equation}
Here the $B^i$ and $W$ are independent Brownian motions. {\yin The initial conditions of the particles $X^i_0$ are i.i.d. random variable with law $\mu_0$ and are independent of the $B^i$ and $W$. Eventually, $K^N$ is a continuous, nondecreasing process adapted to the filtration $\m F^N$ generated by the $B^i$, $X_0^i$ and $W$}.  Assuming that $H(\mu_0)>0$, this system is, conditionally to $\Omega_N=\{H(\mu_0^N) \geq 0\}$, a classical reflected SDE in $(\R^{n})^N$, with  normal reflexion on the boundary of the constraint 
\begin{equation*}
{\mathcal O}_N=\left\{(x_1,\dots, x_N)\in (\R^n)^N, \; H\left(\frac{1}{N} \sum_{i=1}^N \delta_{x_i}\right)>  0\right\}. 
\end{equation*}
Following \cite{LLperso}, we understand $W$ as {\it a common noise}, since it  affects all the players. Note that we have chosen to describe the mean field limit in the common noise case, but it reduces to the case without common noise if we let $\sigma_1\equiv 0$. It turns out that the limit system of \eqref{eq:mainps_cond_Intro} is nothing but our SDE with normal reflexion in its condition law \eqref{eq:mainsCond_intro}: as the number of players tends to  infinity, the reflection term $K^N$ in \eqref{eq:mainps_cond_Intro}   no longer  depends on the  position of the other particles, but only on their statistical distribution conditioned by the common noise.

Let us just emphasize that, above, the law of the initial condition is assumed to  satisfy the constraint in order to work conditionally to the event $\Omega_N$. This assumption relies on the fact that $H(\mu_0) \geq 0$ is not a sufficient condition to ensure that $H(\mu_0^N) \geq 0$. In order to remove such a strict inequality, one may either work conditionally to the event $\Omega_{N,\eta}=\{H(\mu_0^N) \geq -\eta_N\}$ which, provided {\yin that }$(\eta_N)_{N\geq 0}$ is chosen appropriately, becomes of full measure as $N$ tends to the infinity (so that the particle system solves the ``original'' Skorohod problem asymptotically) or either construct suitable initial conditions $\tilde X_0^i$  whose empirical measure $\tilde \mu_0^N$ satisfies (under suitable assumptions) $H(\tilde \mu_0^N)\geq 0$ and which is such that $\tilde \mu_0^N \to \mu$ as $N \to \infty$. In this last case, for any $N \geq 1$, system \eqref{eq:mainps_cond_Intro} with $\tilde X_0^i$ instead of $X_0^i$ is a solution of the Skorohod problem in mean field on ${\mathcal O}_N$. We refer to Sections \ref{subsec:A_mean_field_limit_cond} and \ref{SEC_MF_FOR_BSDE} respectively for more details.\\

We eventually conclude this outline on SDE with normal constraint in law by giving another interpretation of \eqref{eq:mainps_cond_Intro} (or of \eqref{eq:mains} when $\sigma_1\equiv 0$) through the dynamic of the family of mapping $(U_t)_{t\geq 0}$ from $\mathcal P^2(\R^n) \to \R$ defined by $\forall t_0 \geq 0,\ U_{t_0}(\mu_0) = G([X_T^{t_0,\mu_0}|W])$, for some given $G:\mathcal P^2(\R^n)\to \R$. Enlarging naively the definition of viscosity solution on finite dimensional space, one can show that the dynamic of $U$ is given (in the viscosity sense) by the the following (backward) Neumann  problem in the set ${\mathcal O}=\{\mu, \; H(\mu)>0\}$:
\begin{equation}\label{lqbzlsdCond_intro}
	\left\{ 
		\begin{split} 
		(i)\;\;\; &-\partial_t U(t,\mu)-\frac12\int_{\R^n} {\rm Tr}\left((a(t,y)) \partial_y \ld U(t,\mu)(y)\right)\mu(dy)\\ 
		& \qquad  -\int_{\R^n} \ld U(t,\mu)(y)\cdot b(t,y)\mu(dy)\\
		& \qquad - \frac12 \int_{\R^n\times \R^n} {\rm Tr}\left(D^2_{\mu\mu} U(t,\mu)(x,y)\sigma_1(t,x)\sigma_1^T(t,y)\right)\mu(\d x)\mu(\d y) =0\\		
		& \qquad \qquad \qquad {\rm in }\; (0,T)\times {\mathcal O},\\
		(ii)\;\;&\int_{\R^n} \ld U(t,\mu)(y)\cdot \ld H(\mu)(y)\mu(dy)= 0 \; {\rm in }\; (0,T)\times \partial{\mathcal O},\\
		(iii)\;& U(T,\mu)=G(\mu) \; {\rm in }\; {\mathcal O}{\yin ,}
		 \end{split}
	\right. 
\end{equation}
where $a=\sigma_0\sigma_0^*+\sigma_1\sigma_1^*$. 
Condition (ii) is exactly the  Neumann boundary condition associated with the set ${\mathcal O}$. Conversely, any smooth solution to \eqref{lqbzlsdCond_intro} can be written as ${\yin \E\left[G([X_T^{t,\mu}|W])\right]}$.\\

\subsubsection*{BSDE with normal constraint in law.} The above discussion leads to investigate whenever the original constrained in law problem stated on BSDE from \cite{BEH1} can be generalized. We here aim at showing that the generalization done in the case of forward SDE remains valid, up to the \emph{common noise} setting (and under a slightly different set of assumptions, see the discussion in the next paragraph). Concerning this last framework, it seems that the results remain valid, provided additional assumptions (insuring suitable boundedness conditions on $Z$) are satisfied. Nevertheless we do not explore it in order to short the current paper.

As underlined in the previous paragraph, we aim at extending the results of \cite{BEH1} in a multi-dimensional setting. Following \eqref{eq:mains}, we are  lead to investigate  the problem
\begin{equation}\left \lbrace \begin{array}{ll}\label{eq:mainb_intro}
&\displaystyle  Y_t  =\xi+\int_t^T f(s,Y_s,Z_s)\, ds - \int_t^T Z_s\,dB_s + \int_t^T \ld H([Y_s])(Y_s) dK_s,\quad 0\leq t\leq T, \\
&\displaystyle  H([Y_t]) \geq 0, \quad 0\leq t\leq T, \qquad \int_0^T H([Y_s])  dK_s = 0,
\end{array}\right.
\end{equation}
where $\xi \in \mathbb L_2(\mathcal{F}_T)$, where the processes $Y,Z$ are respectively of $n$ and $n \times d$ dimension and where $f:\Omega \times \rset_+\times\rset^n \times \rset^{ n \times d}  \fl \rset^n $. Here again, $H$ is a map from $\m P^2 (\rset^n)$ to $\rset$ and $\ld H$ denotes the Lions' derivative so that the ``outward normal'' to the set ${\mathcal O}:=\{\mu,\; H(\mu)>0\}$ at a point $\mu\in \partial {\mathcal O}$  is again, at least formally,  $\ld H(\mu)(\cdot)$. For this reason \eqref{eq:mainb_intro} is now called a \emph{BSDE with normal constraint in law}.

Also we aim at investigating the mean field counterpart {\pe of} such {\pe a} system, extending to the current framework the results obtained in \cite{BH17}. Namely, we consider the interacting particle system
\begin{equation}\label{eq:backward:part_intro}\left \lbrace\begin{array}{ll}
\displaystyle Y_t^i = \xi^i  + \int_t^T f(s,Y_s^i,Z_s^{i,i}) ds - \int_t^T \sum_{j=1}^N Z_s^{i,j} d B_s^j + \int_t^T D_\mu H(\mu_s^N)(Y_s^i) d K_s^N,\\
\displaystyle \forall t \in [{\yin 0},T]:\quad \mu_t^N = \frac{1}{N} \sum_{i=1}^N \delta_{Y_t^i},\quad H(\mu_t^N) \geq 0,\quad \int_0^T H(\mu_s^N) d K_s^N =0,\quad 1 \leq i \leq N,
\end{array}\right.
\end{equation}
where for each $i,j,$ $Z_s^{i,j}$ is a $n\times d$ matrix, $\{B^i\}_{1\leq i \leq N}$ are $N$ independent $d$-dimensional Brownian motions and $K^N$ is a continuous non decreasing process. Assuming that $H(\mu_T)>0$ this system is, conditionally to $\Omega_N=\{H(\mu_T^N) \geq 0\}$, a classical reflected BSDE in $(\R^{n})^N$, with  normal reflexion on the boundary of the constraint
\begin{equation*}
{\mathcal O}_N=\left\{(y_1,\dots, y_N)\in (\R^n)^N, \; H\left(\frac{1}{N} \sum_{i=1}^N \delta_{y_i}\right) > 0\right\},
\end{equation*}
and \eqref{eq:mainb_intro} is precisely the asymptotic dynamic (as $N \to +\infty$) of one of the particles in \eqref{eq:backward:part_intro}. As in the forward case, the strict inequality assumed on the law of the terminal condition comes from the the fact that $H(\mu_T) \geq 0$ is not a sufficient condition to ensure that $H(\mu_T^N) \geq 0$. We refer to the corresponding point in the above discussion on forward SDE with constraint in normal law and to Sections \ref{subsec:A_mean_field_limit_cond} and \ref{SEC_MF_FOR_BSDE} respectively for more details to handle properly such a problem.\\

Eventually, one may wonder if, {\ph for the one dimensional case when $Y$ is a real process} and in a Markovian set up (\emph{i.e.} when $\xi = \phi(X_T)$ for some $\phi :\R^{{\yin n}} \to \R$ and some diffusion process $X$ evolving according to coefficients $b,\sigma : \R_+\times \R^{{\yin n}} \to \R$ and to the scalar valued Brownian motion $B$), the solution of the system \eqref{eq:mainb_intro} can be written in term of the position $X$ and its law thanks to a \emph{decoupling field} $u:\R_+\times \R^{{\yin n}} \times \mathcal{P}^2(\R^{{\yin n}}) \to \R$ and if the dynamic of such a decoupling field is given thanks to an obstacle problem, as shown in \cite{EKKPPQ97} when the reflection is on the path. Assuming that the generator $f$ is deterministic we show that such a decoupling field exists and, when $f$ does not depend on the $Z$ argument, we obtain that $u$ solves, in the viscosity sense, the following obstacle problem on the Wasserstein space:
\begin{equation}\label{eq:obstacle_intro}\left\lbrace\begin{array}{lll}
\displaystyle \min  \bigg\{\big\{(\partial_t + \mathcal{L} )u(t,x,\mu) + f(t,x,u(t,x,\mu))\big\}; \\
\displaystyle \quad  H({\color{black}  u(t,\cdot,\mu)\sharp \mu} )\bigg\}= 0,\quad \mbox{on}\quad [0,T)\times \R^n \times \mathcal{P}^2(\R^n),\\
u(T,\cdot,\cdot) = \phi,
\end{array}\right.
\end{equation}
where {\yin for any probability measures $\nu$ and mesurable function $\varphi$, $\varphi \sharp \nu$ denotes the push-forward of the measure $\nu$ by the map $\varphi$ and where}
$\mathcal{L}$ is given by, for all smooth $\varphi : [0,T]{\pe \times \R^n} \times \mathcal P^2(\R^n) \to \R${\pe
\begin{eqnarray*}
\mathcal{L}\varphi(t,x,\mu)&=&\frac12\int_{\R^n} {\rm Tr}\left[((\sigma\sigma^*)(t,y)) \partial_y \ld \varphi(\mu)(t,x,\mu)(y)\right]\mu(dy) +\int_{\R^n} \ld \varphi(\mu)(t,x,\mu)(y)\cdot b(t,y)\mu(dy)\\
&&+ \frac12{\rm Tr}\left[((\sigma\sigma^*)(t,x)) D_x^2 \varphi(t,x,\mu)\right] +D_x\varphi(t,x,\mu)\cdot b(t,x)
\end{eqnarray*}
}

\subsubsection*{Main results.} We now describe our main results. In all the cases (forward setting with and without common noise and backward setting), the existence part differs from the ones in \cite{BEH1} or \cite{BCdRGL1}: we build the solution by a penalization technique inspired from \cite{LiSz, LiMeSz}. Also, and still in all the cases considered, our approach for the existence part heavily relies on suitable bi-Lipschitz property of the map $H$ which is assumed to hold in a neighbourhood of the {\pe boundary of the} constraint set.\\

Concerning the forward equations, we first prove in Theorem \ref{en:euss} that \eqref{eq:MRSDE-scalarCons2} is well posed {\yin when $K_t = \int_0^t \nabla h(X_s)d|K|_s$} for stochastic Lipschitz in space coefficients bounded in $\mathbb{L}_2$ at point 0 uniformly in time. We start by handling this specific case for the following reasons: firstly, this allows to introduce the main tools needed for our analysis in a simple setting; secondly, the proof is done under some concavity assumption assumed on $H(\mu) = \int h \d \mu$ (which will be no longer assumed for the rest of the forward part) that allows to work with diffusion coefficients not necessarily globally bounded in space (and then enables to sketch what will be usefull in the backward setting); thirdly, this allows to weaken a little bit the bi-{\yin L}ipschitz assumption on the constraint function $H$ (through the map $h$) and finally, thanks to a suitable transformation, this allows to connect the results obtained in this case of normal reflection with the previous results of \cite{BCdRGL1} where the reflection is oblique. 

We then enter our framework and we prove in Theorem \ref{en:eus} and Theorem \ref{thm:Cond} the well-posedness of \eqref{eq:mains} and \eqref{eq:mainsCond_intro} respectively: we show that, under suitable assumptions on the data, there exists a unique solution to these equations. {\pe We emphasize that in both cases the uniqueness is shown to hold pathwise and in law.} We present both results separately because they involve different sets of assumptions: for instance, for \eqref{eq:mains}, we can allow the diffusion $b$ to be unbounded while it is not the case for problems with common noise. The case with common noise also requires stronger assumptions on the constraint $H$. One reason for this is that, as can be seen on the Kolmogorov equations,  \eqref{eq.KolmDet} (which is associated with \eqref{eq:mains})  is a {\it deterministic} reflexion problem, while \eqref{eq.KolmSto} (associated with \eqref{eq:mainsCond_intro}) involves a diffusion term. As a consequence, for \eqref{eq:mains}, the process $(K_t)$ is deterministic and Lipschitz continuous, while it is random and merely continuous for \eqref{eq:mainsCond_intro}. 

{\yin Let us again emphasize that in a recent paper \cite{Jab}, Jabir considers process very similar to \eqref{eq:mains} (i.e. satisfying same type of constraint but not enunciated or either identified as a solution of Skorohod problem on the law). Using a different penalization approach and working under weaker regularity assumptions, provided that the constraint set is convex and the diffusion does not degenerate, he builds weak solutions to the problem  through tightness of the law of the penalization. However no uniqueness result is stated in \cite{Jab} and propagation of chaos as well as the common noise setting are not discussed.}

Next, we present in Theorem \ref{thm:chaosS_cond} a conditional propagation of chaos for the particle system \eqref{eq:mainps_cond_Intro}. In other words, we show that  system \eqref{eq:mainps_cond_Intro} converges, as $N\to+\infty$, to the reflected SDEs  with common noise \eqref{eq:mainsCond_intro}. Let us recall that \cite{BCdRGL1} already studied the propagation of chaos in the one dimensional setting. The main differences with  \cite{BCdRGL1} is that we work here with a normal reflexion (in \cite{BCdRGL1}  it was oblique) and in a common noise setup. {\pe The convergence rate} are also explicitly given.

Eventually, we make the connection with the Neumann problem on the Wasserstein space.  This connection remains at a basic level, as we only prove is that the map $U$ defined as $U(t_0,[X_0])= {\yin \E\left[G([X_T|W])\right]}$ is \emph{a} viscosity solution of \eqref{lqbzlsdCond_intro} in Theorem \ref{thm.viscsol} and that any smooth solution of \eqref{lqbzlsdCond_intro} can be written as $U(t_0,[X_0])= {\yin \E\left[G([X_T|W])\right]}$ in Proposition \ref{prop.FK.cond}. Here, the notion of viscosity solution we adopt consists in natural extension of the finite dimensional definition by testing the map $U$ again{\pe st} smooth functions from $\mathcal{P}_2(\R^n)$ to $\R$, the smoothness being understood in the Lions's sense. In that respect, we follow the definition proposed in \cite{CD17-2} and extend it to the case of Neumann condition. Nevertheless, we do not prove a uniqueness result which, in infinite dimensional metric space, appears to be involved and definitely out of the scope of this paper. Let us just mention that, in that case, it seems to us that the map $U$ could be shown to be the unique classical solution of \eqref{lqbzlsdCond_intro}, avoiding then the difficulties relying on the notion of viscosity solution. Again, such an investigation exceeds the scope of this paper.\\

Concerning the Backward equation, we first give in Theorem \ref{thm:existence_uniqueness} an existence result for \eqref{eq:mainb_intro}. This result relies on the concavity property assumed on $H$ in the current setting and is done through a three steps scheme: we first tackle the case of deterministic bounded and space homogeneous driver; we then extend the result to a stochastic generator in $L^2(\Omega,L^2([0,T])$ and eventually to a space dependent generator. As in the forward case without common noise, the process $(K_t)$ is shown to enjoy a Lipschitz property. {\pe Again, uniqueness holds pathwise{\yin ly} and in law.}

{\pe Then} we construct in Section \ref{SEC_MF_FOR_BSDE} an interacting particle system satisfying a Skorohod problem in mean field of type \eqref{eq:backward:part_intro} and prove in Theorem \ref{th:estiparticleBSDE} that the interacting particle system converges to \eqref{eq:mainb_intro} (meaning that \eqref{eq:mainb_intro} precisely gives the asymptotic dynamic of one of the interacting particles). Again, we specify the rate at which system \eqref{eq:backward:part_intro} converges to \eqref{eq:mainb_intro}.

Finally, thanks to our well-posedness result, we investigate the obstacle problem \eqref{eq:obstacle_intro} through (decoupled) Forward-Backward SDE (FBSDE) with normal constraint in law (on the backward part). When the driver $f$ no longer depends on the $Z-$variable, we show in Theorem \ref{THM:Obstacle} that the FBSDE with normal constraint in law admits a decoupling field $u$ which satisfies, in the viscosity sense, the PDE \eqref{eq:obstacle_intro}. The reason why the driver is not allowed to depend on the variable $Z$ relies to the lack of comparison principle when considering non-linear equations which makes the usual procedure to tackle this term unavailable. For the same reasons as pointed out for the Neumann problem, we do not prove any uniqueness result of viscosity solutions and let it for future considerations.

\subsection{Organization of this work} The paper is organized in the following way. We complete the introduction by describing our notations, discussing the notion of Lions's derivative and presenting the various It\^{o}'s formulas used in the text {\pe as well as the tools needed to ensure the convergence of the particle systems}. Section \ref{sec:sdes_with_normal_reflexion_in_law} is devoted to the  SDEs with normal reflexion in law: starting with the simplest setting \eqref{eq:MRSDE-scalarCons2} to illustrate the method of proof, we then show the existence and the uniqueness of the solution of \eqref{eq:mains} (Theorem \ref{en:eus}). SDEs with normal reflexion in conditional law (i.e., the ``common noise" case) are studied in 
Section \ref{sec:conditional_law}, where we prove the well-posedness of \eqref{eq:mainsCond_intro} (Theorem \ref{thm:Cond}), the propagation of chaos (Theorem \ref{thm:chaosS_cond}) and make the link with the Neumann problem (Theorem \ref{thm.viscsol} and Proposition \ref{prop.FK.cond}).  Finally, the case of backward equation is treated in Section \ref{sec:the_backward_case}, where we also prove the existence and the uniqueness of the solution to \eqref{eq:mainb_intro} (Theorem \ref{thm:existence_uniqueness}), the propagation of chaos (Theorem \ref{th:estiparticleBSDE}) and a Feynman-Kac formula (Proposition \ref{THM:Obstacle}).

\subsection{Notations and Mathematical tools}
Throughout the paper we will make an intensive use of Lions's derivative, introduced in \cite{LLperso} and later discussed in \cite{BLP09, CDLL, CD17-1, CCD14}. Let us recall that other notions of derivatives in the space of measure have been developed in the literature: see for instance \cite{AGS, Ko10, MiMo}. Related to Lions's derivatives, we also need several It\^{o}'s formulas, allowing to compute the value of a map depending on the law (or the conditional law) of a process. These formulas have been developed in various stages of generality in  \cite{BLP09, CDLL, CD17-1, CCD14}. {\pe Nevertheless, it is not yet clearly stated that this formula holds for reflected processes, especially when they are submitted to a common noise as it is the case in \eqref{eq:mainsCond_intro}. Therefore, we prove below that such  It\^{o}'s formulas can be generalized to our framework. Eventually, in order to ensure the convergence of our particle system \eqref{eq:mainps_cond_Intro} and \eqref{eq:backward:part_intro}, we provide at the end of this part a Lemma allowing to control the uniform convergence in 2-Wasserstein metric of an empirical measure of conditionally i.i.d. processes.}

\subsubsection*{Notations.} 

Throughout the paper, we work on a complete probability space $(\Omega, \m F,\p)$ with a complete and right-continuous filtration $\{\m F_t\}_{t\geq 0}$. We denote by $\m E$ the $\sigma$-algebra of progressive sets of $\Omega\times \rset_+$ and let $B$ and $W$ be two independent $d$-dimensional Brownian motion with respect to this filtration. We denote by ${\mathcal F}^W$ the completed filtration generated by $W$ and, given a random variable $X$ on $(\Omega, \m F,\p)$, let $[X]$ be its law and  $[X|W]$ be its law given ${\mathcal F}^W$. 

Let $\m P(\R^n)$ be the set of Borel probability measures on $\R^n$ and{\color{black}, given $q\geq 1$, let $\m P^q(\R^n)$ be} the subset of measures $m$ in $\m P(\R^n)$ with finite $q$ order moment: 
$$
\M_q(\mu) := \int_{\R^n} |x|^q\mu(dx)<+\infty{\yin ,} \qquad \mu\in \m P^q(\R^n).
$$
We denote by $W_2(\mu,\mu')$ the Wasserstein distance between two measures of $\m P^2(\R^n)$. Let us recall that 
$$
W_2(\mu,\mu')= \inf_{X,X'} \left(\E\left[|X-X'|^2\right]\right)^{1/2},
$$
where the infimum is taken over all random variables ${\yin (X,X')}$ such that $[X]= \mu$, $[X']=\mu'$.\\

\subsubsection*{Derivatives.} Given a map $H:\m P^2(\R^n)\to \R$, we denote by $\ld H: \m P^2(\R^n)\times \R^n\to \R^n$ its derivative, when it exists, in Lions's sense (or ``L-derivative"): see \cite{LLperso} or Definition 5.22 in \cite{CD17-1}. Let us recall that $H$ has an L-derivative at a measure $\mu_0\in \m P^2(\R^n)$ if there exists a random variable $X_0$ with $[X_0]=\mu_0$ at which the lifted map $\tilde H: L^2(\Omega)\to \R$, defined by $\tilde H(X)=H([X])$, is differentiable. If $H$ is L-differentiable on $\m P^2(\R^n)$, then the derivative of $\tilde H$  takes the form $\nabla \tilde H(X)= \ld H([X])(X)$ (\cite{LLperso} or \cite[Proposition 5.25]{CD17-1}). As explained in \cite[Remark 5.27]{CD17-1}, the derivative $\ld H$ allows to quantify the Lipschitz regularity of $H$. Namely, $H$ is Lipschitz continuous on $\m P^2(\R^n)$ if and only if {\pe 
$$
\H_p:= \sup_{\mu\in \m P^p(\R^n)} \int_{\R^n} | \ld H(\mu)(x)|^p\mu(dx)<+\infty,
$$
with $p=2$.} In this case, $\H_2$ is the Lipschitz constant of $H$. We will often require that the map $\nabla \tilde H$ itself is Lipschitz continuous on $L^2(\Omega)$: 
in order words, 
$$
\forall X,Y\in L^2(\Omega), \qquad \E\left[ \left| \ld H([X])(X)- \ld H([Y])(Y)\right|^2\right]\leq C^2 \E\left[|X-Y|^2\right].
$$
Then it is proved in \cite[Proposition 5.36]{CD17-1} that (up to redefining $\ld H$), 
$$
\forall \mu \in \m P^2(\R^n), \; \forall x,x'\in \R^n, \qquad\left| \ld H(\mu)(x)-\ld H(\mu)(x')\right|\leq C|x-x'|.
$$

\subsection*{Second order derivatives} We now discuss further regularity of $H$. We say that $H$ is {\it partially $C^2$} if the map $y\to \ld H(\mu)(y)$ is continuously differentiable with a derivative $(\mu, y)\to \partial_y \ld H(\mu)(y)$ jointly continuous on $\m P^2(\R^n)\times \R^n$ (see also \cite[Definition 5.95]{CD17-1}). We say that $H$ is {\it globally $C^2$} if it is partially $C^2$ and if the map $\mu \to \ld H(\mu)(y)$ is $L-$differentiable with $(\mu,y,y')\to \ld^2 H(\mu)(y,y')$ continuous (see also \cite[Definition 5.82]{CD17-1}).

\subsubsection*{It\^{o}'s formulas.} When $H$ is partially $C^2$, It\^{o}'s formula holds for the law of a diffusion of the form 
$$
dX_t =b_t dt +\sigma_t dB_t
$$
where $(b_t)$ and $(\sigma_t)$ are progressively measurable with values in $\R^n$ and $\R^{n\times d}$ respectively and satisfy 
$$
\E\left[ \int_0^T (|b_s|^2+|\sigma_s|^4)ds \right]<+\infty.
$$
Namely, if $H$ is partially $C^2$ with 
\begin{equation}\label{hypLdH}
\sup_{ \mu\in \m P^2(\R^n)} 
\int_{\R^n} \Bigl| 
\partial_y\ld H(\mu)(y)
\Bigr|^2
\mu(dy) <+\infty, 
\end{equation}
then 
\begin{equation}\label{I1}
H([X_t])= H([X_0]) + \int_0^t\E\left[ \ld H([X_s])(X_s)\cdot b_s\right]ds +\frac12 \int_0^t\E\left[ {\rm Tr} \left(a_s \partial_y \ld H([X_s])(X_s)\right)\right]ds, 
\end{equation}
where $a_s=\sigma_s\sigma_s^*$   \cite[Theorem 5.98]{CD17-1}.  

A similar It\^{o}'s formula holds for the {\it conditional law} of an It\^{o}'s process of the form, 
$$
dX_t= b_tdt+ \sigma^0_tdB_t+ \sigma^1_tdW_t{\yin ,}
$$
where $B$ and $W$ are independent $d-$dimensional Brownian motions living on different probability spaces $(\Omega^0, \m F^0,\p^0)$ and $(\Omega^1, \m F^1,\p^1)$ and where $b$, $\sigma^0$ and $\sigma^1$ are progressively measurable with respect to the filtration generated by $W$ and $B$, with 
$$
\E\left[ \int_0^T (|b_s|^2+|\sigma^0_s|^4+|\sigma^1_s|^4)ds \right]<+\infty.
$$
We assume that $H$ is globally $C^2$ with 
$$
\sup_{\mu\in \m P^2} \int_{\R^n} \left| \ld H(\mu)(x)\right|^2 \mu(dx) + \int_{\R^n} \left| \partial_x \ld H(\mu)(x)\right|^2 \mu(dx) + 
\int_{\R^n\times \R^n} \left|D^2_{\mu\mu} H(\mu)(x,y)\right|^2\mu(dx)\mu(dy) <+\infty.
$$
Then, letting $\mu_t(\omega^1)=[X_t|W](\omega^1)$, we have, $\p^1-$a.s.,  
\begin{align} 
H(\mu_t) & = H(\mu_0)+\int_0^t \E^0\left[ \ld H(\mu_s)(X_s)\cdot b_s\right] ds
	+ \int_0^t \E^0\left[(\sigma^1_s)^*\ld H(\mu_s)(X_s)\right]\cdot dW_s\label{I2}\\
	& \qquad + \frac12 \int_0^t  \E^0\left[ {\rm Tr}\left(a_s \partial_y\ld H(\mu_s)(X_s))\right)\right] ds \notag
	+ \frac12 \int_0^t \E^0\tilde \E^0\left[{\rm Tr}\left(D^2_{\mu\mu}H(\mu_s)(X_s,\tilde X_s)\sigma^1_s(\tilde \sigma^1_s)^*\right)\right] ds\notag
\end{align}
where $\tilde X$ and $\tilde \sigma^1$ are independent copies of $X$ and $\sigma^1$ {\yin is }defined on the space {\pe $(\tilde \Omega^0\times \tilde \Omega^1, \tilde \p^0\otimes \tilde \p^1)$}, while $a_s:= (\sigma^0_s( \sigma^0_s)^*+\sigma^1_s(\sigma^1_s)^*)$. 
See \cite[Theorem 11.13]{CD17-2}. 

{\pierre We actually need below more general versions of It\^{o}'s formulas \eqref{I1} and \eqref{I2}  in which there is an additional drift term $\beta_s dK_s$. In the generalization of \eqref{I1}, we assume that  $\beta$ is a continuous process adapted to the filtration generated by $B$ and that $K$ is a deterministic, continuous and nondecreasing process process with $K_0=0$. We are then interested in the law $(\mu_t)$ of an It\^{o} process 
$$
dX_t= b_tdt+ \sigma_tdB_t+\beta_tdK_t. 
$$
Assuming, in addition to the conditions on $b$, $\sigma$ and $H$ given above, that 
$$
\E\left[ \sup_{t\in [0,T]} |\beta_t|^2 \right]<+\infty, 
$$
we have 
\begin{align}
H([X_t]) = & H([X_0]) + \int_0^t\E\left[ \ld H([X_s])(X_s)\cdot b_s\right]ds +\int_0^t\E\left[ \ld H([X_s])(X_s)\cdot \beta_s\right]dK_s \label{I3} \\ 
& \qquad + \frac12 \int_0^t\E\left[ {\rm Tr} \left(a_s \partial_y \ld H([X_s])(X_s)\right)\right]ds, \notag
\end{align}
where $a_s=\sigma_s\sigma_s^*$.  For the generalization of \eqref{I2}, we assume instead that $\beta$ is a continuous process, adapted to the filtration generated by $B$ and $W$ and that $K$ is a continuous and nondecreasing process adapted to the filtration generated by $W$ with $K_0=0$. We also assume 
{\color{black} 
\begin{equation}\label{hypbetaK}
\E\left[ \sup_{t\in [0,T]} \E^0\left[ |\beta_t|^2\right] K_T^2 \right]<+\infty, 
\end{equation}
} We are then interested in the conditional law, given $W$, of the process 
$$
dX_t= b_tdt+ \sigma^0_tdB_t+ \sigma^1_tdW_t+\beta_tdK_t. 
$$
Under the above conditions on $b$, $\sigma^0$, $\sigma^1$, $H$, $\beta$ and $K$, we have, $\p^1-$a.s.,  
\begin{align} 
H(\mu_t) & = H(\mu_0)+\int_0^t \E^0\left[ \ld H(\mu_s)(X_s)\cdot b_s\right] ds 
	+ \int_0^t \E^0\left[(\sigma^1_s)^*\ld H(\mu_s)(X_s)\right]\cdot dW_s\notag \\
	& \qquad + \int_0^t \E^0\left[ \ld H(\mu_s)(X_s)\cdot \beta_s\right] dK_s
	+ \frac12 \int_0^t  \E^0\left[ {\rm Tr}\left(a_s \partial_y\ld H(\mu_s)(X_s))\right)\right] ds \label{I4} \\
	& \qquad + \frac12 \int_0^t \E^0\tilde \E^0\left[{\rm Tr}\left(D^2_{\mu\mu}H(\mu_s)(X_s,\tilde X_s)\sigma^1_s(\tilde \sigma^1_s)^*\right)\right] ds. \notag 
\end{align}
\begin{proof}[Proof of  \eqref{I3} and \eqref{I4}]  Formula \eqref{I3} and \eqref{I4} can  be derived from \eqref{I1} and \eqref{I2} respectively by regularizing the process $K$. We give the details for \eqref{I4} to fix the idea,  as \eqref{I3} can be proved with the same argument. We first extend $K$ to $\R$ by setting $K_t=0$ for $t\leq 0$ and then let $K^\ep:= \phi^\ep \ast K$, where $\phi^\ep(s)= \ep^{-1}\phi(s/\ep)$, $\phi$ being a smooth nonnegative kernel with compact support in $\R_+$. Then, $K^\ep$ is a smooth nondecreasing process adapted to $W$.  Let us set 
$$
X^\ep_t= X_0+\int_0^t  b_sds+ \int_0^t \sigma^0_sdB_s+ \int_0^t \sigma^1_sdW_s+\int_0^t \beta_s (K^\ep)'_sds
$$
and $\mu^\ep_t =[X^\ep_t|W]$.  By \eqref{hypbetaK}  {\color{black} and the fact that $K$ is $\m F^W$ adapted, we have  
\begin{align*}
& \E\left[ \int_0^T |\beta_s (K^\ep)'_s|^2ds \right] \leq \|(\phi^\ep)'\|_\infty \E\left[ \sup_{0\leq s\leq T}\E^0\left[  |\beta_s|^2 \right]  |K_T|^2 \right] <+\infty.
\end{align*}
} So, by It\^{o}'s formula  \eqref{I2}, we  obtain
\begin{align*} 
H(\mu^\ep_t) & = H(\mu_0)+\int_0^t \E^0\left[ \ld H(\mu^\ep_s)(X^\ep_s)\cdot b_s\right] ds 
	+ \int_0^t \E^0\left[(\sigma^1_s)^*\ld H(\mu^\ep_s)(X^\ep_s)\right]\cdot dW_s \\
	& \qquad + \int_0^t \E^0\left[ \ld H(\mu^\ep_s)(X^\ep_s)\cdot \beta_s\right] (K^\ep)'_sds 
	+ \frac12 \int_0^t  \E^0\left[ {\rm Tr}\left(a_s \partial_y\ld H(\mu^\ep_s)(X^\ep_s))\right)\right] ds \\
	& \qquad + \frac12 \int_0^t \E^0\tilde \E^0\left[{\rm Tr}\left(D^2_{\mu\mu}H(\mu^\ep_s)(X^\ep_s,\tilde X^\ep_s)\sigma^1_s(\tilde \sigma^1_s)^*\right)\right] ds.  
\end{align*}
In order to  let $\ep\to 0$, we first prove the uniform convergence of $X^\ep$ to $X$. We note that 
$$
\E\left[ \sup_{0\leq t\leq T} \left|X^\ep_t-X_t\right|^2\right] \leq \E\left[\sup_{0\leq t\leq T}  \left| \int_0^t (\beta_s- \int_s^t \beta_u \phi^\ep(u-s)du) dK_s\right|^2 \right], 
$$
where 
$$
\sup_{0\leq t\leq T}  \left| \int_0^t (\beta_s- \int_s^t \beta_u \phi^\ep(u-s)du) dK_s\right|^2 \leq  2 \sup_{0\leq s\leq T}|\beta_s|^2 K_T^2, 
$$
 the right-hand side being integrable by \eqref{hypbetaK}. Fix $\delta>0$ small. By the continuity of $\beta$,  
 we have, uniformly in $s\in (0,t-\delta)$,  
 $$
 \lim_{\ep\to 0}  \int_s^t \beta_u \phi^\ep(u-s)du = \lim_{\ep\to 0}  \int_\R \beta_u \phi^\ep(u-s)du = \beta_s.
$$
So, a.s.,  
$$
\lim_{\ep\to 0} \sup_{\delta \leq t\leq T}  \left| \int_0^{t-\delta} (\beta_s- \int_s^t \beta_u \phi^\ep(u-s)du) dK_s\right|^2 = 0.
$$
On the other hand, 
$$
 \left| \int_{(t-\delta)\vee 0}^t (\beta_s- \int_s^t \beta_u \phi^\ep(u-s)du) dK_s\right|\leq 2 \sup_{0\leq u\leq T}|\beta_u|^2 (K_t-K_{(t-\delta)\vee 0}).
 $$
So, by the continuity of $K$, 
$$
\lim_{\delta \to 0} \limsup_{\ep\to 0}  \sup_{0 \leq t\leq T}  \left| {\color{black} \int_{(t-\delta)\vee 0}^t } (\beta_s- \int_s^t \beta_u \phi^\ep(u-s)du) dK_s\right|^2=0. 
$$
We can then infer by dominate convergence  that  
\begin{equation}\label{mecznregpdfv0} 
\lim_{\ep\to0} \E\left[ \sup_{0\leq t\leq T}\left|X^\ep_t-X_t\right|^2\right] =0 \qquad \mbox{\rm and thus} \qquad 
\lim_{\ep\to 0}\E\left[ \sup_{0\leq t\leq T} {\mathcal W}_2^2(\mu^\ep_t,\mu_t) \right]=0. 
\end{equation}
By our assumptions on the data, one easily derive from this the convergence in $L_2(\Omega)$ of 
\begin{align*} 
H(\mu^\ep_t) &- H(\mu_0)-\int_0^t \E^0\left[ \ld H(\mu^\ep_s)(X^\ep_s)\cdot b_s\right] ds 
	- \int_0^t \E^0\left[(\sigma^1_s)^*\ld H(\mu^\ep_s)(X^\ep_s)\right]\cdot dW_s \\
&  	- \frac12 \int_0^t  \E^0\left[ {\rm Tr}\left(a_s \partial_y\ld H(\mu^\ep_s)(X^\ep_s))\right)\right] ds 
	 - \frac12 \int_0^t \E^0\tilde \E^0\left[{\rm Tr}\left(D^2_{\mu\mu}H(\mu^\ep_s)(X^\ep_s,\tilde X^\ep_s)\sigma^1_s(\tilde \sigma^1_s)^*\right)\right] ds  
\end{align*}
to 
\begin{align*} 
H(\mu_t) & - H(\mu_0)+\int_0^t \E^0\left[ \ld H(\mu_s)(X_s)\cdot b_s\right] ds 
	- \int_0^t \E^0\left[(\sigma^1_s)^*\ld H(\mu_s)(X_s)\right]\cdot dW_s \\
&  	- \frac12 \int_0^t  \E^0\left[ {\rm Tr}\left(a_s \partial_y\ld H(\mu_s)(X_s))\right)\right] ds 
	- \frac12 \int_0^t \E^0\tilde \E^0\left[{\rm Tr}\left(D^2_{\mu\mu}H(\mu_s)(X_s,\tilde X_s)\sigma^1_s(\tilde \sigma^1_s)^*\right)\right] ds  {\color{black}.}
\end{align*}
To complete the proof we just need to check that, in $L_2(\Omega)$,  
\begin{equation}\label{mecznregpdfv}
\lim_{\ep\to 0}  \int_0^t \E^0\left[ \ld H(\mu^\ep_s)(X^\ep_s)\cdot \beta_s\right] (K^\ep)'_sds =  \int_0^t \E^0\left[ \ld H(\mu_s)(X_s)\cdot \beta_s\right] dK_s.
\end{equation}
Indeed 
\begin{align*}
& \int_0^t \E^0\left[ \ld H(\mu^\ep_s)(X^\ep_s)\cdot \beta_s\right] (K^\ep)'_s \d s = \int_0^t \Bigl( \int_r^t \E^0\left[ \ld H(\mu^\ep_s)(X^\ep_s)\cdot \beta_s\right]  \phi_\ep(s-r)ds\Bigr) \d K_r, 
\end{align*}
where, by Cauchy-Schwartz and our assumption on  $\ld H$, 
\begin{align*}
& \left| \int_0^t \Bigl( \int_r^t \E^0\left[ \ld H(\mu^\ep_s)(X^\ep_s)\cdot \beta_s\right]  \phi_\ep(s-r)ds\Bigr) \d K_r\right| \\
& \; \leq \int_0^t  \Bigl( \int_r^t  \left(\E^0\left[ |\ld H(\mu^\ep_s)(X^\ep_s)|^2\right]\right)^{1/2} \left(\E^0[|\beta_s|^2]\right)^{1/2}   \phi_\ep(s-r)ds\Bigr) \d K_r
\leq 
C\left(\E^0\left[ \sup_{0\leq s\leq T} |\beta_s|^2\right] \right)^{1/2} K_T. 
\end{align*}
{\color{black} The right-hand side of the above inequality is in $L^2(\Omega)$ thanks to \eqref{hypbetaK}.} Using the uniform convergence of $s\to \E^0\left[ \ld H(\mu^\ep_s)(X^\ep_s)\cdot \beta_s\right]$ implied by \eqref{mecznregpdfv0}, we can conclude exactly as in the proof of \eqref{mecznregpdfv0} that \eqref{mecznregpdfv} holds. 
\end{proof}
}

{\pe

\noindent\textbf{On uniform convergence of conditionally  i.i.d. sequence of processes.} The aim of this part is to recast in the conditionally i.i.d. setting the result of \cite{RR98a} (see Theorem 10.2.7 therein) taking into account the recent results obtained in  \cite{FG15} and in Chapter 5 of \cite{CD17-1}. \\

\begin{lemme}\label{ConvPart} 
 Let $(\Omega^0,\mathcal{F}^0,\mathbb{P}^0)$ and $(\Omega^1,\mathcal{F}^1,\p^1)$ be two different probability spaces. Let $\{X^i_\cdot\}_{1 \leq i \leq N}$ be a sequence of processes defined on $(\Omega^0\times \Omega^1,\mathcal{F}^0\otimes \mathcal{F}^1,\p^0\otimes\p^1)$ satisfying that the $X^i$ are, conditionally to $\mathcal{F}^1$, i.i.d.  processes. Assume that there exists constants {\color{black} $C_0>0$,} $q>4$ and $p>2$ {\color{black} such that  
 $$
 {\ph\E^1\left[\sup_{t \in [0,1]} \e^0\left[|X_t|^q\right]^{2p/q} \right]\leq C_0}
 $$
  and}

\begin{eqnarray}
\e [ |X_s - X_r|^p |X_s-X_t|^p] &\leq & {\color{black} C_0} |t-r|^2,\quad {\rm for }\ 0 \leq r < s < t \leq 1;\notag\\
\e [|X_t-X_s|^p] &\leq &{\color{black}  C_0} |t-s|,\quad {\rm for }\ 0 \leq  s \leq  t \leq 1; \label{hypRacRu}\\
\forall 0 \leq s \leq t\leq 1,\quad \e\left[\sup_{u \in [s,t]}\left|\e^0 [|X_u - X_s |^2]\right|^p\right]& \leq& {\color{black} C_0}  |t-s|^{{\pee p/2}}. \notag 
\end{eqnarray}
Then, {\color{black} setting $\mu_s=[X_s|\mathcal{F}^1]$} and $\bar \mu^N_s = N^{-1} \sum_{i=1}^N \delta_{X_s^i}$, there exists a constant $C>0$, {\color{black} depending on $n$, $p$, $q$ 	and $C_0$ only,} such that
\begin{equation}
\E \left[\sup_{s \leq T} W_2^2(\bar \mu_{s}^N, \mu_{s}) \right]\leq C\epsilon_N:=C  \times \left\lbrace\begin{array}{ll} \displaystyle N^{{\yin -1/2+1/p}} \text{ if } d < 4,\\ \displaystyle  N^{{\yin -1/2+1/p}}(\log(1+N))^{{\yin 1-2/p}}  \text{ if } d=4,\\ \displaystyle N^{{\yin-  2(1-2/p)/d}}  \text{ if } d > 4. \end{array}\right.
\end{equation}
\end{lemme}
\begin{proof}
Using a time discretization $(t_k)_{{\yin 0} \le k \le m}$ with constant time step $1/m$ of the interval $[0;1]$ we obtain that 
\begin{equation*}
\sup_{s \leq 1} W_2^2(\bar \mu_s^N,\mu_s) \leq {\yin 3}\left\{ \max_{k}  \chi_k+ \max_{k} W_2^2(\bar \mu_{t_k}^N, \mu_{t_k})  + \max_{k}\sup_{t_k \leq t \leq t_{k+1}} W_2^{{\yin 2}}(\mu_{t_k},\mu_t) \right\},
\end{equation*}
where 
$$\chi_k = \sup_{t_k \leq t \leq t_{k+1}}\left\{ W_2^2(\bar \mu_t^N,\bar \mu_{t_k}^N) \wedge W_2^2(\bar \mu_t^N,\bar \mu_{t_{k+1}}^N)\right\}.$$
Now we have
\begin{eqnarray*}
\E\left[\max_{k}\sup_{t_k \leq t \leq t_{k+1}} W_2^2(\mu_{t_k},\mu_t)\right] &\leq & \E\left[\max_{k}\sup_{t_k \leq t \leq t_{k+1}} \E^0\left[|X_t-X_{t_k}|^2 \right]\right]\\
&\leq & \E^{1/p}\left[\max_{k}\sup_{t_k \leq t \leq t_{k+1}} \left|\E^0\left[|X_t-X_{t_k}|^2 \right]\right|^p\right]\\
&\leq & \E^{1/p}\left[\sum_{k}\sup_{t_k \leq t \leq t_{k+1}} \left|\E^0\left[|X_t-X_{t_k}|^2 \right]\right|^p\right]\\
\end{eqnarray*}
Therefore, using third line in  {\color{black} Assumption \eqref{hypRacRu}} we obtain
\begin{eqnarray}\label{controle_discretization}
\E\left[\max_{k}\sup_{t_k \leq t \leq t_{k+1}} W_2^2(\mu_{t_k},\mu_t)\right] &\leq & {\yin C} m^{-1/2+1/p}{\yin .}
\end{eqnarray}
Following now the proof of Theorem 10.2.7 of \cite{RR98a} we have
\begin{equation}\label{tochangeincommonnoise2}
\E\left[\max_{k} \chi_k \right]\leq \frac{C}{\sqrt{m}},
\end{equation}
which will lead to a negligible contribution. 
%
Concerning the last term we write

\begin{equation}
\e \left[  \max_{k} W_2^2(\bar \mu_{t_k}^N, \mu_{t_k})  \right] \leq m^{1/p} \left(\max_k \e \left[ W_2^{2p}(\bar \mu_{t_k}^N, \mu_{t_k}) \right]\right)^{1/p}.
\end{equation}
Since $(\E^0)^{2p/q}[|X_t|^q|]<+\infty$ $\mathbb{P}^1-$a.s. for some $q>4$, we have, from Theorem 2 case (3) of \cite{FG15} (see also Theorem 5.8 and Remark 5.9 in \cite{CD17-1}) that there exists a deterministic constant $c >0$ such that $\p^1-$ a.s. 
\begin{equation*}
W_2^{2p}(\bar \mu_{t_k}^N, \mu_{t_k})\leq c(\E^0)^{2p/q}[|X_{t_k}|^q|]  \left\lbrace\begin{array}{ll} \displaystyle N^{-p/2} \text{ if } d \leq 3\\
\displaystyle N^{-p/2}{\pee (\log(1+N))^p} \text{ if } d = 4\\
 \displaystyle N^{-2p/d}  \text{ if } d > 4 \end{array}\right.
\end{equation*}
Hence,
\begin{equation}\label{tochangeincommonnoise}
\left(\E \left[\max_{k} W_2^{2p}(\bar \mu_{t_k}^N, \mu_{t_k}) \right]\right)^{1/p}\leq c \E\left[\sup_{0\leq t \leq 1}(\E^0)^{2p/q}[|X_t|^q|] \right] \left\lbrace\begin{array}{lll} \displaystyle N^{-1/2} \text{ if } d \leq 3\\ 
\displaystyle N^{-1/2}\left(\log(1+N)\right) \text{ if } d = 4\\
\displaystyle N^{-2/d}  \text{ if } d > 4 \end{array}\right.
\end{equation}
We conclude by optimizing over $m$.
\end{proof}
\begin{rem}\label{remconvpart}
We emphasize that, when working with i.i.d. sequence of processes, the above result  {\color{black} holds} assuming only that
$$\e\left[|X_s-X_t|^2\right] \leq |t-s|,\quad \text{for all } 0\leq s \leq t \leq 1,$$
instead of the third line in \eqref{hypRacRu}. In such case the control in \eqref{controle_discretization} is direct and is $ m^{-1}$ (this follows essentially from the fact that in that case we have $\E = \E^0$ so that $p$ could be, formally, $\infty$). {\pee The resulting rate of convergence are then given by the $\epsilon_N$ defined above with $p=\infty$ therein up to the ad hoc constant.}
\end{rem}

}

\section{SDEs with normal reflexion in law} 
\label{sec:sdes_with_normal_reflexion_in_law}

In this section, we are  concerned with the existence and uniqueness for the solution of~\eqref{eq:mains}. To avoid unnecessary repetitions, we postpone the discussion of the particle system and of the Feynman-Kac formula to the more general setting of the SDE with normal reflexion in  conditional law. As a warm up we start by studying{\yin
\begin{equation}\label{eq:MRSDE-scalarCons}
	\left\{
	\begin{split}
	& X_t  =X_0+\int_0^t b(s,X_s)\, ds + \int_0^t \sigma(s,X_s)\,dB_s + \int_0^t \nabla h(X_s) dK_s,\quad t\geq 0, \\
	& \e[h(X_t)] \geq 0, \quad \int_0^t \e[h(X_s)] \, dK_s = 0, \quad t\geq 0{\yin ,}
	\end{split}
	\right.
\end{equation}
which is the \emph{normally reflected} version of \eqref{eq:MRSDE-scalarCons2}. 
}

\subsection{Warm-up} 
\label{sub:warmup}
In this subsection, we study the case where $H(\mu)$ is given by $\int h d\mu$ in the spirit of the works~\cite{BEH1}, \cite{BCdRGL1}. Roughly speaking, we extend the results of the previous papers to the multidimensional setting.

In the following, we say that assumptions \A{Hwu} hold when the following assumptions are in force:
\begin{description}
	\item[(Hc)] The functions $b:\Omega\times\rset_+\times\rset^n \fl \rset^n$ and $\sigma : \Omega\times \rset_+\times\rset^{n} \fl \rset^{n\times d}$ are measurable with respect to $\m E \otimes \m B\left(\rset^n\right)$ and 
	\begin{enumerate}
		\item[(i)] For all $T>0$, there exists $L_T$ such that, $\p$-a.s., for each $t\in[0,T]$,
		\begin{equation*}
			\forall x\in\rset^n, \forall y\in\rset^n,\qquad |b(t,x)-b(t,y)| + |\sigma(t,x)-\sigma(t,y)| \leq L_T\, |x-y| ;
		\end{equation*} 
		\item[(ii)] For all $T>0$,
		\begin{equation*}
			\sup_{t\leq T} \e\left[|b(t,0)|^2 + |\sigma(t,0)|^2\right] <+\infty.
		\end{equation*}
	\end{enumerate}
	\item[(Hh)] The function $h:\rset^n \fl \rset$ is $\m C^2$ with $|\nabla^2 h|_\infty < \infty$  and 
	\begin{enumerate}
		\item[(i)] $h$ is concave;
		\item[(ii)] For all $x\in\rset^n$, $|\nabla h(x)|^2 > 0$;
		\item[(iii)] There exists $x_0\in\rset^n$ such that $h(x_0) \geq 0$.
	\end{enumerate}
	\item[(H0)] The initial condition is $\m F_0$--measurable, square integrable and $H([X_0])\geq 0${\yin .}
\end{description}
Before going further, let us recall that, in this subsection, $H([X_0])=\e[h(X_0)]$.

\begin{rem}
	Let us observe that the assumption \A{Hh}-(iii) is needed to ensure that the set
	\begin{equation*}
		\left\{\mu\in \m P^2(\rset^n) : H(\mu) = \int h\, d\mu \geq 0\right\}
	\end{equation*}
	is not empty. For example, if $h(x) = -e^{-x}$ this set is empty. Moreover, when {\bf (Hh)} holds true the set
	\begin{equation*}
		\left\{\mu\in \m P^2(\rset^n) : H(\mu) = \int h\, d\mu > 0\right\}
	\end{equation*}
	is also non empty since it contains, in particular, a Dirac mass. Indeed, let $(x_t)_{t\geq 0}$ be the solution to
	\begin{equation*}
		x'(t) = \nabla {\pe h}(x(t))\, dt, \quad t\geq 0, \qquad x(0)=x_0.
	\end{equation*}
	Then, for $t>0$,
	\begin{equation*}
		h(x(t)) = h(x_0) + \int_0^t |\nabla h(x(s))|^2 ds {\yin\, \geq} \int_0^t |\nabla h(x(s))|^2 ds >0.
	\end{equation*}
\end{rem}

\begin{rem}
	Let us point out that the concavity of $h$ allows to weaken the assumption on $\nabla h$ used in \cite{BEH1} and \cite{BCdRGL1} namely: there exists $\beta >0$ such that
	\begin{equation*}
		\forall x\in \rset^n,\qquad |\nabla h(x)|^2 \geq \beta.
	\end{equation*}
	We only assume here that $|\nabla h|^2$ does not vanish instead of being bounded from below by some positive constant. Moreover, the gradient of $h$ is only supposed to have a linear growth instead of being bounded as assumed in the two papers mentioned above.
	
	To conclude the comparison with \cite{BEH1, BCdRGL1}, let us mention that, in the one dimensional case, one can switch from the framework of the two quoted papers, namely
	\begin{equation*}
		\left\{
		\begin{split}
		& P_t  =X_0+\int_0^t d(s,P_s)\, ds + \int_0^t S(s,P_s)\,dB_s + K_t,\quad t\geq 0, \\
		& \e[l(P_t)] \geq 0, \quad \int_0^t \e[l(P_s)] \, dK_s = 0, \quad t\geq 0,
		\end{split}
		\right.
	\end{equation*}
	with $K$ deterministic to the setting of \eqref{eq:MRSDE-scalarCons} by the transformation $P_t = F(X_t)$ where
	\begin{equation*}
		F(x) = \int_0^x \frac{dx}{h'(x)}, \quad x\in\rset.
	\end{equation*}
	
\end{rem}

\begin{df}
	By a solution to~\eqref{eq:MRSDE-scalarCons}, we mean a couple of adapted and continuous processes $(X,K)$, $K$ being {\yin deterministic and} non decreasing with $K_0=0$, such that \eqref{eq:MRSDE-scalarCons} holds.
\end{df}

\begin{thm}\label{en:euss}
	Under assumptions \A{Hwu}, the SDE~\eqref{eq:MRSDE-scalarCons} has a unique square integrable solution with $K$ deterministic.
\end{thm}

\begin{proof}
	Let us start by uniqueness. Suppose that $\left(X,K^X\right)$ and $\left(Y,K^Y\right)$ are two solutions of \eqref{eq:MRSDE-scalarCons}. Then, we deduce from Itô's formula that
\begin{eqnarray*}
\E\left[|X_t - Y_t|^2\right] &=& 2\E\left[ \int_0^t (X_s-Y_s) \cdot (b(s,X_s)-b(s,Y_s)) ds \right] +  \E \left[\int_0^t  |\sigma(s,X_s)-\sigma(s,Y_s)|^2 ds\right] \\
&& + 2\E \left[\int_0^t  (X_s-Y_s) \cdot \left(\nabla h(X_s) dK_s^X - \nabla h(Y_s) d K_s^Y\right)\right].
\end{eqnarray*} 
Since, $h$ is concave, for all $x\in\rset^n$, $y\in\rset^n$, $h(y) \leq h(x) + \nabla h(x)\cdot (y-x)$. This last inequality rewrites 
\begin{equation}\label{eq:pente}
	\forall x\in\rset^n,\quad \forall y\in\rset^n, \qquad  (x-y)\cdot \nabla h(x) \leq h(x)-h(y).
\end{equation}
Thus, using the Skorokhod condition together with the constraint, we have 
\begin{eqnarray*}
\int_0^t  \E \left[(X_s-Y_s) \cdot \nabla h(X_s)\right] d K_s^X  &\leq &  \int_0^t  \left(\e\left[h(X_s)\right]-\e\left[h(Y_s)\right] \right) d K_s^X \leq 0,
\end{eqnarray*}
and, similarly,
\begin{equation*}
	-\E \left[\int_0^t  (X_s-Y_s) \cdot  \nabla h(Y_s) d K_s^Y\right] \leq   0.
\end{equation*}
Using the fact that $b$ and $\sigma$ are Lipschitz, we get, for all $t\leq T${\yin ,}
\begin{equation*}
	\E\left[|X_t - Y_t|^2\right] \leq \left(2L_T+L_T^2\right) \int_0^t \E\left[|X_s-Y_s|^2\right] ds,
\end{equation*}
and Gronwall's lemma implies that $X$ and $Y$ are equal on $[0,T]$ for each $T>0$. 

Since $X=Y$, we get from Itô's formula, {\yin that }for any $t${\yin ,}
\begin{align*}
	A^X_t & := \int_0^t \e\left[ |\nabla h(X_u)|^2\right] \, dK^X_u   \\
	 & = \e\left[h(X_t)\right] - \e\left[h(X_0)\right]  - \int_0^t \e\left[\nabla h(X_u) b(u,X_u)\right] du - \frac{1}{2}\,\int_0^t \e\left[\text{Tr}\left(\nabla^2h(X_u)\sigma\sigma^*(u,X_u)\right)\right] du \\
	 & =  \int_0^t \e\left[ |\nabla h(Y_u)|^2\right] \, dK^Y_u = \int_0^t \e\left[ |\nabla h(X_u)|^2\right] \, dK^Y_u =:A^Y_t.
\end{align*}
Thus the nondecreasing continuous functions $A^X$ and $A^Y$ are equal. For any step function $\varphi$ on $[0,T]$,
\begin{multline}
	\label{eq:uKs}
	\int_0^T \varphi(u)\, dA^X_u = \int_0^T \varphi(u) \e\left[ |\nabla h(X_u)|^2\right] \, dK^X_u 
	= \int_0^T \varphi(u) \e\left[ |\nabla h(X_u)|^2\right] \, dK^Y_u = \int_0^T \varphi(u)\, dA^Y_u.
\end{multline}
The previous equality hold true also for any continuous function on $[0,T]$ as uniform limit of step functions. The map $u\longmapsto \e\left[ |\nabla h(X_u)|^2\right]$ is continuous and does not vanish since $|\nabla h(x)|^2>0$ for all $x\in\rset^n$. As a byproduct,
\begin{equation*}
	u\longmapsto \varphi(u) = \e\left[ |\nabla h(X_u)|^2\right]^{-1}
\end{equation*}
is continuous and we get, plugging this $\varphi$ in \eqref{eq:uKs}, $K^X_T = K^Y_T$.

\smallskip

Let us fix $T>0$ and let us now construct a solution to~\eqref{eq:MRSDE-scalarCons} on $[0,T]$. For this, let us consider, for $k\geq 1$, $X^k$ solution to the following McKean-Vlasov SDE:
\begin{equation}\label{eq:pens}
	X^k_t = X_0+\int_0^t b\left(s,X^k_s\right) ds + \int_0^t \sigma\left(s,X^k_s\right) dB_s + \int_0^t \nabla h \left(X^k_s\right)\, \psi_k\left(\e\left[h\left(X^k_s\right)\right]\right) ds,\quad t\geq 0,
\end{equation}
where the function $\psi_k$ is defined as
\begin{equation*}
	\psi_k(x)= r \text{ if } x\leq -1/k, \quad \psi_k(x) = -krx, \text{ if } -1/k\leq x\leq 0, \quad \psi_k(x) = 0, \text{ if } x \geq 0.
\end{equation*}
The function $\psi_k$ depends on the constant $r>0$ which will be chosen later.
 
Existence and uniqueness of solutions to~\eqref{eq:pens} under the assumptions \A{HWu} follows from {\yin straightforward} generalizations of results from \cite{BPS91}.
We set
\begin{equation*}
	K_{\yin t}^k = \int_0^t \psi_k\left(\e\left[h\left(X^k_s\right)\right]\right)  ds,
\end{equation*}
and rewrite \eqref{eq:pens} as
\begin{equation*}
	X^k_t = X_0+\int_0^t b\left(s,X^k_s\right) ds + \int_0^t \sigma\left(s,X^k_s\right) dB_s + \int_0^t \nabla h \left(X^k_s\right) dK^k_s,\quad t\geq 0.
\end{equation*}

Let us start with some $\lp^2$ bounds. Let us compute $|X^k_t-X_0|^2$ with the help of Itô's formula. We have, for $0\leq t\leq T$,
\begin{align*}
 \left|X_t^k - X_0\right|^2 &= 2 \int_0^t \left(X_s^k-X_0\right) \cdot b\left(s,X_s^k\right) ds +  \int_0^t  \left|\sigma\left(s,X_s^k\right)\right|^2 ds \\
& \quad + 2 \int_0^t \left(X_s^k-X_0\right) \cdot \sigma\left(s,X_s^k\right) dB_s + 2 \int_0^t  \left(X_s^k-X_0\right) \cdot \nabla h \left(X^k_s\right) d K_s^k.
\end{align*}
Since $b$ and $\sigma$ are Lipschitz in space, uniformly in time, we deduce from {\yin the above} equality that, for 
\begin{align}
	 \left|X_t^k - X_0\right|^2 & \leq \left(2L_T+2L_T^2\right) \int_0^t \left|X_s^k-X_0\right|^2 ds + 2\, \int_0^t\left( \left|b(s,X_0)\right|^2 + \left|\sigma(s,X_0)\right|^2\right) ds \notag  \\
	 \label{eq:forbdgs}
	& \quad + 2 \int_0^t \left(X_s^k-X_0\right) \cdot \sigma\left(s,X_s^k\right) dB_s + 2 \int_0^t  \left(X_s^k-X_0\right) \cdot \nabla h \left(X^k_s\right) dK_s^k .
\end{align}
Since $h$ is concave and $K^k$ is deterministic, we obtain, using~\eqref{eq:pente},
\begin{align*}
	\e\left[\int_0^t  \left(X_s^k-X_0\right) \cdot \nabla h \left(X^k_s\right) dK_s^k \right] & \leq \int_0^t \left( \e\left[h\left(X^k_s\right)\right] - \e\left[h(X_0)\right]\right) dK_s^k \\
	& = \int_0^t \left( \e\left[h\left(X^k_s\right)\right] - \e\left[h(X_0)\right]\right) \psi_k\left(\e\left[h\left(X^k_s\right)\right]\right) ds \leq 0,
\end{align*}
the last inequality coming from the fact that $x\psi_k(x) \leq 0$ and $\e\left[h(X_0)\right]\geq 0$. It follows from Gronwall's inequality, that, 
\begin{equation}\label{eq:formins}
	\sup_{k\geq 1}\sup_{t\leq T} \e\left[\left|X^k_t\right|^2\right] \leq C\left(T,L_T,\e\left[|X_0|^2\right]\right)=: \gamma.
\end{equation}
Coming back to~\eqref{eq:forbdgs}, since $\psi_k$ is bounded by $r$ and $\nabla h$ has a linear growth, BDG{\yin 's} inequality gives
\begin{equation*}
	 \sup_{k\geq 1}\e\left[\sup_{t\leq T}|X^k_t|^2\right] \leq C\left(T,L_T,\e\left[|X_0|^2\right],|\nabla^2 h|_\infty,r\right).
\end{equation*}

Let us apply Itô's formula to compute the expected value of $h\left(X^k_t\right)$. We get, for $0\leq s\leq t\leq T$,
\begin{align*}
	\e\left[h\left(X^k_t\right)\right] & = \e\left[h\left(X^k_s\right)\right] + \int_s^t \e\left[ \left|\nabla h\left(X^k_u\right)\right|^2\right] \psi_k\left(\e\left[h\left(X^k_u\right)\right]\right) du \\
	&\quad + \int_s^t \e\left[\nabla h\left(X^k_u\right)\cdot b\left(u,X^k_u\right)\right] du + \frac{1}{2}\,\int_s^t \e\left[\text{Tr}\left(\nabla^2 h\left(X^k_u\right)\sigma\sigma^*\left(u,X^k_u\right)\right)\right] du.
\end{align*}
Since $\nabla^2 h$ is bounded, $b$ and $\sigma$ are Lipschitz, we deduce, taking into account~\eqref{eq:formins} and assumption \A{Hc}-(ii) that, for some constant $C\geq 0$ independent of $k$ and $r$,
\begin{equation*}
	\e\left[h\left(X^k_t\right)\right]  \geq \e\left[h\left(X^k_s\right)\right] + \int_s^t \e\left[ \left|\nabla h\left(X^k_u\right)\right|^2\right] \psi_k\left(\e\left[h\left(X^k_u\right)\right]\right) du -C (t-s).
\end{equation*}
On the other hand, for any $k\geq 1$, $0\leq u \leq T$ and $a>0$, we have using Markov{\yin 's} inequality,
\begin{align*}
	\e\left[ \left|\nabla h\left(X^k_u\right)\right|^2 \right] & \geq \e\left[ \left|\nabla h\left(X^k_u\right)\right|^2 \ind_{|X^k_u| \leq a}\right] \geq \inf\left\{ |\nabla h|^2(x) : |x|\leq a\right\} \, \p\left(\left|X^k_u\right| \leq a\right) ,\\
	& \geq \inf\left\{ |\nabla h|^2(x) : |x|\leq a\right\} \left( 1 - a^{-2}\e\left[\left|X^k_u\right|^2\right]\right){\yin .}
\end{align*}
Having in mind the bound~\eqref{eq:formins}, we choose {\yin $a=\sqrt{2\gamma}$}, to get, since $|\nabla h|^2$ is continuous, for any $k\geq 1$, $0\leq u \leq T$,
\begin{align*}
	\e\left[ \left|\nabla h\left(X^k_u\right)\right|^2 \right] & \geq {\yin \min\left\{ |\nabla h|^2(x) : |x|\leq \sqrt{2\gamma} \right\}/2} = m^2 >0.
\end{align*}
It follows that, for any $k\geq 1$, $0\leq s\leq t \leq T$,
\begin{equation}\label{eq:givemins}
	\e\left[h\left(X^k_t\right)\right]  \geq \e\left[h\left(X^k_s\right)\right]  + m^2\, \int_s^t \psi_k\left(\e\left[h\left(X^k_u\right)\right]\right) du -C (t-s).
\end{equation}
We choose $r$ such that $rm^2 > C$. As a byproduct, we have 
\begin{equation}\label{eq:rigoles}
	\forall k\geq 1,\quad \forall 0\leq t\leq T,\qquad \e\left[h\left(X^k_t\right)\right] \geq -\frac{1}{k}.
\end{equation}
Indeed, if $\e\left[h\left(X^k_t\right)\right] < -1/k$ for some $k\geq 1$ and $t>0$, let $s=\sup\{u \leq t : \e\left[h\left(X^k_u\right)\right]\geq -1/k\}$. Since $\e\left[h(X_0)\right]\geq 0$, we have $0<s<t$ and $\e\left[h\left(X^k_u\right)\right]\leq -1/k$ for $s\leq u \leq t$. Thus, by definition of $\psi_k$, $\psi_k\left(\e\left[h\left(X^k_u\right)\right]\right)\geq r$ for $u\in [s,t]$ and it follows from \eqref{eq:givemins} that
\begin{equation*}
	\e\left[h\left(X^k_t\right)\right] \geq \e\left[h\left(X^k_s\right)\right] + (m^2r-C) (t-s) > -\frac{1}{k}.
\end{equation*}
This is a contradiction.

Let $k,m\in \mathbb{N}^*$, by Itô's formula we have, for $t\geq 0$,
\begin{align*}
 \left|X_t^k - X_t^m\right|^2 &= 2 \int_0^t \left(X_s^k-X_s^m\right) \cdot \left(b\left(s,X_s^k\right)-b\left(s,X_s^m\right)\right) \d s +  \int_0^t  \left|\sigma\left(s,X_s^k\right)-\sigma\left(s,X_s^m\right)\right|^2 ds \notag\\
& \quad + 2 \int_0^t \left(X_s^k-X_s^m\right) \cdot \left( \sigma\left(s,X_s^k\right)-\sigma\left(s,X_s^m\right) \right) d B_s\\
& \quad + 2 \int_0^t  \left(X_s^k-X_s^m\right) \cdot \left(\nabla h\left(X_s^k\right) d K_s^k - \nabla h \left(X_s^m\right) d K_s^m\right).
\end{align*}
Using the fact that $b$ and $\sigma$ are Lipschitz continuous, we get, for $0\leq t\leq T$,
\begin{align}
	 \left|X_t^k - X_t^m\right|^2 & \leq \left(2L_T+L_T^2\right) \int_0^t \left|X_s^k-X_s^m\right|^2 ds +  2\, \int_0^t \left(X_s^k-X_s^m\right) \cdot \left(\sigma\left(s,X_s^k\right)-\sigma\left(s,X_s^m\right) \right) dB_s \notag\\
	 \label{eq:interpens}
	& \quad + 2 \int_0^t  \left(X_s^k-X_s^m\right) \cdot \left(\nabla h\left(X_s^k\right) d K_s^k - \nabla h \left(X_s^m\right) d K_s^m\right).
\end{align}

Arguing as in the proof of uniqueness, taking the expectation, we have, since $h$ is concave, 
\begin{multline*}
\E\left[\int_0^t   \left(X_s^k-X_s^m\right) \cdot \left(\nabla h\left(X_s^k\right) d K_s^k - \nabla h \left(X_s^m\right) d K_s^m\right)\right] \\
\leq   \e \left[\int_0^t  \left(h\left(X_s^k\right)-h\left(X_s^m\right)\right) \left(\d K_s^k-\d K_s^m\right)\right]\\
=   \int_0^t  \left( \e \left[h\left(X_s^k\right)\right] - \e\left[h\left(X_s^m\right)\right)\right] \left(\psi_k\left(\e\left[h\left(X^k_s\right)\right]\right)-\psi_m\left(\e\left[h\left(X^m_s\right)\right]\right)\right) ds{\yin ,}
\end{multline*}
and from \eqref{eq:rigoles}, since, for any $k\geq 1$, $x\psi_k(x)\leq 0$, $\e\left[h\left(X^k_s\right)\right]\geq -1/k$ and $0\leq \psi_k\leq r$, we obtain that 
\begin{equation*}
	\E\left[\int_0^t   \left(X_s^k-X_s^m\right) \cdot \left(\nabla h\left(X_s^k\right) d K_s^k - \nabla h \left(X_s^m\right) d K_s^m\right)\right] \leq \frac{r}{k}+\frac{r}{m}.
\end{equation*}
We conclude using Gronwall's lemma that there exists $C\geq 0$ depending on $T$ and $L_T$ such that for any $k,m \in \mathbb{N}^*$ we have:
\begin{equation}\label{eq:forstestis}
\sup_{0\leq t\leq T} \E \left[|X_t^k - X_t^m|^2\right] \leq C\,r \left( \frac{1}{k} + \frac{1}{m}\right).
\end{equation}

Coming back to \eqref{eq:interpens}, since $\sup_{k\geq 1}|\psi_k|_\infty \leq r$ and $\nabla h$ has at most a linear growth, taking into account \eqref{eq:formins}, we derive, from BDG and Cauchy-Schwarz inequalities, the estimate
\begin{equation*}
	 \E\left[\sup_{0\leq t\leq T}|X_t^k - X_t^m|^2\right] \leq C \left( \frac{1}{\sqrt k} + \frac{1}{\sqrt m}\right),
\end{equation*}
where the constant $C$ does not depend on $k$ and $m$. 

Thus, $\left(X^k\right)_{k\geq 1}$ is a Cauchy sequence in $\ys$. Let us denote by $X$ its limit. Finally, since $\psi_k$ is bounded by $r$ for all $k$, $K^k$ is Lipschitz with $| K^k |_{\text{Lip}}\leq r$. By {\yin Arzel\`a-Ascoli's} theorem, up to a subsequence, $\left(K^k\right)_{k\geq 1}$ converges, uniformly on $[0,T]$, towards a non decreasing, Lipschitz continuous function $K$.

It is {\yin straightforward} to check that $(X,K)$ solves~\eqref{eq:MRSDE-scalarCons}. Indeed, since $\nabla h$ has a linear growth, we have
\begin{align*}
	\sup_{0\leq t \leq T} \e\left[\left|h\left(X^k_t\right) -h(X_t)\right|\right] 
	& \leq C \, \sup_{0\leq t \leq T} \e\left[\left|X^k_t - X_t\right|\, \left(1 + \left|X^k_t\right| + |X_t| \right) \right] \\
	& \leq C \, {\yin \sup_{0\leq t \leq T}}\e^{1/2} \left[\left|X^k_t - X_t\right|^2\right] \left(1 + \sup_{k\geq 1} \e^{1/2}\left[ \left| X^k_t \right|^2\right]\right)
\end{align*}
and thus, for $0\leq t\leq T$,
\begin{itemize}
	\item $\e\left[h(X_t)\right] = \lim_{k\to\infty} \e\left[h\left(X_t^k\right)\right]\geq 0$ by \eqref{eq:rigoles}, 
	\item the Skorokhod condition is also satisfied : since $x\psi_k(x)\leq 0$,
	\begin{equation*}
		0\leq \int_0^t \e\left[h(X_s)\right] dK_s = \lim_{k\to\infty} \int_0^t \e\left[h\left(X_s^k\right)\right] dK^k_s = \lim_{k\to\infty} \int_0^t \e\left[h\left(X_s^k\right)\right] \psi_k\left(\e\left[h\left(X_t^k\right)\right]\right) ds \leq 0.
	\end{equation*}
\end{itemize}

The proof of this result is complete.
\end{proof}

\subsection{Existence and uniqueness of the solution}
\label{subsec:Existence_and_uniqueness_of_the_solution}

In this section, we are concerned with the existence and the uniqueness of a solution to~\eqref{eq:mains} when $H(\mu)$ is not necessarily of the form $\int hd\mu$ and is not necessarily concave in $\mu$.  However, in this case,  we have to assume that the volatility $\sigma$ is bounded. More precisely we {\color{black} assume} that the following conditions \A{H} hold:
\begin{description}
	\item[(Hc)] The functions $b:\Omega\times\rset_+\times\rset^n \fl \rset^n$ and $\sigma : \Omega\times \rset_+\times\rset^{n} \fl \rset^{n\times d}$ are measurable with respect to $\m E \otimes \m B\left(\rset^n\right)$ {\ph 
	and, for all $T>0$, there exists $L_T$ such that, $\p$-a.s.,
	\begin{enumerate}
		\item[(i)] For each $t\in[0,T]$,
		\begin{equation*}
			\forall x\in\rset^n, \forall y\in\rset^n,\qquad |b(t,x)-b(t,y)| + |\sigma(t,x)-\sigma(t,y)| \leq L_T\, |x-y| ;
		\end{equation*} 
		\item[(ii)] For each $t\in[0,T]$,
		\begin{equation*}
			\e\left[|b(t,0)|^2\right]  + \sup_{x\in \R^n} |\sigma(t,x)|  \leq L_T.
		\end{equation*}
	\end{enumerate}
	}
	\item[(H0)] The initial condition is $\m F_0$--measurable, square integrable and $H([X_0])\geq 0$ ;
	\item[(HH)] The function $H: \m P^2(\rset^n)\fl \rset$ is partially $\m C^2$ and 
	\begin{itemize}
		\item[(i)] there exists $M>0$ such that 
		\begin{align}
			\forall \mu \in \m P^2(\rset^n),\qquad \int_{\rset^n} |\ld H(\mu)|^2(x)\, \mu(dx)\leq M^2,  \label{eq:bilip0} 
		\end{align}
		\item[(ii)] there exist $\beta>0$ and $\eta>0$ such that
		\begin{align}
			\forall \mu \in \m P^2(\rset^n)\; \mbox{\rm with} \; -\eta\leq H(\mu)\leq 0, \; \qquad \int_{\rset^n} |\ld H(\mu)|^2(x)\, \mu(dx)\geq  \beta^2,  \label{eq:bilip} 
		\end{align}
		\item[(iii)] there exists $C\geq 0$ such that
		\begin{equation}\label{eq:DFL}
			\forall X,Y \in L^2(\Omega), \qquad \e\left[|\ld H([X])(X) -\ld H([Y])(Y)|^2\right] \leq C\, \e\left[|X-Y|^2\right]{\yin .}
		\end{equation}
	\end{itemize}
\end{description}

{\color{black} 
\noindent {\bf Examples.} We now illustrate our assumptions through {\color{black} two} examples. In the first one we consider the case where $H$ depends on the first order moment of the measure: Let $f_1:\R^n\to \R$ be of class $C^2$ with bounded first and second derivatives, 
and let 
$$
H_1(\mu) = f_1\left( \int_{\R^n} y\mu(dy)\right){\yin ,}\qquad \forall \mu\in {\color{black} \m P^2(\rset^n)}.
$$
We also assume that the set $\{f_1=0\}$ is compact and that $0$ is a noncritical value of $f_1$: 
$\nabla f_1(x)\neq 0$ if $f_1(x)=0$. Note that this implies the existence of a constant $C>0$ such that $C^{-1}\leq |\nabla f_1(x)|\leq C$  for any $x\in \R^n$ with $f_{{\pe 1}}(x)=0$. As 
$$
\ld H_1(\mu)(x)= \nabla f_1\left( \int_{\R^n} y\mu(dy)\right)
$$
is independent of $x$, it is clear that \eqref{eq:bilip0},  \eqref{eq:bilip} and \eqref{eq:DFL} hold. \\

In the second example, we assume that $H_2$ depends on the second order moments of the measure. Let  $S^n$ be the set of $n\times n$ symmetric matrices, endowed with the usual euclidean distance: $|A|= ({\rm Tr}(A^2))^{\frac12}$, and let $S^n_+$ be the subset of nonnegative matrices. 
Let $f_2:S^n\to \R$  be of class $C^2$. We assume that the set $\{f_2\geq -\eta_0\}\cap S^n_+$ is compact for some $\eta_0>0$, that $0$ is a noncritical value of $f_2$ and that $f_2(0)\neq 0$. Then there exists $C>0$ with $C^{-1}\leq |\nabla f_2|\leq C$ on $\{|f_2|\leq \eta\}{\pe \cap S^n_+}$, for $\eta\in (0,\eta_0)$ small enough, where $\nabla f_2$ is the gradient of $f_2$ in $S^n$. Moreover, as $f_2(0)\neq 0$ and $\{f_2\geq -\eta\}\cap S^n_+$ is compact, reducing $\eta>0$ if necessary, there exists $\delta>0$ such that  $\displaystyle {\rm Tr}(A)\geq \delta$ if $|f_2(A)|\leq \eta$ and $A\in {\pe S^n_+}$. Finally, we can assume, without changing the level-set $\{f_2=0\}\cap S^n_+$, that  the restriction of $f_2$ to $S^n_+$ is bounded and that its derivatives have compact support. Let
$$
H_2(\mu)= f_2\left( \int_{\R^n} xx^* \mu(dx)\right) \qquad \forall \mu\in \mathcal P^2(\rset^n).
$$
Then 
$$
\ld H_2(\mu)(x)= \nabla f_2\left( \int_{\R^n} xx^* \mu(dx)\right) x .
$$
Hence 
$$
\int_{\R^n} |\ld H_2(\mu)(x)|^2 \mu(dx) = \left| \nabla f_2\left( \int_{\R^n} xx^* \mu(dx)\right)\right|^2 \int_{\R^n}|x|^2 \mu(dx). 
$$
Let us note that condition \eqref{eq:bilip0} holds because, for any $\mu\in{\mathcal P}^2(\R^n)$, $\int_{\R^n} xx^* \mu(dx)$ belongs to $S^n_+$ and because the restriction $\nabla f_2$ to $S^n_+$ has a compact support. We now check the bound \eqref{eq:bilip}. Let $\mu$ be such that $-\eta\leq H_2(\mu)\leq 0$ and let us set $\displaystyle A:=  \int_{\R^n} xx^* \mu(dx)$. Then $A\in S^n_+$ and, as $|f_2(A)|\leq \eta$, we have ${\rm Tr}(A)\geq \delta$ and $|\nabla f_2(A)|\geq C^{-1}$. So 
$$
\int_{\R^n} |\ld H_2(\mu)(x)|^2 \mu(dx)= \left| \nabla f_2\left(A\right)\right|^2{\rm Tr}(A) \geq C^{-2}\delta, 
$$
which proves   \eqref{eq:bilip}. Finally, we check  \eqref{eq:DFL}.  Let us recall that the restrictions to $S^n_+$ of $\nabla f_2$ and $\nabla^2f_2$ have a compact support (say in a ball of radius $R$).  Let $C_0$ be a bound on $\nabla f_2$ and $\nabla^2f_2$ on $S^n_+$. Let also $X, Y\in L^2_{\p}$, $\mu=[X]$, $\nu=[Y]$ and $A:= \int_{\R^n} xx^*\mu(dx)$, $B:=  \int_{\R^n} xx^*\nu(dx)$. We assume that $|A|, |B|\leq R$, since otherwise the result is obvious. Then 
\begin{align*}
\e\left[|\ld H_2([X])(X) -\ld H_2([Y])(Y)|^2\right] & = \e\left[|\nabla f_2(A)X -\nabla f_2(B)Y|^2\right] \\ 
& \leq C_0 ( |A-B|^2 \E[|X^2|] + \e\left[|X -Y|^2\right]) \;  \leq \; C \e\left[|X -Y|^2\right],  
\end{align*}
where the constant $C$ depends also on $R$. This shows \eqref{eq:bilip}. \\
}

{\color{black} {\bf Comments on the assumptions.}} Let us recall that, as explained in the  introduction, {\yin A}ssumption \eqref{eq:bilip0} implies that $H$ is globally Lipschitz continuous in $ \m P^2(\rset^n)$.  Since \eqref{eq:DFL} holds, $H$ is semiconcave and semiconvex, in the sense,  if $\mu=[X]$ and $\nu=[Y]$, then
		\begin{equation}\label{eq:LSC}
			\left|H(\nu)-H(\mu) - \e\left[\ld H([X])(X)\cdot(Y-X)\right]\right| \leq C \e\left[|X-Y|^2\right].
		\end{equation}
Indeed, we have
	\begin{align*}		
	& \left| H(\nu)-H(\mu) - \e\left[\ld H([X])(X)\cdot(Y-X)\right] \right| \\
	& \qquad = \left| \int_0^1  \e\left[\{\ld H\left([(1-t)X+tY]\right)((1-t)X+tY)-\ld H([X])(X)\}\cdot(Y-X)\right]\, dt \right| \\
	& \qquad \leq \int_0^1  \e^{1/2}\left[|\ld H\left([(1-t)X+tY]\right)((1-t)X+tY)-\ld H([X])(X)|^2\right] \e^{1/2}\left[|Y-X|^2\right]^{1/2}\, dt\\
	& \qquad \leq C \e\left[|Y-X|^2\right],
	\end{align*}		
where we used \eqref{eq:DFL} in the last inequality. Moreover, under \eqref{eq:DFL}, for each $\mu\in\m P^2(\rset^n)$, there exists a Lipschitz {\color{black} continuous} version of
\begin{equation*}
	\ld H(\mu)(.) : \rset^n \fl \rset^n,
\end{equation*}
with a Lipschitz constant independent of $\mu$ and such that $(\mu,x)\longmapsto \ld H(\mu)(x)$ is measurable and continuous at each point $(\mu,x)$ such that $x\in \text{Supp}(\mu)$ (Corollary 5.38 in \cite{CD17-1}). Moreover,
\begin{equation*}
	\e\left[|\ld H(\mu)(X)-\ld H(\nu)(X)|^2\right] \leq C\, \e\left[|X-Y|^2\right].
\end{equation*}
As a byproduct (inequality~(5.49) in \cite{CD17-1}), for $[X]=\mu$
\begin{equation}\label{eq:utilechaos}
	\e\left[|\ld H(\mu)(X)-\ld H(\nu)(X)|^2\right] \leq C\, W_2^2(\mu,\nu). 
\end{equation}

{\color{black} \begin{df} By a solution to~\eqref{eq:mains}, we mean a couple of adapted and continuous processes $(X,K)$, $K$ being deterministic and non decreasing, with $K_0=0$.
\end{df}
}

\begin{thm}\label{en:eus}
	Under assumptions \A{H}, the SDE~\eqref{eq:mains} has a unique square integrable solution.
\end{thm}

\begin{rem} The result can be easily generalized to the case where $b$ and $\sigma$ depend also on the {\yin law}. In the proof we actually show that the process $K$ is locally Lipschitz continuous. 
\end{rem}

\begin{proof} Let us start by uniqueness. Suppose that $(X,K^X)$ and $(Y,K^Y)$ are two solutions of \eqref{eq:mains}. Then, we deduce from It\^{o}'s formula that
\begin{eqnarray*}
\E\left[|X_t - Y_t|^2\right] &=& 2\E\left[ \int_0^t (X_s-Y_s) \cdot (b(s,X_s)-b(s,Y_s)) \d s \right] +  \E \left[\int_0^t  |\sigma(s,X_s)-\sigma(s,Y_s)|^2 \d s\right] \\
&& + 2\E \left[\int_0^t  (X_s-Y_s) \cdot (\ld H([X_s])(X_s)\d K_s^X - \ld H([Y_s])(Y_s)\d K_s^Y)\right].
\end{eqnarray*} 
 Since $H$ satisfies \eqref{eq:LSC}, using the Skorokhod condition together with the constraint, we have 
\begin{eqnarray*}
\E \left[\int_0^t  (X_s-Y_s) \cdot \ld H([X_s])(X_s) \d K_s^X \right] \leq   \int_0^t  \left((H([X_s])-H([Y_s]))+ C\E\left[|X_s-Y_s|^2\right]\right) \d K_s^X\\
\qquad = \int_0^t  H([X_s]) \d K_s^X -\int_0^t  H([Y_s]) \d K_s^X+ \int_0^t C\E\left[|X_s-Y_s|^2\right] \d K_s^X \\
\qquad \leq  \int_0^t C\E\left[|X_s-Y_s|^2\right] \d K_s^X .
\end{eqnarray*}
Similarly,
\begin{equation*}
	-\E \left[\int_0^t  (X_s-Y_s) \cdot  \ld H([Y_s])(Y_s)\d K_s^Y\right] \leq    \int_0^t C\E\left[|X_s-Y_s|^2\right] \d K_s^Y.
\end{equation*}
Using the fact that $b$ and $\sigma$ are Lipschitz continuous, we get, for some constant $C>0$, 
\begin{equation*}
	\E\left[|X_t - Y_t|^2\right] \leq C\,  \int_0^t \E\left[|X_s-Y_s|^2\right]  (ds+ \d K_s^X+\d K_s^Y),
\end{equation*}
and {\color{black} equality $X=Y$} follows from Gronwall's lemma 
applied to the continuous maps $s\mapsto K^X_s$, $s\mapsto K^Y_s$ (see Lemma 4 in \cite{J64} or Theorem 17.1 in \cite{BS13}) and $s\to \E\left[|X_s-Y_s|^2\right]$.

{\color{black} Next we show the equality $K^X=K^Y$. }
We get from It\^{o}'s formula \eqref{I3} and for any $t\geq0$
\begin{align*}
	A^X_t := \int_0^t \e\left[ |\ld H([X_u])(X_u)|^2\right] \, dK^X_u & =
	  H([X_t]) -H([X_0])  
	- \int_0^t \e\left[\ld H([X_u])(X_u)\cdot b(u,X_u)\right] du \\
	& \quad - \frac{1}{2}\,\int_0^t \e\left[\text{Tr}\left(\partial_x\ld H([X_u])(X_u)\sigma\sigma^*(u,X_u)\right)\right] du \\
	& = A^Y_t := \int_0^t \e\left[ |\ld H([X_u])(X_u)|^2\right] \, dK^Y_u.
\end{align*}
Thus the nondecreasing continuous functions $A^X$ and $A^Y$ are equal. For any step function $\varphi$ on $[0,T]$,
\begin{multline}
	\label{eq:uK}
	\int_0^T \varphi(u)\, dA^X_u = \int_0^T \varphi(u) \e\left[ |\ld H([X_u])(X_u)|^2\right] \, dK^X_u  \\
	= \int_0^T \varphi(u) \e\left[ |\ld H([X_u])(X_u)|^2\right] \, dK^Y_u = \int_0^T \varphi(u)\, dA^Y_u.
\end{multline}
The previous equality hold true also for any continuous function on $[0,T]$ as uniform limit of step functions. Since
 $H(\left[X_u\right]) = 0$ for $\d K^X+dK^Y-$a.e. $u\in [0,T]$, we have
\begin{equation*}
	\e\left[ |\ld H([X_u])(X_u)|^2\right] \geq \beta^2 >0\qquad \mbox{\rm for $\d K^X+dK^Y-$a.e. $u\in [0,T]$.}
\end{equation*}
As a byproduct, we can extend the map
\begin{equation*}
	u\longmapsto \varphi(u) = \e\left[ |\ld H([X_u])(X_u)|^2\right]^{-1}
\end{equation*}
from the support of $\d K^X+\d K^Y$ into a continuous map on $[0,T]$ and we get, plugging this $\varphi$ in \eqref{eq:uK}, $K^X_T = K^Y_T$.
{\color{black} This completes the proof of the uniqueness of the solution: $(X,K^X)= (Y,K^Y)$. }
\smallskip

Let us now construct a solution to~\eqref{eq:mains} on a time interval $[0,T]$ (for $T>0$ arbitrary). For this, let us consider, for $k{\yin >} 1/\eta$, the solution $X^k$ to the following McKean-Vlasov SDE:\\
\begin{equation}\label{eq:pen}
	X^k_t = X_0+\int_0^t b(s,X^k_s)\, ds + \int_0^t \sigma(s,X^k_s)\,dB_s + \int_0^t \ld H([X^k_s])(X^k_s)\, \psi_k(s,H([X^k_s])) \, ds,\quad t\geq 0,
\end{equation}
where the function $\psi_k:\R_+\times \R^n\to \R_+$ is defined as
\begin{equation*}
	\psi_k(t,x)= r(t) \text{ if } x\leq -1/k, \quad \psi_k(t,x) = -kr(t)x, \text{ if } -1/k\leq x\leq 0, \quad \psi_k(t,x) = 0, \text{ if } x \geq 0
\end{equation*}
and $t\to r(t)$ is a continuous, positive and increasing map to be chosen later. 
We set
\begin{equation*}
	K_{{\yin t}}^k = \int_0^t \psi_k(s,H([X^k_s])) \, ds.
\end{equation*}

Let us start with some $\lp^2$ bounds. In order to use condition \eqref{eq:bilip}, we introduce the deterministic time 
$$
T_k=\inf\{t\geq 0, \; H([X^k_t])\leq -\eta\},
$$
with the convention that $T_k=T$ if the right-hand side is empty. Note that $T_k>0$ since $H([X_0])\geq 0$. 
Let us  compute $|X^k_t|^2$  for $t\in [0,T_k]$ with the help of It\^{o}'s formula. We have
\begin{align}
 |X_t^k|^2 &= {\yin |X_0|^2+  } 2 \int_0^t X_s^k \cdot b(s,X_s^k) \d s +  \int_0^t  {\color{black} {\rm Tr}(\sigma\sigma^*(s,X_s^k))} \d s \notag \\
& \quad + 2 \int_0^t X_s^k \cdot \sigma(s,X_s^k) \d B_s + 2 \int_0^t  X_s^k\cdot \ld H([X_s^k])(X_s^k)\d K_s^k. \label{eq:lzeivbnld}
\end{align}
Since $b$ and $\sigma$ are Lipschitz continuous in space, uniformly in time, we deduce from this equality that there exists a constant $C=C(L_T)$ such that 
\begin{align}
	 |X_t^k|^2 & \leq {\yin |X_0|^2+ }C\, \int_0^t ({\ph 1 +}{\yin |b(s,0)|^2}+ |X_s^k|^2) ds
	 + 2 \int_0^t X_s^k\cdot \sigma(s,X_s^k) \d B_s + 2 \int_0^t  X_s^k \cdot \ld H([X_s^k])(X_s^k)\d K_s^k  .
		 \label{eq:forbdg}
\end{align}
Using Cauchy-Schwarz{\yin 's} inequality, {\yin A}ssumption \eqref{eq:bilip0} and then Young's inequality, we obtain,
\begin{align*}
	&\e\left[\int_0^t  X_s^k \cdot \ld H([X_s^k])(X_s^k)\d K_s^k\right] 	\\
	&\qquad \leq  \int_0^t \left(\e\left[|X_s^k|^2\right]\right)^{1/2}\left(\e\left[ \left|\ld H([X^k_s](X^k_s)\right|^2\right]\right)^{1/2} \psi_k(s,H([X^k_s])\, ds \\
	&\qquad \leq \int_0^t \left(\e\left[|X_s^k|^2\right]\right)^{1/2}M r(s)\, ds \;  \le   \int_0^t \left(\e\left[|X_s^k|^2\right]+M^2r^2(s)\right) \, ds.
\end{align*}
Hence 
\begin{align*}
 \e\left[ |X_t^k|^2\right] & \leq {\yin \E[|X_0|^2]+ } C\, \int_0^t ({\ph 1}+M^2r^2(s)+ \e\left[|X_s^k|^2\right]) \d s.
\end{align*}
 It follows from Gronwall's inequality that there exists $C_0>0$, and, for any $t\in [0,T]$,  $C_1(T)>0$, depending on $T$, $\M_2([X_0])$, $L_T$ {\yin and $\sup_{t \leq T}\E[|b(t,0)|^2]$} only, such that
\begin{equation}\label{eq:formin}
	\sup_{s\leq t} \e\left[|X^k_s|^2\right] \leq C_1(T)+  M^2 \int_0^t e^{C_0(t-s)} r^2(s)ds.
\end{equation}
Coming back to~\eqref{eq:forbdg}, since $\psi_k$ is bounded by $r$, the upper bound of {\yin A}ssumption~\eqref{eq:bilip0} together with BDG's inequality give
\begin{equation*}
	 \e\left[\sup_{t\leq T}|X^k_t|^2\right] \leq C\left(T,L_T,{\pee \M_2([X_0])},{\yin \sup_{t \leq T}\E[|b(t,0)|^2]},M,r(\cdot)\right).
\end{equation*}

Let us apply It\^{o}'s formula \eqref{I1} to compute $H([X^k_t])$. We get, for $0\leq s\leq t$, such that $-\eta\leq H([X^k_u])\leq -1/k$ on $[s,t]$, 
\begin{align}
	H([X^k_t]) & = H([X^k_s]) + \int_s^t \e\left[ |\ld H([X^k_u])(X^k_u)|^2\right] \psi_k({\yin u},[X^k_u])\, du \notag \\
	&\quad + \int_s^t \e\left[\ld H([X^k_u])(X^k_u)\cdot  b({\yin u},X^k_u)\right] du + \frac{1}{2}\,\int_s^t \e\left[\text{Tr}\left(\partial_x\ld H([X^k_u])(X^k_u)\sigma\sigma^*(u,X^k_u)\right)\right] du \notag\\
	&\geq H([X^k_s]) +\beta^2 \int_s^t r(u)du
	-C\,\int_s^t \left({\ph 1+} \left(\e\left[\left|X^k_u\right|^2\right]\right)^{1/2}\right) du,  \label{eq:ineqsigmabd}
\end{align}
for some constant $C\geq 0$ independent of $k$ and $r$. In the above inequality, we used the fact that $x\longmapsto \ld H(\mu)(x)$ is Lipschitz uniformly in $\mu$, {\color{black} that $\sigma$ is bounded,} that $\beta^2 \leq \e\left[|\ld H([X])(X)|^2\right]\leq M^2$ for any $X$ such that $-\eta\leq H([X])\leq 0$ and that $b$ has a linear growth while  $\sigma$ is bounded. 
We deduce from~\eqref{eq:formin} that, 
\begin{align}\label{eq:jlbaenrlzkg}
H([X^k_t]) & \geq H([X^k_s]) +\int_s^t \left(\beta^2 r(u) -C \left(1+ \left(C_1(T)+M^2 \int_0^u e^{C_0(u-v)}  r^2(v)dv\right)^{1/2}\right)\right) du  \notag \\
& \geq H([X^k_s]) +\int_s^t \left(\beta^2 r(u) -C \left(1+C_1(T)+M^2e^{C_0 u}  \int_0^u  r^2(v)dv\right)^{1/2}\right) du  .
\end{align} 
We can then choose the map $t\to r(t)$,  independent of $k$, such that $\rho(t):=\int_0^t r^2(v)dv$ satisfies the ODE
\begin{align*}
 \rho'(u) =C^2\beta^{-4} \left(1+ C_1(T)+ M^2e^{C_0 u} \rho(u)\right),\qquad \forall u\geq 0,
\end{align*}
so that 
\begin{align}\label{eq:jlbaenrlzkgBIS}
 r^2(u) =C^2\beta^{-4} \left(1+ C_1(T)+ M^2e^{C_0 u} \int_0^u   r^2(v)dv\right),\qquad \forall u\geq 0.
\end{align} 
With this choice of $r$, we claim that 
\begin{equation}\label{eq:rigole}
	\forall k\geq 1,\quad \forall t\leq T_k,\qquad H([X^k_t]) \geq -\frac{1}{k}.
\end{equation}
 Indeed, if $H([X^k_{t_0}]) < -1/k$ for some $t_0>0$, let 
\begin{equation*}
s=\sup\left\{u \leq t_0 : H([X^k_u])\geq -1/k\right\}\; {\rm and }\; t=\inf\left\{u\geq s : H([X^k_u])\leq H([X^k_{t_0}])\vee (-\eta)\right\}.
\end{equation*} 
Since $H([X_0])\geq 0$, we have $0<s<t\leq t_0$ and $-\eta\leq H([X^k_u])\leq -1/k$ for $s\leq u \leq t$. Hence, by \eqref{eq:jlbaenrlzkg} and \eqref{eq:jlbaenrlzkgBIS}, we get
\begin{equation*}
	H([X^k_{t_0}])\vee (-\eta)= H([X^k_{t}]) \geq H([X^k_s]) = -\frac{1}{k}.
\end{equation*}
This is a contradiction and \eqref{eq:rigole} holds. Note that, with this choice of $r(\cdot)$, we have $T_k=T$ and 
\begin{equation}\label{eq.boundKkt}
\sup_k K^k_t \leq C(t)\qquad \forall t\in [0,T]. 
\end{equation}

Let $k,m \in \mathbb{N}^*$, with $k,m{\yin >} 1/\eta$. By It\^{o}'s formula we have, for $t\geq 0$,
\begin{align*}
 |X_t^k - X_t^m|^2 &= 2 \int_0^t (X_s^k-X_s^m) \cdot (b(s,X_s^k)-b(s,X_s^m)) \d s \\
& \quad  +  \int_0^t {\color{black} {\rm Tr}((\sigma(s,X_s^k)-\sigma(s,X_s^m))
 (\sigma(s,X_s^k)-\sigma(s,X_s^m))^*)} \d s \notag\\
& \quad + 2 \int_0^t (X_s^k-X_s^m) \cdot (\sigma(s,X_s^k)-\sigma(s,X_s^m) ) \d B_s\\
& \quad + 2 \int_0^t  (X_s^k-X_s^m) \cdot (\ld H([X_s^k])(X_s^k)\d K_s^k - \ld H([X_s^m])(X_s^m)\d K_s^m).
\end{align*}
Using the fact that $b$ and $\sigma$ are $L_T-$Lipschitz continuous, we get
\begin{align}
	 |X_t^k - X_t^m|^2 & \leq (2L_T+L_T^2)\, \int_0^t |X_s^k-X_s^m|^2 \d s +  2 \int_0^t (X_s^k-X_s^m) \cdot (\sigma(s,X_s^k)-\sigma(s,X_s^m) ) \d B_s \notag\\
	 \label{eq:interpen}
	& \quad + 2 \int_0^t  (X_s^k-X_s^m) \cdot (\ld H([X_s^k])(X_s^k)\d K_s^k - \ld H([X_s^m])(X_s^m)\d K_s^m) .
\end{align}

Arguing as in the proof of uniqueness, taking the expectation, we have,  since $H$ satisfies \eqref{eq:LSC}, 
\begin{multline*}
\E \left[\int_0^t   (X_s^k-X_s^m) \cdot (\ld H([X_s^k])(X_s^k)\d K_s^k - \ld H([X_s^m])(X_s^m)\d K_s^m)\right] \\
\leq   \int_0^t  (H([X_s^k])-H([X_s^m])) (\d K_s^k-\d K_s^m)+C  \int_0^t  \E\left[|X_s^k-X_s^m|^2\right] (\d K_s^k+\d K_s^m)\\
=   \int_0^t  (H([X_s^k])-H([X_s^m])) (\psi_k(s,H([X_s^k])-\psi_m(s,H([X_s^m]))\d s+C  \int_0^t  \E\left[|X_s^k-X_s^m|^2\right] (\d K_s^k+\d K_s^m).
\end{multline*}
From \eqref{eq:rigole}, since $x\psi_k(t,x)\leq 0$, $H([X^k_t])\geq -1/k$ and $\psi_k({\yin s},H([X_s^k])$ and $\psi_m({\yin s},H([X_s^m])$ are in $[0, r(t)]$ for $s\in [0,t]$, we obtain that 
\begin{multline*}
	\E \left[\int_0^t   (X_s^k-X_s^m) \cdot (\ld H([X_s^k])(X_s^k)\d K_s^k - \ld H([X_s^m])(X_s^m)\d K_s^m)\right] \\
	 \leq t(\frac{r(t)}{k}+\frac{r(t)}{m})+C  \int_0^t  \E\left[|X_s^k-X_s^m|^2\right] (\d K_s^k+\d K_s^m).
\end{multline*}
Therefore 
\begin{equation*}
\E\left[|X_t^k - X_t^m|^2\right]  \leq C\, \int_0^t \E\left[|X_s^k-X_s^m|^2\right](ds+\d K_s^k+\d K_s^m) +  t(\frac{r(t)}{k}+\frac{r(t)}{m}).
\end{equation*}

We conclude using Gronwall's lemma and the bound on $(K^k)$ in \eqref{eq.boundKkt} that, for any $T>0$,  
\begin{equation*}
\sup_{t\in [0,T]}\E\left[|X_t^k - X_t^m|^2\right]  \leq T\left(\frac{r(T)}{k}+\frac{r(T)}{m}\right)\exp\{C(T+K_T^k+K_T^m)\}\leq C\, \Big( \frac{1}{k} + \frac{1}{m}\Big).
\end{equation*}
Coming back to \eqref{eq:interpen}, we derive, from {\yin \eqref{eq:forstestis}}, \eqref{eq.boundKkt}, {\yin the }BDG and Cauchy-Schwarz inequalities, the estimate
\begin{equation*}
	 \E \left[\sup_{t\leq T}|X_t^k - X_t^m|^2\right] \leq C \Big( \frac{1}{\sqrt k} + \frac{1}{\sqrt m}\Big),
\end{equation*}
where the constant $C$ does not depend on $k$ and $m$. Thus, $\{X^k\}_{k\geq 1}$ is a Cauchy sequence in $\ys$. Let us denote by $X$ its limit. Since $\psi_k$ is bounded on $[0,T]$ by $r(T)$ for all $k$, $K^k$ is Lipschitz continuous with a Lipschitz constant bounded by $r(T)$ on $[0,T]$. By Ascoli-Arzela theorem, up to a subsequence, $(K^k)$ converges locally uniformly to a non decreasing, locally Lipschitz continuous function $K$.

It is {\yin straightforward} to check that $(X,K)$ solves~\eqref{eq:mains}. Indeed, 
\begin{itemize}
	\item $H([X_t]) = \lim_{k\to\infty} H([X^k_t])\geq 0$ by \eqref{eq:rigole}, 
	\item the Skorokhod condition is also satisfied : since $x\psi_k(t,x)\leq 0$,
	\begin{equation*}
		0\leq \int_0^t H([X_s]) dK_s  = \lim_{k\to\infty} \int_0^t H([X^k_s]) \psi_k(s,H([X^k_s])) ds \leq 0.
	\end{equation*}
\end{itemize}
\end{proof}


\section{SDEs with normal reflexion in conditional law}
\label{sec:conditional_law}

In  this section, we consider the reflected SDE:
\begin{equation}\label{eq:mainsCond}
	\left\{
	\begin{split}
	& X_t  =X_0+\int_0^t b(s,X_s)\, ds + \int_0^t \sigma_0(s,X_s)\,dB_s +\int_0^t \sigma_1(s,X_s)dW_s+ \int_0^t \ld H([X_s|W])(X_s)\,dK_s,
	\\
	& H(\left[X_t|W\right]) \geq 0, \quad \int_0^t H(\left[X_s|W\right]) \, dK_s = 0, \quad t\geq 0,
	\end{split}
	\right.
\end{equation}
where now $B$ and $W$ are independent Brownian motions and $(K_s)$ is a continuous nondecreasing process adapted to the filtration $\m F^W$ associated with $W$. As before, $H$ is a map from $\m P^2 (\rset^n)$ to $\rset$. The notation $[X_t|W]$ stands for the conditional probability of $X_t$ given $W$. We assume that the initial condition of the process, $X_0$, is independent of $B$ and of $W$  and in $L^2$. As explained below, this reflected process is the limit of a reflected particle system with a common noise. 

We assume in this part that the following conditions \A{Hcl} hold:
\begin{description}
	\item[(H$\Omega$)] The probability space is $(\Omega, \p)=(\Omega^0\times \Omega^1, \p^0\otimes \p^1)$, where $\Omega^0$ supports the $X_0$ and $B$, while $\Omega^1$ supports $W$ with associated filtration $\m F^W=\m F^1$.
	\item[(Hc)] The functions $b:\Omega\times\rset_+\times\rset^n \fl \rset^n$ and {\ph $\sigma_0$, $\sigma_1 : \Omega\times \rset_+\times\rset^{n} \fl \rset^{n\times d}$} are measurable with respect to $\m E \otimes \m B\left(\rset^n\right)$ and 
	\begin{enumerate}
		\item[(i)] For all $T>0$, there exists $L_T$ such that, $\p$-a.s., for each $t\in[0,T]$,
		\begin{equation*}
			\forall x\in\rset^n, \forall y\in\rset^n,\qquad |b(t,x)-b(t,y)| + |\sigma(t,x)-\sigma(t,y)| \leq L_T\, |x-y| ;
		\end{equation*} 
		\item[(ii)] $b$, $\sigma_0$, $\sigma_1$ are globally bounded: {\ph for all $T>0$, there exists $C_T$ such that, $\p$-a.s.,
		\begin{equation*}
			\sup_{t\leq T, \ x\in \R^n} \left\{|b(t,x)| + |\sigma_0(t,x)| + |\sigma_1(t,x)| \right\} \leq C_T.
		\end{equation*}

		}
	\end{enumerate}
	\item[(H0)] The initial condition is $X_0$ is independent of $B$ and $W$, in $L^2(\Omega^0)$ and with $H([X_0])\geq 0$;
	\item[(HH)] The function $H: \m P^2(\rset^n)\fl \rset$ is globally $\m C^2$ and 
	\begin{itemize}
		\item[(i)] there exists $M>0$ such that: $\forall \mu \in \m P^2(\rset^n),$ 
		\begin{align}
			 & |H(\mu)|+ \int_{\rset^n} |\ld H(\mu)|^2(x)\, \mu(dx)+\int_{\R^n} \left|\partial_y \ld H(\mu)(y)\right|\mu(dy) \notag\\
			 & \qquad + \int_{\R^n\times \R^n} \left|D^2_{\mu\mu} H(\mu)(x,y)\right|\mu(dx)\mu(dy)
			 \leq M^2,  \label{eq:bilip0Cond} 
		\end{align}
		\item[(ii)] there exist $\beta>0$ and $\eta>0$ such that
		\begin{align}
			\forall \mu \in \m P^2(\rset^n)\; \mbox{\rm with} \; -\eta\leq H(\mu)\leq 0, \; \qquad \int_{\rset^n} |\ld H(\mu)|^2(x)\, \mu(dx)\geq  \beta^2,  \label{eq:condCond} 
		\end{align}
		\item[(iii)] there exists $C_1\geq 0$ such that
		\begin{equation}\label{eq:DFLCond}
			\forall X,Y \in L^2_{\p}, \qquad \e\left[|\ld H([X])(X) -\ld H([Y])(Y)|^2\right] \leq C_1\, \e\left[|X-Y|^2\right].
		\end{equation}
	\end{itemize}
\end{description}

\noindent {\bf Examples.} They  are  the same as in the previous section. For the first example, let $f_1:\R^n\to \R$ be of class $C^2$ and such that the set $\{f_1=0\}$ is compact and  $0$ is a noncritical value of $f_1$. We also assume, without loss of generality, that $f_1$ is bounded and that $\nabla f_1$ and $\nabla^2 f_1$ have compact support. We set 
$$
H_1(\mu) = f_1\left( \int_{\R^n} y\mu(dy)\right)\qquad \forall \mu\in \mathcal P^2(\rset^n).
$$
Then
$$
\ld H_1(\mu)(x)= \nabla f_1\left( \int_{\R^n} y\mu(dy)\right)\; {\rm and }\; D^2_{\mu\mu}H_1(\mu)(x,x') = \nabla^2 f_1\left( \int_{\R^n} y\mu(dy)\right).
$$
So conditions ${\bf (Hcl)}$ hold as in the previous section. \\

For the second example, we assume that $H_2$ is given by 
$$
H_2(\mu)= f_2\left( \int_{\R^n} xx^* \mu(dx)\right) \qquad \forall \mu\in \mathcal P^2(\rset^n),
$$
where $f_2$ is as in the previous section. Then one can show with the same argument as in the previous section that $H_2$ is bounded and that its first and second order derivatives have bounded support and that \eqref{eq:condCond} holds. This easily implies that $H_2$ satisfies conditions ${\bf (Hcl)}$.

\subsection{Existence and uniqueness of the solution}

\begin{df} We say that $(X,K)$ is a solution to~\eqref{eq:mainsCond} if $(X,K)$ is a continuous and progressively measurable process such that $X_t$ is square integrable for any $t\geq 0$ and such that $K$ is nondecreasing and adapted to the filtration ${\mathcal F}^W$ with $K_0=0$. 
\end{df}
Note that, under the above assumptions, the equation has a meaning since (denoting by $\E^0$ the conditional expectation with respect to $W$) we have, for any $t\geq 0$ and $\p^1-$a.s., 
\begin{align*}
\E^0\left[ \int_0^t \left|\ld H([X_s|W])(X_s)\right|\,dK_s\right] &  =  \int_0^t \E^0\left[ \left|\ld  H([X_s|W])(X_s)\right| \right] \,dK_s \\
& \leq \left(\sup_{\mu\in {\mathcal P}^2} \int_{\R^d} |\ld H(\mu)(x)|^2\mu(dx)\right)^{1/2}  K_t \; <\; +\infty.
\end{align*}

\begin{thm}\label{thm:Cond}
	Under assumptions \A{Hcl}, the SDE~\eqref{eq:mainsCond} has a unique square integrable solution. In addition, we have, for any $T>0$ and for any $\theta>0$, 
\begin{equation}\label{estim.Cond}
\E\left[ \sup_{0\leq s\leq T} |X_s|^2\right]  \leq C(T), \qquad \E\Bigl[ \exp\{\theta K_T\}\Bigr] \leq C_\theta(T),
\end{equation}
for some constants $C(T)$ and $C_\theta(T)$ depending on the data, $T$ and, for $C_\theta(T)$,  $\theta$.
\end{thm}

{\color{black}
\begin{rem}\label{rem:uniquelaw}
We emphasize that weak uniqueness also holds for  \eqref{eq:mains} and for  \eqref{eq:mainsCond}. We explain why {\color{black} at the end of the proof of Theorem \ref{thm:Cond}}.
\end{rem}
}

\begin{rem}\label{momentbizarres}\
In contrast with the previous section (without ``common noise'') the process $K$ is  neither deterministic nor Lipschitz continuous in general. For this reason, the assumptions under which we work here are much stronger than the ones in Section \ref{sec:sdes_with_normal_reflexion_in_law}.
\end{rem}

{\pe From the proof of the Theorem (see estimate \eqref{iukzqesbrlfdkc4} below), we also deduce the following result.
\begin{cor}\label{cor:momentdoubleesp}  
Let $p\geq2$ and $T>0$. Assume that in addition to \A{Hcl} we have $M_p([X_0]) + \H_p < +\infty$. Then, there exists a positive constant $C_p(T) := C\left((p,T,M_p([X_0]), \H_p\right)>0$ such that
\begin{equation}\label{zebound2ouf}
\E\left[ \sup_{0\leq s\leq T} \left|\E^0\left[ |X_s|^2\right]\right|^p\right]  \leq C_p(T){\yin .}
\end{equation}
\end{cor}
}

One can also show that the process depends in a continuous way of the initial condition: 
\begin{prop}\label{prop.CondInitCN} Let $X_0$ and $Y_0$ be two initial conditions with $H([X_0])\geq 0$ and $H([Y_0])\geq 0$. Let $(X, K^X)$ and $(Y, K^Y)$ be solutions of~\eqref{eq:mainsCond} with initial conditions $X_0$ and $Y_0$, respectively. Then, for any $T>0$, 
$$
\E\Bigl[\sup_{0\leq t\leq T}\E^0[|X_t-Y_t|^2]\Bigr] \leq 
C \E\left[|X_0-Y_0|^2\right]
$$
while
{\ph
$$
\E\Bigl[\sup_{0\leq t\leq T}|X_t-Y_t|^2\Bigr] \leq 
C\left(\E\left[ |X_0-Y_0|^2 \right]+\E\left[ |X_0-Y_0|^2 \right]^{1/2}\right). 
$$
}

\end{prop} 

{\pe
Eventually, we have the following regularity estimate on the continuity of the path $K$ and $X$.
\begin{prop}\label{prop.moduleconti} 
Let $(X,K)$ be the unique solution of SDE~\eqref{eq:mainsCond}. Assume in addition to {\yin \A{Hcl}} that there exists $p\geq2$ such that $\M_p([X_0]) + \H_p <+\infty$.

Then, there exists a constant $C_p(T) := C\left((p,T,\M_p([X_0]),\H_p\right)>0$ such that
\begin{equation}\label{modulecontiK}
\E\left[|K_t - K_s|^p\right] \leq C_{p,T}|t-s|^{p/4},
\end{equation}
and if moreover $p$ is in $\mathbb{N}$
\begin{equation}\label{modulecontiX}
\E\left[|X_t - X_s|^p\right] + \E\left[\sup_{u \in [s,t]}\left|\e^0\left[|X_u - X_s|^2\right]\right|^{{\yin p/2}}\right] \leq C_{p,T}|t-s|^{p/4}.
\end{equation}
\end{prop} 
}

\begin{proof}[Proof of Theorem \ref{thm:Cond}: Existence under an additional moment condition] The structure of proof is  roughly the same as for Theorem \ref{en:eus}. However, because the process $K$ is no longer Lipschitz continuous nor deterministic, we have to pay extra attention in the various estimates. Here we address the existence of the solution under the additional condition that $X_0$ has a  moment of order $4$: $\E[|X_0|^{4}]<+\infty$. This extra condition is removed after the proof of Proposition \ref{prop.CondInitCN}. The uniqueness is a straightforward consequence of Proposition \ref{prop.CondInitCN}. 

\noindent{\bf Step~1: Approximate solutions.} We build the solution on the time interval $[0,T]$ (for an arbitrary $T>0$). All the constant{\pe s} below depend on this horizon $T$. As before we argue by penalization. 
We consider, for $k\geq 1/\eta$, $X^k$ solution to the following McKean-Vlasov SDE:\\
\begin{align}
	X^k_t & = X_0+\int_0^t b(s,X^k_s)\, ds + \int_0^t \sigma_0(s,X^k_s)\,dB_s + \int_0^t \sigma_1(s,X^k_s)\,dW_s\notag \\
	& \qquad  +\int_0^t \ld H([X^k_s|W])(X^k_s)\, \psi_k(H([X^k_s|W]) \, \d s, \qquad t\geq 0,  \label{eq:penCond}
\end{align}
where the function $\psi_k$ is defined as
\begin{equation*}
	\psi_k(x)= k^2 \text{ if } x\leq -1/k, \quad \psi_k(x) = -k^3x, \text{ if } -1/k\leq x\leq 0, \quad \psi_k(x) = 0, \text{ if } x \geq 0.
\end{equation*}
We set $\mu^k_t= [X^k_t|W]$ and $\displaystyle K^k_t =\int_0^t \psi_k(H([X^k_s|W])\d s= \int_0^t \psi_k(H(\mu^k_s))\d s$ and define the $\m F^W-$stopping time $\tau^k$ by
$$
\tau^k:= T \wedge \inf\{t\geq 0, \; H(\mu^k_t)\leq -\eta\}. 
$$
We note for a later use that $\mu^k_{t\wedge \tau^k}= [X^k_{t\wedge \tau^k}|W]$ since $\tau^k$ is {\yin an $\m F^W$-stopping time}. We also stress that $H(\mu^k_s)\in [-1/k,0]$ for $dK^k-$a.e. $s\in [0,\tau^k]$, so that, by \eqref{eq:condCond} and for any nonnegative process $\varphi_s$ adapted to $\m F^W$,  
\begin{equation}\label{eq:condCondCons}
\int_0^{t\wedge \tau^k} \varphi_s \E^0\left[\left|\ld H(\mu^k_s)(X^k_s)\right|^2\right]dK^k_s \geq \beta^2 \int_0^{t\wedge \tau^k} \varphi_s \d K^k_{s\wedge \tau^k}. 
\end{equation}
 We finally note that \eqref{eq:penCond} can be rewritten as
\begin{align*}
	X^k_t & = X_0+\int_0^t b(s,X^k_s)\, ds + \int_0^t \sigma_0(s,X^k_s)\,dB_s + \int_0^t \sigma_1(s,X^k_s)\,dW_s\notag \\
	& \qquad  +\int_0^t \ld H([X^k_s|W])(X^k_s)\, \d K^k_s, \qquad t\geq 0.
\end{align*}

\noindent{\bf Step~2: Estimate on $K^k$.} For $0\leq s\leq t \leq T$ and by It\^{o}'s formula \eqref{I2}, we have 
 \begin{align}
	H(\mu^k_{t\wedge \tau^k}) & = 
			H(\mu^k_{s\wedge \tau^k})+ {\yin \int_{s\wedge \tau^k}^{t\wedge \tau^k} \E^0\left[ \ld H(\mu^k_u)(X^k_u)\cdot b(u,X^k_u)\right] \d u\notag} \\
			& \qquad +\int_{s\wedge \tau^k}^{t\wedge \tau^k} \E^0\left[ \left|\ld H(\mu^k_u)(X^k_u)\right|^2\right] \d K^k_u \notag\\
			& \qquad  +\frac12 \int_{s\wedge \tau^k}^{t\wedge \tau^k} \E^0\left[  {\rm Tr}\left(\partial_x \ld H(\mu^k_u)(X^k_u)(a_1(u,X^k_u)+a_2(u,X^k_u))\right)\right]\d u \notag\\
			& \qquad + \frac12 \int_{s\wedge \tau^k}^{t\wedge \tau^k} \E^0\tilde \E^0\left[  {\rm Tr}\left(D^2_{\mu\mu} H(\mu^k_u)(X^k_u,\tilde X^k_u)\sigma_1(u,X^k_u)\sigma_1^*(u, \tilde X^k_u\right)\right]\d u \notag\\
			& \qquad + \int_{s\wedge \tau^k}^{t\wedge \tau^k} \E^0\left[  \sigma_1^*(u,X^k_u)\ld H(\mu^k_u)(X^k_u) \right]\cdot \d W_u.\label{eq:ljhebrjshdx}
\end{align} 
 We use \eqref{eq:condCondCons}  to bound below the term $\displaystyle \int_{s\wedge \tau^k}^{t\wedge \tau^k} \E^0\left[ \left|\ld H(\mu^k_u)(X^k_u)\right|^2\right] \d K^k_u$. We also use our bounds on $H$ in \eqref{eq:bilip0Cond} and the $L^\infty$ bounds on $b$ and $\sigma_i$ to estimate the other terms and obtain: 
\begin{align}\label{eq:H-HCond}
	H(\mu^k_{t\wedge \tau^k}) & \geq  H(\mu^k_{s\wedge \tau^k})- C({t\wedge \tau^k}-{s\wedge \tau^k})+ \beta^2 (K^k_{t\wedge \tau^k}-K^k_{s\wedge \tau^k}) \notag \\
		&\qquad 	+ \int_{s\wedge \tau^k}^{t\wedge \tau^k} \E^0\left[  \sigma_1^*(u,X^k_u)\ld H(\mu^k_u)(X^k_u) \right]\cdot \d W_u.
\end{align} 
Taking $s=0$ and $t=T$ and using once again the fact that $\sigma_1$ is bounded and \eqref{eq:bilip0Cond}, we infer that the $K^k$ have uniformly bounded exponential moments: 
\begin{equation}\label{eq:expmomK}
\sup_{k} \e\left[ \exp\{\theta K^k_{\tau^k}\}\right] \leq C(\theta), \qquad \forall \theta>0. 
\end{equation}

{\pe
\noindent {\bf Step~3: Some estimates on the ($p$)-moments of $X^k_t$.} We claim that:
\begin{eqnarray}
\E\left[ \sup_{0\leq t\leq \tau^k}|X^k_t|^2\right] &\leq &C_{2,T}( \E\left[|X_0|^2\right]+1);\label{iukzqesbrlfdkc}
\end{eqnarray}
and, assuming in addition to the current assumption that there exists $p \geq 2$ such that $\M_p([X_0])+\H_p<+\infty$, it holds:
\begin{eqnarray}
\E\left[ \sup_{0\leq t\leq \tau^k}\E^0\left[|X^k_t|^p\right]\right] &\leq& C_{p,T},\label{iukzqesbrlfdkcbisbis}\\
\E\left[(\E^0)^{\ph p/2}\left[ \sup_{0\leq t\leq \tau^k}|X^k_{t}|^2\right]\right] &\leq &C_{p,T}\label{iukzqesbrlfdkc4},
\end{eqnarray}
where $C_{p,T} := C_{p,T}( \M_p([X_0])+ \H_p)>0$.
}

{\color{black}
Let us set  $\gamma^k_t:= \exp\{-\alpha (\delta t+H(\mu^k_t))\}$ for $\alpha,\delta \geq 1$ to be chosen below. We  note for later use that, because of our boundedness assumption on $H$, $\gamma^k_t$ is bounded above and below by positive constants. We have, by It\^{o}'s formula \eqref{I2}, 
\begin{eqnarray}\label{unerefdeplus}
 \gamma^k_{t\wedge \tau^k} |X^k_{t\wedge \tau^k}|^p &=& \gamma^k_0|X_0|^p + \int_0^{t\wedge \tau^k} \gamma^k_s p|X^k_s|^{p-2}X^k_s\cdot b(s,X^k_s)\d s\notag\\
&& +\frac12 \int_0^{t\wedge \tau^k} \gamma^k_s {\rm Tr} [a(s,X^k_s)(p |X^k_s|^{p-2}I_d+p(p-2)|X^k_s|^{p-4} X^k_s\otimes X^k_s) ]\, ds \notag\\
&& + \int_0^{t\wedge \tau^k} p\gamma^k_s \sigma_0^*(s,X^k_s)|X^k_s|^{p-2} X^k_s\cdot dB_s +\int_0^{t\wedge \tau^k} p\gamma^k_s \sigma_1^*(s,X^k_s)|X^k_s|^{p-2}X^k_s\cdot dW_s\notag\\
&& + \int_0^{t\wedge \tau^k} p\gamma^k_s \ld H(\mu^k_s)(X^k_s)\cdot |X^k_s|^{p-2}X^k_s\,dK^k_s-\alpha\delta  \int_0^{t\wedge \tau^k} \gamma^k_s|X^k_s|^p ds  \notag\\
&& - \alpha \int_0^{t\wedge \tau^k} \gamma^k_s |X^k_s|^p \E^0\left[ \ld H(\mu^k_s)(X^k_s)\cdot b(s,X^k_s)\right] ds 
\notag\\
&& -  \alpha \int_0^{t\wedge \tau^k} \gamma^k_s|X^k_s|^p \E^0\left[\sigma_1^*(s,X^k_s)\ld H(\mu^k_s)(X^k_s)\right]\cdot dW_s\notag\\
&& - \frac\alpha2  \int_0^{t\wedge \tau^k} \gamma^k_s|X^k_s|^p  \E^0\left[ {\rm Tr}\left(a_0(s,X^k_s) \partial_y\ld H(\mu^k_s)(X^k_s))\right)\right] ds \notag\\
&& - \frac\alpha2   \int_0^{t\wedge \tau^k} \gamma^k_s|X^k_s|^p  \E^0\tilde \E^0\left[{\rm Tr}\left(D^2_{\mu\mu}H(\mu^k_s)(X^k_s,\tilde X^k_s)\sigma_1(s,X^k_s)\sigma_1^*(s,\tilde X^k_s)\right)\right] ds\notag\\
&& - p \alpha \int_0^{t\wedge \tau^k}  \gamma^k_s{\rm Tr}\left( \sigma_1^*(s, X^k_s)|X^k_s|^{p-2}X^k_s ( \E^0\left[\sigma_1^*(s,X^k_s)\ld H(\mu^k_s)(X^k_s))^*\right) \right] \d s\notag	\\
&& +\frac{\alpha^2}{2}  \int_0^{t\wedge \tau^k} \gamma^k_s |X^k_s|^p{\rm Tr} \left(\E^0\left[ \sigma_1^*(s,X^k_s)\ld H(\mu^k_s)(X^k_s)\right]
(\E^0\left[ \sigma_1^*(s,X^k_s)\ld H(\mu^k_s)(X^k_s)\right])^*\right)\d s\notag\\
&&- \alpha \int_0^{t\wedge \tau^k} \gamma^k_s |X^k_s|^p \E^0\left[|\ld H(\mu^k_s)(X^k_s)|^2\right] dK^k_s.
\end{eqnarray}
{\color{black} From Young{\yin 's} inequality, we have
\begin{eqnarray}\label{Unyoungdeplus}
&&\E^0\left[\int_0^{t\wedge \tau^k} p\gamma^k_s \ld H(\mu^k_s)(X^k_s)\cdot |X^k_s|^{p-2}X^k_s\,dK^k_s\right]\notag \\
 &\leq &\int_0^{t\wedge \tau^k} p\gamma^k_s \left( \E^0\left[\frac{p-1}{p}|X^k_s|^{p}\right]+\E^0\left[\frac{1}{p} |\ld H(\mu^k_s)(X^k_s)|^p\right]\right)\,dK^k_s.
\end{eqnarray}
}

By \eqref{eq:bilip0Cond}, the terms $\E^0\left[|\ld H(\mu^k_s)(X^k_s)|\right]$, $\E^0\left[|\partial_y \ld H(\mu^k_s)(X^k_s)|\right]$  and $\E^0\tilde \E^0\left[\left|D^2_{\mu\mu}H(\mu^k_s)(X^k_s,\tilde X^k_s)\right|\right]$ are a.s. bounded. So, by \eqref{eq:condCondCons} applied to the last term and for $\alpha$ large enough (to absorb partially the first term in $dK^k_s$) and then $\delta$  large enough (to absorb partially all the terms in $ds$), we have 
\begin{align}
 \E^0\left[\gamma^k_{t\wedge \tau^k} |X^k_{t\wedge \tau^k}|^p\right] &\leq \E^0\left[\gamma^k_0|X_0|^p\right] +C_{\alpha,\delta}T+C\H_pK^k_T\notag \\
& \qquad +\int_0^{t\wedge \tau^k} p\gamma^k_s \E^0\left[\sigma_1^*(s,X^k_s)|X^k_s|^{p-2}X^k_s\right]\cdot dW_s\notag
 \\
	&  
\qquad 	-  \alpha \int_0^{t\wedge \tau^k} \E^0\left[\gamma^k_s|X^k_s|^p\right] \E^0\left[\sigma_1^*(s,X^k_s)\ld H(\mu^k_s)(X^k_s)\right]\cdot dW_s.
\label{izakehjzrbdnfg0}
\end{align}
Taking expectation and using the bound on $K^k$ in \eqref{eq:expmomK} we find 
\begin{align}\label{izakehjzrbdnfg}
& \sup_{t\in [0,T]} \E\left[ |X^k_{t\wedge \tau^k}|^p\right] \leq C (\E\left[|X_0|^p\right] + \H_p)+C(T).
\end{align}
We come back to \eqref{izakehjzrbdnfg0} and use {\color{black} the boundedness of $K^k$ in} \eqref{eq:expmomK} and BDG inequality to find 
\begin{align}
& \E\left[ \sup_{s\leq T}\E^0\left[ \gamma^k_{t\wedge \tau^k} |X^k_{t\wedge \tau^k}|^p\right]\right]\notag\\
& \leq C\E\left[|X_0|^p\right]  +{\color{black} {\pee C_{\alpha,\delta} (T,\H_p)}}
+ C\E\left[ \left( \int_0^T (\gamma^k_{s\wedge \tau^k})^2 (\E^0)^2\left[ |X^k_{s\wedge \tau^k}|^{p}\right]ds\right)^{1/2}\right]\notag\\
&   \leq C\E\left[|X_0|^p\right]  +{\color{black} C_{\alpha,\delta} (T,\H_p)}
+ C\E\left[\left(\sup_{s\leq T} \E^0\left[\gamma^k_{t\wedge \tau^k} |X^k_{s\wedge \tau^k}|^{p}\right] \right)^{1/2} \left( \int_0^T\gamma^k_{t\wedge \tau^k} \E^0\left[ |X^k_{s\wedge \tau^k}|^{p}\right]ds\right)^{1/2}\right]\notag\\ 
&   \leq C\E\left[|X_0|^p\right]  +{\color{black} C_{\alpha,\delta} (T, \H_p)} + \frac12 \E\left[\sup_{s\leq T}\E^0\left[ \gamma^k_{t\wedge \tau^k} |X^k_{s\wedge \tau^k}|^{p}\right]\right]
+ C\E\left[ \int_0^T \E^0\left[\gamma^k_{t\wedge \tau^k}  |X^k_{s\wedge \tau^k}|^{p}\right]ds\right].\label{encoreuneref}
\end{align}
Using \eqref{izakehjzrbdnfg}, this proves \eqref{iukzqesbrlfdkcbisbis}. \\
Now, from Itô's formula on $ |X^k_{t\wedge \tau^k}|^2$ we obtain an expression similar to \eqref{unerefdeplus} (with $p=2$ therein and $\alpha=\delta=0$) without the last eight terms. Taking first the supremum, then the expectation $\E$ and using the fact that
\begin{eqnarray*}
&&\E^1\otimes \E^0\left[\sup_{s \leq t} \int_0^{s\wedge \tau^k}   |\ld H(\mu^k_s)(X^k_s)|\ |X^k_s|\,dK^k_s\right]\\
&\leq & \E^1\left[\int_0^{T\wedge \tau^k}  (\E^0)^{1/2}\left[|\ld H(\mu^k_s)(X^k_s)|^2\right] (\E^0)^{1/2}\left[|X^k_s|^2\right]\,dK^k_s\right]\\
&\leq & {\yin C(M)(\E^1)^{1/2}\left[\sup_{0\leq t \leq T\wedge \tau^k}\left\{\E^0\left[|X^k_t|^2\right]\right\}\right](\E^1)^{1/2}\left[(K_T^k)^2\right],}
\end{eqnarray*}
thanks to Cauchy-Schwarz inequality we deduce, from  \eqref{iukzqesbrlfdkcbisbis} {\pe with $p=2$}, estimates on the moment of $K$ and BDG{\pee 's} inequality that \eqref{iukzqesbrlfdkc} holds. {\pe Also, still working with the It\^o's expansion of $ |X^k_{t\wedge \tau^k}|^2$, we obtain, taking first the supremum of this expression, then the expectation $\E^0$, taking the $p^{{\rm}}$ power of this expression, taking next the expectation $\E$,  using BDG inequality, and Gronwall's lemma that \eqref{iukzqesbrlfdkc4} holds.}
}

\noindent {\bf Step~4: Estimate of the difference $X^k-X^m$, first part.} Let $\tau^{k,m}:=\tau^k\wedge \tau^m$ and let us set $\gamma_t:= \exp\{-\alpha(\delta t+H(\mu^k_t)+H(\mu^m_t))\}$ for $\alpha,\delta \geq 1$ to be chosen below. By It\^{o}'s formula \eqref{I2},
{\pe by using} the Lipschitz continuity of $b$ and $\sigma_i$ and absorbing the  terms in $dt$ as  in Step~3, using the bounds on the derivatives of $H$ and on $\sigma_1$ as well as \eqref{eq:condCondCons}, we get, for $\delta$ large enough:
\begin{align}
& \gamma_{t\wedge\tau^{k,m}}  |X^k_{t\wedge \tau^{k,m}}-X^m_{t\wedge \tau^{k,m}}|^2 \notag \\
&  \leq  \int_0^{t\wedge \tau^{k,m}} 2 \gamma_s  (X^k_s-X^m_s) \cdot \ld H(\mu^k_s)(X^k_s)\d K^k_s  
 - \int_0^{t\wedge \tau^{k,m}} 2\gamma_s  (X^k_s-X^m_s) \cdot \ld H(\mu^m_s)(X^m_s)\d K^m_s \notag  \\ 
&  + \int_0^{t\wedge \tau^{k,m}} 2\gamma_s (X^k_s-X^m_s)\cdot (\sigma_0(s,X^k_s)-\sigma_0(s, X^m_s))\,dB_s 
\notag \\
&  
+ \int_0^{t\wedge \tau^{k,m}}2\gamma_s  (X^k_s-X^m_s)\cdot  (\sigma_1(s,X^k_s)-\sigma_1(s, X^m_s))\,dW_s 
- \frac{\alpha\delta}{2} \int_0^{t\wedge \tau^{k,m}}\gamma_s |X^k_s-X^m_s|^2\d s \notag\\
	&  -  \alpha \int_0^{t\wedge \tau^{k,m}} \gamma_s|X^k_s-X^m_s|^2 (\E^0\left[\sigma_1^*(s,X^k_s)\ld H(\mu^k_s)(X^k_s) \right]+
	\E^0\left[\sigma_1^*(s,X^m_s)\ld H(\mu^m_s)(X^m_s) \right]\cdot dW_s  \notag\\
	&   - \alpha \beta^2 \int_0^{t\wedge \tau^{k,m}} \gamma_s |X^k_s-X^m_s|^2 (dK^k_s+ dK^m_s) .\label{lkjrhesrdcvCN22}
\end{align}

Let us now estimate the first two terms in the right-hand side: by \eqref{eq:LSC}, we have, for $0\leq s\leq {t\wedge \tau^{k,m}}$ and $\p^1-$a.s., 
\begin{align*}
& \E^0\left[ (X^k_s-X^m_s) \cdot \ld H(\mu^k_s)(X^k_s) \, \psi_k(H(\mu^k_s))\right] \\
 &\qquad   \leq \left(H(\mu^k_s)-H(\mu^m_s) +C\E^0\left[  |X^k_s-X^m_s|^2\right] \right)\psi_k(H(\mu^k_s))\\
 &\qquad   \leq \Bigl(-H(\mu^m_s)({\bf 1}_{\{\inf_{u\leq {t\wedge \tau^{k,m}}} H(\mu^m_u)> -2/m\}} + {\bf 1}_{\{\inf_{u\leq {t\wedge \tau^{k,m}}} H(\mu^m_u)\leq -2/m\}} ) \\
 &\qquad \qquad \qquad   +C\E^0\left[  |X^k_s-X^m_s|^2\right] \Bigr)\psi_k(H(\mu^k_s))\\
& \qquad \leq \left(\frac{2}{m} +C{\bf 1}_{\{\inf_{u\leq {t\wedge \tau^{k,m}}} H(\mu^m_u)\leq -2/m\}} +C \E^0\left[ |X^k_s-X^m_s|^2\right]\right)\, \psi_k(H(\mu^k_s)),
\end{align*} 
where we used that $z\psi_k(z)\leq 0$ and that $H$ is bounded. The symmetric inequality holds when exchanging the roles of $k$ and $m$. Therefore, 
taking the conditional expectation with respect to $\E^0$ in \eqref{lkjrhesrdcvCN22}, we obtain: 
\begin{align*}
& \gamma_{t\wedge \tau^{k,m}} \E^0\left[|X^k_{t\wedge \tau^{k,m}}-X^m_{t\wedge \tau^{k,m}}|^2\right] \notag \\
&  \leq  C\int_0^{t\wedge \tau^{k,m}} \gamma_s \left((1/m) +{\bf 1}_{\{\inf_{u\leq {t\wedge \tau^{k,m}}} H(\mu^m_u)\leq -2/m\}}
+ \E^0\left[ |X^k_s-X^m_s|^2\right]
\right)\, \d K^k_s  \notag\\ 
&   +C \int_0^{t\wedge \tau^{k,m}} \gamma_s \left((1/k) +{\bf 1}_{\{\inf_{u\leq {t\wedge \tau^{k,m}}} H(\mu^k_u)\leq -2/k\}}
+ \E^0\left[ |X^k_s-X^m_s|^2\right]
\right)\, \d K^m_s   \notag\\ 
&  + \int_0^{t\wedge \tau^{k,m}}2\gamma_s \E^0\left[(X^k_s-X^m_s) \cdot  (\sigma_1(s,X^k_s)-\sigma_1(s, X^m_s))\right]\,dW_s -\frac{\alpha\delta}{2} \int_0^{t\wedge \tau^{k,m}} \gamma_s\E^0\left[ |X^k_s-X^m_s|^2\right] ds \notag\\
	&   -  \alpha \int_0^{t\wedge \tau^{k,m}} \gamma_s\E^0\left[ |X^k_s-X^m_s|^2\right] \\
	& \qquad \qquad \times (\E^0\left[\sigma_1^*(s,X^k_s)\ld H(\mu^k_s)(X^k_s) \right]+
	\E^0\left[\sigma_1^*(s,X^m_s)\ld H(\mu^m_s)(X^m_s) \right])\cdot dW_s  \notag\\
	&  - \alpha \beta^2\int_0^{t\wedge \tau^{k,m}} \gamma_s \E^0\left[ |X^k_s-X^m_s|^2\right] \left\{  dK^k_s+ dK^m_s\right\}	. \notag 
\end{align*} 
So, for $\alpha$ large enough, we find (recalling that $\gamma_t$ is positive and bounded),	
\begin{align}	
& \gamma_{t\wedge \tau^{k,m}} \E^0\left[|X^k_{t\wedge \tau^{k,m}}-X^m_{t\wedge \tau^{k,m}}|^2\right] \notag \\
&  \leq  C \left((1/k)+ (1/m) +{\bf 1}_{\{\inf_{u\leq {t\wedge \tau^{k,m}}} H(\mu^m_u)\leq -2/m\}}+{\bf 1}_{\{\inf_{u\leq {t\wedge \tau^{k,m}}} H(\mu^k_u)\leq -2/k\}}\right)
(K^k_{t\wedge \tau^{k,m}}+K^m_{t\wedge \tau^{k,m}})   \notag\\ 
&  + \int_0^{t\wedge \tau^{k,m}}2\gamma_s \E^0\left[(X^k_s-X^m_s) \cdot  (\sigma_1(s,X^k_s)-\sigma_1(s, X^m_s))\right]\,dW_s -\frac{\alpha\delta}{2} \int_0^{t\wedge \tau^{k,m}} \gamma_s\E^0\left[ |X^k_s-X^m_s|^2\right]ds\notag\\
	&   -  \alpha \int_0^{t\wedge \tau^{k,m}} \gamma_s\E^0\left[ |X^k_s-X^m_s|^2\right]  \Bigl(\E^0\left[\sigma_1^*(s,X^k_s)\ld H(\mu^k_s)(X^k_s) \right]\notag \\
& \qquad \qquad\qquad\qquad\qquad\qquad\qquad\qquad  + 	\E^0\left[\sigma_1^*(s,X^m_s)\ld H(\mu^m_s)(X^m_s) \right]\Bigr )\cdot dW_s  .
\label{eq:liaevziud22}
\end{align} 
In order to proceed, we need to control the probability of the events $\{\inf_{u\leq {t\wedge \tau^{k,m}}} H(\mu^k_u)\leq -2/k\}$. This is the aim of the next step.\\

\noindent{\bf Step~5: Estimate of the probability $\p\left[\{ \inf_{u\leq T} H(\mu^k_u)\leq -2/k\}\right]$.} 
As $H([X_0])\geq 0$ and the map $u\to \mu^k_u$ is continuous on ${\mathcal P}^2(\R^n)$, on the event  $\Bigl\{ \inf_{u\leq T} H(\mu^k_u)\leq -2/k\Bigr\}$, there exist  $0\leq s\leq t'\leq \tau^k$ such that $H([X^k_s|W])= -1/k$, $H(\mu^k_u)\leq -1/k$ on $[s,t']$ and 
$H([X^k_{t'}|W])\leq -2/k$. Then \eqref{eq:H-HCond} (applied between $s$ and $t'$) implies that 
\begin{align*}
-(2/k)\geq -(1/k) +\beta^2 k^2(t'-s) + \int_s^{t'} \int_{\R^n} \sigma_1^T(u,x)\ld H^k(\mu^k_u)(x) \mu^k_u(\d x)\cdot \d W_u,
\end{align*} 
from which we derive that, for any $\gamma\in (1/3,1/2)$,  
\begin{align*}
& \sup_{0\leq s<t'\leq T} \left| \frac{1}{(t'-s)^\gamma} \int_s^{t'} \int_{\R^n} \sigma_1^T(u,x)\ld H^k(\mu^k_u)(x) \mu^k_u(\d x)\cdot \d W_u\right| \\
& \qquad\qquad  \geq
\inf_{0\leq s<t'\leq T} \left( \frac{1}{k(t'-s)^\gamma} +\beta^2 k^2(t'-s)^{1-\gamma}\right)  \geq C^{-1} k^{3\gamma-1}.
\end{align*} 
If we {\yin s}et 
\begin{align*}
M^k_t := \int_0^t \int_{\R^n} \sigma_1^T(u,x)\ld H^k(\mu^k_u)(x) \mu^k_u(\d x)\cdot \d W_u, 
\end{align*}
then, by Dubin-Schwarz,  
there is a standard $1-$dimensional Brownian motion $\tilde W^k$ such that 
$\displaystyle M^k_t = \tilde W^k_{{\yin \langle}M^k{\yin \rangle}_t}.$
Hence
\begin{align*}
\sup_{0\leq s<t\leq T} \left| \frac{1}{(t-s)^\gamma} \Bigl(M^k_t-M^k_s\Bigr)\right| & = \sup_{0\leq s<t\leq T} \left| \frac{1}{(t-s)^\gamma} \Bigl(\tilde W^k_{\langle M^k \rangle_t}-\tilde W^k_{\langle M^k\rangle_s}\Bigr)\right|\\
&  \leq C \sup_{0\leq s<t\leq CT}  \left| \frac{1}{(t-s)^\gamma} \Bigl(\tilde W^k_{t}-\tilde W^k_{s}\Bigr)\right|,
\end{align*}
because 
$$
\left|\langle M^k \rangle_t-\langle M^k\rangle_s\right|\leq C(t-s),
$$
since $\sigma_1$ is bounded and \eqref{eq:bilip0Cond} holds. 
This proves that 
\begin{align}
\p\left[\tau^k<T\right]\leq  \p \left[ \Bigl\{ \inf_{u\leq T} H(\mu^k_u) \leq -2/k\Bigr\}\right] \leq \p\left[  \sup_{0\leq s<t\leq CT}  \left| \frac{1}{(t-s)^\gamma} \Bigl(\tilde W^k_{t}-\tilde W^k_{s}\Bigr)\right|\geq  C^{-1} k^{3\gamma-1}\right] =: \ep_k,\label{eq:estipp}
\end{align}
where $\ep_k\to 0$ since $\gamma\in (1/3,1/2)$ by standard properties of  the Brownian motion. \\

\noindent {\bf Step~6: Estimate of the difference $X^k-X^m$, second part.}  We come back to \eqref{eq:liaevziud22}: taking expectation we obtain, by \eqref{eq:expmomK} and \eqref{eq:estipp},
\begin{align*}	
& \E\left[\gamma_{t\wedge \tau^{k,m}} |X^k_{t\wedge \tau^{k,m}}-X^m_{t\wedge \tau^{k,m}}|^2  \right] 
\notag \\
& \qquad  \leq  C \Bigl(\frac{1}{k}+ \frac{1}{m}\Bigr) +C\Bigl(\p^{1/2}\left[\{\inf_{u\leq {t\wedge \tau^{k,m}}} H(\mu^m_u)\leq -2/m\}\right] \notag \\
& \qquad \qquad \qquad + \p^{1/2}\left[\{\inf_{u\leq {t\wedge \tau^{k,m}}} H(\mu^k_u)\leq -2/k\}\right] \Bigr) 
\times \E^{1/2}\left[(K^k_{t\wedge \tau^{k,m}}+K^m_{t\wedge \tau^{k,m}})^2\right] 
 \notag\\ 
& \qquad \leq  C \left((1/k)+ (1/m) +\ep_k^{1/2}+\ep_m^{1/2}\right). 
\end{align*} 
Therefore
\begin{align}	
& \sup_{t\in [0,T]} \E\left[|X^k_{t\wedge \tau^{k,m}}-X^m_{t\wedge \tau^{k,m}}|^2\right] 
\leq  C \left((1/k)+ (1/m) +\ep_k^{1/2}+\ep_m^{1/2}\right).
 \label{oileazrendtlfgk:c}
\end{align} 
Coming back once {\yin again} to \eqref{eq:liaevziud22},  we have by the BDG inequality, \eqref{eq:expmomK}  and \eqref{eq:estipp},
\begin{align*}	
&\E\left[ \sup_{0\leq s\leq t} \gamma_{s\wedge \tau^{k,m}} \E^0\left[|X^k_{s\wedge \tau^{k,m}}-X^m_{s\wedge \tau^{k,m}}|^2\right]\right]  \\
& \qquad  \leq  C \left((1/k)+ (1/m) +\ep_k^{1/2}+\ep_m^{1/2}\right)   
+ C (\alpha+1) \E\Bigl[ \Bigl( \int_0^{t\wedge \tau^{k,m}} \left|\E^0\left[ |X^k_s-X^m_s|^2\right]\right|^2ds \Bigr)^{1/2}\Bigr]  .
\end{align*} 
Hence by Jensen's inequality 
\begin{align*}	
&\E\left[ \sup_{0\leq s\leq t} \gamma_{s\wedge \tau^{k,m}} \E^0\left[|X^k_{s\wedge \tau^{k,m}}-X^m_{s\wedge \tau^{k,m}}|^2\right]\right]  \\
& \leq 	C \left((1/k)+ (1/m) +\ep_m^{1/2}+\ep_k^{1/2} \right) \\
& \qquad \qquad +  C_\alpha\E\Bigl[ \sup_{0\leq s\leq t} (\gamma_{s\wedge \tau^{k,m}} \E^0\left[|X^k_{s\wedge \tau^{k,m}}-X^m_{s\wedge \tau^{k,m}}|^2\right])^{1/2}
\Bigl( \int_0^{t\wedge \tau^{k,m}} \E^0\left[ |X^k_s-X^m_s|^2\right] ds \Bigr)^{1/2}\Bigr] \\
& \leq 	C \left((1/k)+ (1/m) +\ep_m^{1/2}+\ep_k^{1/2} \right) 
+ \frac12 \E\Bigl[ \sup_{0\leq s\leq t} \gamma_{s\wedge \tau^{k,m}} \E^0\left[|X^k_{s\wedge \tau^{k,m}}-X^m_{s\wedge \tau^{k,m}}|^2\right]\Bigr]\\
& \qquad + C \E\left[\int_0^{t} |X^k_{s\wedge \tau^{k,m}}-X^m_{s\wedge \tau^{k,m}}|^2\right]{\yin ds},
\end{align*} 
from which we derive, by \eqref{oileazrendtlfgk:c},   
\begin{align}	
&\E\left[ \sup_{0\leq s\leq T} \E^0\left[|X^k_{s\wedge \tau^{k,m}}-X^m_{s\wedge \tau^{k,m}}|^2\right]\right]  
\leq 	C \left((1/k)+ (1/m) +\ep_m^{1/2}+\ep_k^{1/2} \right). \label{oilaersfd}
\end{align} 
We finally need to remove the $\E^0$ in the left-hand side. For this we come back to \eqref{lkjrhesrdcvCN22}.
Taking the supremum in time and the expectation therein we find, by the BDG inequality,  {\ph since $\gamma$ is bounded below and above by positive constants,}
\begin{align*}
& \E\left[ {\ph\sup_{0\leq s\leq T}  |X^k_{s\wedge \tau^{k,m}}-X^m_{s\wedge \tau^{k,m}}|^2 }\right] \notag \\
&\qquad   \leq C \E\left[ \int_0^{\tau^{k,m}}  {\ph \E^0\left[ |X^k_s-X^m_s|^2 \right]^{1/2} \, \E^0\left[ |\ld H([X^k_s|W])(X^k_s)|^2\right]^{1/2} dK^k_s  } \right] \notag \\ 
&\qquad  \qquad  + C \E\left[ \int_0^{\tau^{k,m}} {\ph \E^0\left[ |X^k_s-X^m_s|^2\right]^{1/2} \, \E^0\left[|\ld H([X^m_s|W])(X^m_s)|^2\right]^{1/2} dK^m_s } \right]  
\notag  \\ 
&\qquad\qquad   + C\E\left[ \left( \int_0^{ \tau^{k,m}}{\ph |X^k_s-X^m_s|^4\, ds}\right)^{1/2}\right]\\
&\qquad   \leq C \E\left[ \sup_{0\leq s\leq  \tau^{k,m}} {\ph \E^0\left[ |X^k_s-X^m_s|^2\right]^{1/2}  \left(K^k_{  \tau^{k,m}}+K^m_{  \tau^{k,m}}\right) }\right]\\ 
&\qquad \qquad   + C\E\left[ \sup_{0\leq s\leq  \tau^{k,m}} {\ph  |X^k_s-X^m_s| \left( \int_0^{\tau^{k,m}} |X^k_s-X^m_s|^2\, ds\right)^{1/2} }\right]\\
&\qquad   \leq C \E^{1/2}\left[ \sup_{0\leq s\leq \tau^{k,m}} \E^0\left[ |X^k_s-X^m_s|^2\right]\right] \E^{1/2}\left[(K^k_{  \tau^{k,m}}+K^m_{ \tau^{k,m}})^2\right] 
\\
&  \qquad \qquad 
+\frac12 \E\left[ {\ph \sup_{0\leq s\leq  \tau^{k,m}}  |X^k_s-X^m_s|^2 }\right] 
+ C\E\left[ \int_0^{ \tau^{k,m}}{\ph  |X^k_s-X^m_s|^2\d s}\right] ,
\end{align*} 
from which we obtain, in  view of our bounds on $K^k$ and $K^m$  in \eqref{eq:expmomK} and by \eqref{oileazrendtlfgk:c} and \eqref{oilaersfd}, 
\begin{align*}
& \E\left[ \sup_{0\leq s\leq T} |X^k_{s\wedge \tau^{k,m}}-X^m_{s\wedge \tau^{k,m}}|^2\right]  \leq 	C \left((1/k^{1/2})+ (1/m^{1/2}) +\ep_m^{1/4}+\ep_k^{1/4} \right).
\end{align*} 
We finally need to replace the stopping time $\tau^{k,m}$ by  the stopping times $\tau^k$ and $\tau^m$ in the above inequality. We have
\begin{align*}
& \E\left[ \sup_{0\leq s\leq T}  |X^k_{s\wedge \tau^{k}}-X^m_{s\wedge \tau^{m}}|^2\right]  \leq 	
\E\left[\E^0\left[\sup_{0\leq s\leq T}\gamma_s  |X^k_{s\wedge \tau^{k,m}}-X^m_{s\wedge \tau^{k,m}}|^2\right]{\bf 1}_{\{\tau^{k,m}=T\}}\right] \\
&\qquad +2 \E\left[ \left(\E^0\left[\sup_{0\leq s\leq T}  |X^k_{s\wedge \tau^{k}}|^2\right]+ \E^0\left[\sup_{0\leq s\leq t} |X^m_{s\wedge \tau^{m}}|^2\right]\right) 
({\bf 1}_{\{\tau^{k}<T\}}+{\bf 1}_{\{\tau^{m}<T\}})\right] \\
& \leq C \left((1/k^{1/2})+ (1/m^{1/2}) +\ep_m^{1/4}+\ep_k^{1/4} \right)\\
& \qquad 
+C \E^{1/2}\left[ (\E^0)^2\left[\sup_{0\leq s\leq T}  |X^k_{s\wedge \tau^{k}}|^{2}\right]+ (\E^0)^2\left[\sup_{0\leq s\leq t} |X^m_{s\wedge \tau^{m}}|^{2}\right]\right] 
\left(\p^{1/2}\left[ \{\tau^{k}<T\}\right]+\p^{1/2}\left[\{\tau^{m}<T\}\right] \right).
\end{align*} 
By the estimate \eqref{iukzqesbrlfdkc4} and  the estimate of $\p\left[ \{\tau^{k}<T\}\right]$ in \eqref{eq:estipp} we conclude that  
\begin{align*}
& \E\left[ \sup_{0\leq s\leq T}  |X^k_{s\wedge \tau^{k}}-X^m_{s\wedge \tau^{m}}|^2\right]  \leq C \left((1/k^{1/2})+ (1/m^{1/2}) +\ep_m^{1/4}+\ep_k^{1/4} \right).
\end{align*} 

\noindent {\bf Step~7: Analysis of the limit process.}  We infer from the above inequality that
$(X^k_{\cdot \wedge \tau^k})$ is a Cauchy sequence in $\mathcal S^2$. We denote by $X$ the limit of this sequence and set  $\mu_s=[X_s|W]$. We note that 
$$
\E\left[ \sup_{0\leq s\leq T}W_2^2(\mu^k_{s\wedge \tau^{k}}, \mu_s)\right]\leq \E\left[ \sup_{0\leq s\leq T} |X^k_{s\wedge  \tau^{k}}-X_s|^2\right]\to 0. 
$$
Next we remove the stopping time in the above inequality. By  \eqref{eq:estipp},  
$$
\p\left[\tau^k<T\right]\leq  \p \left[ \Bigl\{ \inf_{u\leq T} H(\mu^k_u) \leq -2/k\Bigr\}\right] \leq  \ep_k,
$$
where $(\ep_k)$ tends to zero as $k\to +\infty$. We can choose a  subsequence $(k')$, such that  $\sup_{0\leq s\leq T}W_2^2(\mu^{k'}_{s\wedge \tau^{k'}}, \mu_s)$ tends to $0$ a.s.. and, by the Borel-Cantelli Lemma, such that, a.s. and for $k'$ large enough, $H([X^{k'}_t|W]) \geq -2/k'$ on $[0,T]$.
In particular, this implies that, a.s. and for $k'$ large enough,  $\tau^{k'}=T$ and 
$$
\sup_{0\leq s\leq T}W_2^2(\mu^{k'}_{s}, \mu_s)\to 0. 
$$
Thus $H(\mu_t)\geq 0$ a.s. on $[0,T]$. Moreover, by \eqref{eq:ljhebrjshdx} and the convergence of $X^k$, there exists a continuous process $(L_t)$ such that   
$$
\E\left[ \sup_{0\leq t \leq T} \left| \int_0^t \int_{\R^n} \left|\ld H([X^{k'}_s|W])(x)\right|^2\, \psi_k(H([X^{k'}_s|W]))\d \mu^{k'}_s(\d x)\d s \, - L_t\right|^2\right] \to 0. 
$$
Note that $L$ is nondecreasing and adapted to $\m F^W$. Let us set 
$$
h(\mu):= \left[ \beta^2\vee  \int_{\R^n} \left| \ld H(\mu)(x)\right|^2 \mu(\d x) \right]^{-1}{\yin ,} \qquad \forall \mu \in  \m P^2(\rset^n)
$$
and 
$$
L^k_t= \int_0^t \int_{\R^n} \left|\ld H(\mu^k_s)(x)\right|^2\d \mu^k_s(\d x)\, \d K^k_s.
$$
Then $h$ is continuous and bounded on $\m P^2(\rset^n)$ and, as \eqref{eq:condCond} holds  and $-(1/k')\leq H(\mu^{k'}_s)\leq 0$ $dK^{k'}_s-$a.s. for $k'$ large enough,  
\begin{align*}
\lim_{k'} \int_0^t \ld H([X^{k'}_s|W])(X^{k'}_s)\, \d K^{k'}_s  & = \lim_{k'}  \int_0^t \ld H(\mu^{k'}_s)(X^{k'}_s)\, h(\mu^{k'}_s)\, \d L^{k'}_s \\
& = \int_0^t \ld H([X_s|W])(X_s)\, h([X_s|W])\, \d L_s.
\end{align*} 
Let us set 
$$
K_t = \int_0^t h([X_s|W])\, \d L_s.
$$
Then $K$ is nondecreasing and $\m F^W$ adapted and we have, passing to the limit in \eqref{eq:penCond}, 
\begin{align*}
	X_t & = X_0+\int_0^t b(s,X_s)\, ds + \int_0^t \sigma_0(s,X_s)\,dB_s + \int_0^t \sigma_1(s,X_s)\,dW_s \\
	& \qquad  +\int_0^t \ld H([X_s|W])(X_s)\, \d K_s, \qquad t\geq 0.
\end{align*}
Moreover, for any continuous and bounded map $\phi:\R\to \R$ such that $\phi(z)= 0$ on $[0,+\infty)$, one has 
\begin{align*}
& 0= \lim_{k'} \int_0^T\phi( H([X^{k'}_t|W]))  \psi_{k'}([X^{k'}_t|W]) \, dt  =  \lim_{k'} \int_0^T \phi(H([X^{k'}_t|W]))  h([X^{k'}_t|W])\, \d L^{k'}_t \\
& \qquad \qquad  =\int_0^T \phi(H([X_t|W]))  h([X_t|W])\, \d L_t =  \int_0^T \phi(H([X_t|W]))\, \d K_t.
\end{align*}
Therefore $\displaystyle \int_0^T {\bf 1}_{\{ H([X_t|W])>0\}} \, \d K_t=0.$ Note finally that  \eqref{estim.Cond} can be derived from \eqref{eq:expmomK} and \eqref{iukzqesbrlfdkc}. 
\end{proof}

\begin{proof}[Proof of Proposition \ref{prop.CondInitCN}] The proof goes along the same lines as the estimate of the difference $X^k-X^m$ in the proof of Theorem \ref{thm:Cond} (Step~4). Let $(X,K^X)$ and $(Y, K^Y)$ be two solutions, let us set $\mu^X_t=[X_t|W]$, $\mu^Y_t=[Y_t|W]$ and $\gamma_t:= \exp\{-\alpha({\delta t}+H(\mu^X_{t})+H(\mu^Y_{t}))\}$ for $\alpha, \delta \geq 1$ to be chosen below. By It\^{o}'s formula \eqref{I4} {\color{black}(which holds because \eqref{hypbetaK} is satisfied thanks to assumption \eqref{eq:bilip0Cond})}, we have,  using the Lipschitz continuity of $b$ and $\sigma_i$ and for $\delta>0$ large enough:  
\begin{align}
& \gamma_t  |X_{t}-Y_{t}|^2 \leq |X_0-Y_0|^2 \notag \\
&  +  \int_0^{t} 2 \gamma_s  (X_s-Y_s) \cdot (\ld H(\mu^X_s)(X_s)\, \d K^X_s  - \ld H(\mu^Y_s)(Y_s)\,  \d K^Y_s) \notag  \\ 
&  + \int_0^{t} 2\gamma_s (X_s-Y_s)\cdot (\sigma_0(s,X_s)-\sigma_0(s, Y_s))\,dB_s 
\notag \\
&  
+ \int_0^{t}2\gamma_s  (X_s-Y_s)\cdot  (\sigma_1(s,X_s)-\sigma_1(s, Y_s))\,dW_s 
- \frac{\alpha\delta}{2} \int_0^{t}\gamma_s |X_s-Y_s|^2\d s \notag\\
	&  -  \alpha \int_0^{t} \gamma_s|X_s-Y_s|^2 (\E^0\left[(\sigma_1(s,X_s))^*\ld H(\mu^X_s)(X_s) \right]+
	\E^0\left[(\sigma_1(s,Y_s))^*\ld H(\mu^Y_s)(Y_s) \right]\cdot dW_s  \notag\\
	&   - \alpha \int_0^{t} \gamma_s |X_s-Y_s|^2 \left\{ \E^0\left[|\ld H(\mu^X_s)(X_s)|^2\right] dK^X_s
	+ \E^0\left[|\ld H(\mu^Y_s)(Y_s)|^2\right] dK^Y_s\right\} . \label{lkjrhesrdcvCN2233}
\end{align} 
As   $H$ satisfies \eqref{eq:LSC}, we have,  for $\d K^X-$a.e. $s$, 
\begin{eqnarray*}
\E^0\left[(X_s-Y_s) \cdot \ld H(\mu^X_s)(X_s)\right] \leq H(\mu^X_s)-H(\mu^Y_s) + C\E^0\left[ |X_s-Y_s|^2\right]\leq C\E^0\left[ |X_s-Y_s|^2\right]{\yin .}
\end{eqnarray*}
The symmetric inequality holds when exchanging the roles of $X$ and $Y$. Therefore, 
taking the conditional expectation with respect to $\E^0$ in \eqref{lkjrhesrdcvCN2233}, we obtain, for $\alpha$ and $\delta$ large enough  and by \eqref{eq:condCond} 
(recalling that $\gamma_t$ is positive and bounded):	
\begin{align}	
& \gamma_{t} \E^0\left[|X_{t}-Y_{t}|^2\right] \leq   \E^0\left[|X_0-Y_0|^2\right]  \notag \\
&  + \int_0^{t}2\gamma_s \E^0\left[(X_s-Y_s) \cdot  (\sigma_1(s,X_s)-\sigma_1(s, Y_s))\right]\,dW_s -\frac{ \alpha\delta}{4} \int_0^{t} \gamma_s\E^0\left[ |X_s-Y_s|^2\right]ds\label{eq:liaevziud2233}\\
	&   -  \alpha \int_0^{t} \gamma_s\E^0\left[ |X_s-Y_s|^2\right]  \Bigl(\E^0\left[(\sigma_1(s,X_s))^*\ld H(\mu^X_s)(X_s) \right]
+ 	\E^0\left[(\sigma_1(s,Y_s))^*\ld H(\mu^Y_s)(Y_s) \right]\Bigr )\cdot dW_s  . \notag
\end{align} 
Taking expectation, we obtain: 
\begin{align}	
& \sup_{0\leq t\leq T}\E\left[ |X_{t}-Y_{t}|^2\right] \leq   \E\left[|X_0-Y_0|^2\right] . \label{eq:liaevziud2233bis}
\end{align} 
Coming back to \eqref{eq:liaevziud2233}, we have by the BDG inequality
\begin{align*}	
&\E\left[ \sup_{0\leq s\leq t} \gamma_{s} \E^0\left[|X_{s}-Y_{s}|^2\right]\right]  \leq \E\left[|X_0-Y_0|^2\right]
+ C (\alpha+1) \E\Bigl[ \Bigl( \int_0^{t} \left|\E^0\left[ |X_s-Y_s|^2\right]\right|^2ds \Bigr)^{1/2}\Bigr]  ,
\end{align*} 
from which we easily conclude using \eqref{eq:liaevziud2233bis} that
\begin{align}	
&\E\left[ \sup_{0\leq s\leq t} \E^0\left[|X_{s}-Y_{s}|^2\right]\right]  
\leq 	C_\alpha \E\left[|X_0-Y_0|^2\right]. \label{oilaersfd33}
\end{align} 
We finally remove the $\E^0$ in the left-hand side. For this we come back to \eqref{lkjrhesrdcvCN2233}, take the sup in time and the expectation to find, again for $\delta$ large enough, 
\begin{align*}
& \E\left[ \sup_{0\leq s\leq t}\gamma_s  |X_{s}-Y_{s}|^2\right]\leq \E\left[|X_0-Y_0|^2\right]  \notag \\
&  \qquad + C \E\left[ \int_0^{t} (\E^0\left[ |X_s-Y_s|^2\right)^{1/2} (\E^0\left[|\ld H(\mu^X_s)(X_s)|^2\right])^{1/2} dK^X_s  \right] \notag \\ 
&  \qquad  + C \E\left[ \int_0^{t}(\E^0\left[ |X_s-Y_s|^2\right)^{1/2} (\E^0\left[|\ld H(\mu^Y_s)(Y_s)|^2\right])^{1/2} dK^Y_s \right]  \notag  \\ 
&  \qquad  + C(\alpha+1)\E\left[ \left( \int_0^{t} \gamma_s^2 |X_s-Y_s|^4\d s\right)^{1/2}\right]\\
&  \leq \E\left[|X_0-Y_0|^2\right] + C \E\left[ \sup_{0\leq s\leq t} (\E^0\left[ |X_s-Y_s|^2\right])^{1/2}  (K^X_{ t}+K^Y_{ t})\right]  \\ 
& \qquad + C(\alpha+1)\E\left[ \sup_{0\leq s\leq t} (\gamma_s^{1/2} |X_s-Y_s|) \left( \int_0^{t} \gamma_s |X_s-Y_s|^2\d s\right)^{1/2}\right]\\
&  \leq \E\left[|X_0-Y_0|^2\right]  + C \E^{1/2}\left[ \sup_{0\leq s\leq t} \E^0\left[ |X_s-Y_s|^2\right]\right] \E^{1/2}\left[(K^X_{ t}+K^Y_{ t})^2\right] 
\\
&   \qquad  
+
\frac12 \E\left[ \sup_{0\leq s\leq t} \gamma_s |X_s-Y_s|^2\right] + C(\alpha+1) \E\left[ \int_0^{t} |X_s-Y_s|^2\d s\right].
\end{align*} 
In conclusion, by our bounds on $K^X$ and $K^Y$, \eqref{eq:liaevziud2233bis} and \eqref{oilaersfd33}, we have 
{\ph
\begin{align*}
& \E\left[ \sup_{0\leq s\leq t} |X_{s}-Y_{s}|^2\right]  \leq 	C\left(\E^{1/2}\left[|X_0-Y_0|^2\right]+\E\left[|X_0-Y_0|^2\right]\right).
\end{align*} 
}
\end{proof}

We are now ready to show the existence of a solution without the extra moment assumption on the initial condition: 

\begin{proof}[Proof of Theorem \ref{thm:Cond}: existence in the general case] Let $X_0^k:= X_0{\bf 1}_{\{|X_0|\leq k\}}$. Then $(X_0^k)$ has a moment of order $4$  and converges in $L_2(\Omega)$ to $X_0$. Let $(X^k, K^k)$ be the solution  to \eqref{eq:mainsCond} associated with the initial condition $X_0^k$. Proposition \ref{prop.CondInitCN} states that $(X^k)$ is a Cauchy sequence in $\mathcal S^2$. Then we can complete the proof exactly as before.
\end{proof}
{\pee
\begin{proof}[Proof of Theorem \ref{thm:Cond}: uniqueness] Uniqueness of the path $X$ follows from Proposition \ref{prop.CondInitCN} so that uniqueness of the process $K$ can be deduced by reproducing the proof done in Theorem \ref{en:eus}.

{\color{black} We finally explain why the weak uniqueness also holds for \eqref{eq:mainsCond}. Let $(X,K,B,W)$ and $(X',K',B',W')$ be two weak solutions of \eqref{eq:mainsCond}. Let us now built a pair of processes $(\tilde X, \tilde K)$ from the data $(B',W')$ as we did for $(X,K)$. It is hence clear that  $(X,K,B,W)$ and  $(\tilde X,\tilde K,B',W')$ have the same law. But from pathwise uniqueness we have $(X',K') = (\tilde X, \tilde K)$ so that weak uniqueness holds. }
\end{proof}
}
{\pe
\begin{proof}[Proof of Proposition \ref{prop.moduleconti}]
On the one hand we apply It\^o's formula to $H(\mu_t)$ and then proceed as in \eqref{eq:ljhebrjshdx} to obtain similarly to \eqref{eq:H-HCond} that 
\begin{eqnarray*}
\beta^{-2}\Big\{|H(\mu_{t}) - H(\mu_s)| + C(t-s) - \int_{s}^{t} \E^0\left[  \sigma_1^*(u,X_u)\ld H(\mu_u)(X_u) \right]\cdot \d W_u\Big\}
\geq  K_t-K_s,
\end{eqnarray*}
since $u \mapsto K_u$ is {\yin non-decreasing}. Also, from boundedness of $D_\mu H$ we have thanks to Cauchy-Schwartz{\yin 's} inequality that $\mathbb{P}^1$-a.s.
$$|H(\mu_{t}) - H(\mu_s)| \leq M (\E^0)^{1/2}\left[|X_t-X_s|^2\right].$$

On the other hand, applying It\^o's formula on $\gamma(t)|X_t-X_s|^2$ where $\gamma(t) = \exp\{-\alpha H(\mu_t)\}$ for some $\alpha >0$, then taking expectation $\E^0$ and using the boundedness of the coefficients and estimate \eqref{estim.Cond} we obtain that 
\begin{eqnarray*}
\E^0\left[\gamma(t)|X_t-X_s|^2\right] &\leq& C(t-s) + {\yin 2}{\color{black} \int_s^t \gamma(u)\E^0\left[(X_u-X_s)\cdot D_\mu H(\mu_u)(X_u)\right]d K_u } \\
&& - \alpha\int_s^t \gamma(u)\E^0[|X_u-X_s|^2]\E^0[| D_\mu H(\mu_u)(X_u)|^2]d K_u\\
&& +{\yin 2}\int_s^{t} \gamma_u \E^0\left[\sigma_1^*(u,X_u)(X_u-X_s)\right]\cdot dW_u\notag\\
&&-  \alpha \int_s^{t} \E^0\left[\gamma_u|X_u-X_s|^2\right] \E^0\left[\sigma_1^*(u,X_u)\ld H(\mu_u)(X_u)\right]\cdot dW_u,
\end{eqnarray*}
provided {\yin that} $\alpha$ is large enough. Since $H$ is semi-convex we have
 \begin{eqnarray*}
&& {\color{black}  \int_s^t \gamma(u)\E^0\left[(X_u-X_s)\cdot D_\mu H(\mu_u)(X_u)\right] d K_u}\\
 &\leq&  \int_s^t \gamma(u)\Big(H(\mu_u)-H(\mu_s) + {\color{black} \E^0[|X_t-X_s|^2]\Big)}d K_u \leq \int_s^t \gamma(u){\color{black}\E^0[|X_t-X_s|^2]}d K_u,
 \end{eqnarray*}
according to the Skorokhod condition. {\color{black} Hence we have, for $\alpha$ large enough,}
\begin{eqnarray}
\E^0\left[\gamma(t)|X_t-X_s|^2\right] &\leq& C(t-s) + {\yin 2} \int_s^{t} \gamma_u \E^0\left[\sigma_1^*(u,X_u)(X_u-X_s)\right]\cdot dW_u\notag\\
&&-  \alpha \int_s^{t} \E^0\left[\gamma_u|X_u-X_s|^2\right] \E^0\left[\sigma_1^*(u,X_u)\ld H(\mu_u)(X_u)\right]\cdot dW_u.\label{uneeqinterencoreunefois}
\end{eqnarray}
Recalling that $\gamma$ is bounded from below away from 0, then using the above estimate we have {\color{black}
\begin{eqnarray*}
 K_t-K_s &\leq& C\Big\{(t-s) -\int_{s}^{t}   \E^0\left[  \sigma_1^*(u,X^k_u)\ld H(\mu^k_u)(X^k_u) \right]\cdot \d W_u \notag\\
&&+ \Big( C(t-s)  -  \alpha \int_s^{t} \E^0\left[\gamma_u|X_u-X_s|^2\right] \E^0\left[\sigma_1^*(u,X_u)\ld H(\mu_u)(X_u)\right]\cdot dW_u\\
&& +   \int_s^{t} \gamma_u \E^0\left[\sigma_1^*(u,X_u)(X_u-X_s)\right]\cdot dW_u \Big)^{1/2}\Big\}.
\end{eqnarray*}
} Taking the expectation $\E$, using boundedness of the coefficients, \eqref{iukzqesbrlfdkc4} and BDG{\yin 's} inequality we deduce \eqref{modulecontiK}.

To obtain estimate on the first term in the left hand side of \eqref{modulecontiX} one just has to use standard computations and the following estimate:
\begin{eqnarray*}
\E^0\Big[\Big(\int_s^t D_\mu H(\mu_u)(X_u) d K_u\Big)^{p}\Big] &=& \int_t^s\ldots\int_t^s \E^{0}\big[D_\mu H(\mu_{u_1})(X_{u_1})\cdots D_\mu H(\mu_{u_p})(X_{u_p})\big] dK_{u_1}\cdots dK_{u_p}\\
&\leq &  \int_t^s\ldots\int_t^s \E^{0}\left[p^{-1}\sum_{\ell = 1}^p {\color{black} |D_\mu H(\mu_{u_\ell})(X_{u_\ell})|^p}\right]  dK_{u_1}\cdots dK_{u_p}\\
&\leq & \H_p |K_t-K_s|^p.
\end{eqnarray*}
To conclude, let us precise that the second term in the left hand side of \eqref{modulecontiX} is controlled starting from {\pee the difference of the processes $X_t-X_s$. Then, standards computations give
\begin{eqnarray*}
&&\E\left[\sup_{s\leq r \leq t}|\E^0\left[ |X_r-X_s|^2|^{{\yin p/2}}\right]\right] \\
&\leq &C_{T,p}\left\{ 1 +\E\left[\left(\sup_{s\leq r \leq t}\E^0\Big[\Big|\int_s^r D_\mu H(\mu_u)(X_u) d K_u\Big|^2\Big]\right)^{{\yin p/2}}\right] +  \E\left[\sup_{s\leq r \leq t}\left(\E^0\left[\left|\int_s^r \sigma^1(u,X_u) d W_u\right|^2 \right]\right)^{{\yin p/2}}\right]\right\}
\end{eqnarray*}
The second term in the above r.h.s is dealt using similar arguments as above while the third one can be estimated by using the boundedness of $\sigma^1$ and writing, provided {\pee $p/2$ is an integer,}
\begin{eqnarray*}
\left(\E^0\left[\left|\int_s^r \sigma^1(u,X_u) d W_u\right|^2 \right]\right)^{{\yin p/2}} &= &\Bigg(\bigg|\int_s^r \int_s^u \E^0\left[\sigma^1(u',X_{u'})(\sigma^1)^*(u,X_u)\right] d W_{u'}d W_u \\
& &+ \int_s^r \int_s^{u'} \E^0\left[\sigma^1(u',X_{u'})(\sigma^1)^*(u,X_u)\right] d W_{u}d W_{u'}\bigg|\Bigg)^{{\yin p/2}}.
\end{eqnarray*}
}

\end{proof}
}

\subsection{The mean field limit}
\label{subsec:A_mean_field_limit_cond}

We consider the particle system (for $i=1, \dots, N$)
\begin{equation}\label{eq:mainps_cond}	
	\left\{
	\begin{split}
	 & X^i_t  =X^i_0+\int_0^t b(s,X^i_s)\, ds + \int_0^t \sigma_0(s,X^i_s)\,dB^i_s + \int_0^t \sigma_1(s,X^i_s)\,dW_s + \int_0^t  \ld H\left(\mu^N_s\right)(X^i_s)\, dK^N_s, \\
	 & \mu^N_t = \frac{1}{N} \sum_{i=1}^N \delta_{X^i_t}, \quad H\left(\mu^N_t\right) \geq 0, \quad  \quad \int_0^t H\left(\mu^N_s\right) \, dK^N_s = 0, \quad t\geq 0.
	\end{split}
	\right.
\end{equation}
Here the data $b$, $\sigma$ and $H$ satisfy  conditions {\bf (Hcl)} at the beginning of the section and, in addition, $b$, {\yin$ \sigma_0$  and $\sigma_1$} are deterministic. The $B^i$ and $W$ are independent Brownian motions. The initial conditions of the particles $X^i_0$ are i.i.d. random variable with law $\mu_0$ and are independent of the $B^i$ and of $W$. {\yin $K^N$ is a continuous, nondecreasing process adapted to the filtration $\m F^N$ generated by the $X_0^i$, the $B^i$ and $W$.} We still assume that the probability space is $(\Omega, \p)=(\Omega^0\times \Omega^1, \p^0\otimes \p^1)$, where $\Omega^0$ supports the $X_0^i$ and the $B^i$, while $\Omega^1$ supports $W$ with associated filtration $\m F^W=\m F^1$. We  work on the set 
$$
\Omega_N:= \left\{ H\left(\frac{1}{N} \sum_{i=1}^N \delta_{X^i_0}\right) \geq 0\right\}. 
$$
Let us  remark that \eqref{eq:mainps_cond} reads as a classical reflected SDE in $(\R^{n})^N$, with  normal reflexion in the constraint 
\begin{equation*}
{\mathcal O}_N=\left\{(x_1,\dots, x_N)\in (\R^n)^N, \; H\left(\frac{1}{N} \sum_{i=1}^N \delta_{x_i}\right)> 0\right\}. 
\end{equation*}
Indeed, let us recall that, if
\begin{equation*}
	g(x_1,\ldots,x_N) = H\left(\frac{1}{N} \sum_{i=1}^N \delta_{x_i}\right),
\end{equation*}
then (see \cite{CDLL})
\begin{equation*}
	\partial_{x_i} g(x_1,\ldots,x_N) = \frac{1}{N}\, \ld H\left(\frac{1}{N} \sum_{i=1}^N \delta_{x_i}\right)(x_i).
\end{equation*}
Therefore the vector
$-(\ld H\left(\frac{1}{N} \sum_{i=1}^N \delta_{x_i}\right)(x_1),\dots, \ld H\left(\frac{1}{N} \sum_{i=1}^N \delta_{x_i}\right)(x_N))$  is an outward normal to the set ${\mathcal O}_N$ at the point $(x_1,\dots, x_N)\in \partial {\mathcal O}_N$. Existence and uniqueness of the solution are therefore immediate under our standing assumptions (cf. \cite{LiSz}). 

Our aim is to show the convergence, on $\Omega_N$, of the $X^i$ to the solution $\bar X^i$ to 
\begin{equation*}
	\left\{
	\begin{split}
	 & \bar X^i_t  =X^i_0+\int_0^t b(s,\bar X^i_s)\, ds + \int_0^t \sigma_0(s,\bar X^i_s)\,dB^i_s + \int_0^t \sigma_1(s,\bar X^i_s)\,dW_s + \int_0^t  \ld H\left(\bar \mu_s\right)(\bar X^i_s)\, \d K_s, \\
	 & \bar  \mu_t = [X^i_t|W],  \quad H\left(\bar \mu_t\right) \geq 0, \quad  \quad \int_0^t H\left(\bar \mu_s\right) \, \d K_s = 0, \quad t\geq 0.
	\end{split}
	\right.
\end{equation*}
Note that the processes $\bar \mu$ and $K$ do not depend on $i$ and $N$: this is obvious by {\pe the uniqueness in law, see Remark \ref{rem:uniquelaw}}. The $\bar X^i$ are no longer i.i.d., but they are i.i.d. given $W$. {\pe In particular, if $\bar \mu^N_s$ stands for the empirical measure of $(\bar X^i_s)_{1\leq i\leq N}$, then we will show later that
\begin{equation*}
\lim_{N\to+\infty} \E \left[\sup_{0\leq s\leq T} W_2^2(\bar\mu^N_s,\bar\mu_s)\right]=0. 
\end{equation*} 
}

As the initial conditions $X^i_0$ of the particles are i.i.d. and distributed according to $\mu_0$, $\Omega_N$ (the set of events for which the initial condition belongs to ${\mathcal O}_N$ and for which for which system \eqref{eq:mainps_cond} has a meaning) has not a full probability in general. There are several solutions to overcome this issue. We illustrate two of them in this paper. One can, for instance, modify the initial position and ``project'' it onto ${\mathcal O}^N$. We shall follow this approach in Subsection \ref{SEC_MF_FOR_BSDE} when we analyze the particle system associated with  backward SDEs with normal reflexion in law. Here we simply concentrate on the set $\Omega_N$. If one assumes, for instance, that $H(\mu_0)$ is positive, this event occurs with a probability which tends to $1$ as $N\to +\infty$.

\begin{thm} \label{thm:chaosS_cond}	For $T\geq 0$, there exists a constant $C(T)$ independent of $N$ such that
	\begin{align*} 
		\sup_{i=1,\dots, N}\sup_{t\in [0,T]} \e\left[|X^i_t-\bar X^i_t|^2{\bf 1}_{\Omega_N}\right]  \leq C(T)\E^{1/2}\left[\sup_{0\leq t\leq T} W_2^2(\bar\mu^N_s,\bar\mu_s)\right]
	\end{align*}
and, for $N$ large enough,  
	\begin{align*}
		\sup_{i=1,\dots, N} \E\left[\sup_{0\leq t\leq T}  |X^i_t-\bar X^i_t|^2{\bf 1}_{\Omega_N}\right]   \leq  C(T)\, \E^{1/4}\left[\sup_{0\leq s\leq T} W_2^2 (\bar\mu^N_s,\bar\mu_s) \right]. 
	\end{align*}
\end{thm}

{\pe

\begin{lemme}\label{cor:corratebsdepart}
Assume that the initial data satisfies the following moment condition: there exists an integer $p\geq 8$ such that $\M_p([X_0]) + \H_p$ is finite. Then, there exists $C_T:=C\big(T,\A{Hcl},\M_p([\xi]),\H_p\big)>0 $ such that
\begin{align*} 
		\E\left[\sup_{0\leq t\leq T} W_2^2(\bar\mu^N_s,\bar\mu_s)\right]  \leq C_T  \epsilon_N,
\end{align*}
where $\epsilon_N$ is given in Lemma \ref{ConvPart}.
\end{lemme}

\begin{proof}
Since $p\geq 8$ we have, thanks to Proposition \ref{prop.moduleconti} and Corollary \ref{cor:momentdoubleesp},  that the assumptions of Lemma \ref{ConvPart} are satisfied.
\end{proof}

}



Before proving Theorem \ref{thm:chaosS_cond}, let us point out some consequences of \eqref{eq:LSC} and \eqref{eq:utilechaos}. Let $(x_i)_{1\leq i\leq N}$ and $(y_i)_{1\leq i\leq N}$ be two sequences of points of $\rset^n$. We set 
\begin{equation*}
	\mu^N_y = \frac{1}{N} \, \sum_{i=1}^N \delta_{y_i}, \qquad \mu^N_x = \frac{1}{N} \, \sum_{i=1}^N \delta_{x_i}.
\end{equation*}
Inequality~\eqref{eq:LSC} becomes in this case
\begin{equation*}
	H\left(\mu^N_y\right)-H\left(\mu^N_x\right) \leq \e\left[\ld H\left(\mu^N_x\right)(X)\cdot (Y-X)\right]+C\E\left[|X-Y|^2\right],
\end{equation*}
where $[X]=\mu^N_x$ and $[Y]=\mu^N_y$. Of course, this inequality is true for any pair $(X,Y)$ with $[X]=\mu^N_x$ and $[Y]=\mu^N_y$ so that we can choose a specific law for $(X,Y)$. If all $(x_i)$ are distinct, let us pick $Y=f(X)$ with $f(x_i)=y_i$. We get
\begin{equation*}
	 \e\left[\ld H\left(\mu^N_x\right)(X)\cdot (f(X)-X)\right]+C\E\left[|X-f(X)|^2\right] = \frac{1}{N}\, \sum_{i=1}^N \ld H\left(\mu^N_x\right)(x_i)\cdot (y_i-x_i)
	 +  \frac{C}{N}\, \sum_{i=1}^N |y_i-x_i|^2.
\end{equation*}
Therefore,
\begin{equation}\label{eq:LCD}
	\frac{1}{N}\, \sum_{i=1}^N \ld H(\mu^N_x)(x_i)\cdot (x_i-y_i) \leq H\left(\mu^N_x\right) - H\left(\mu^N_y\right)+ \frac{C}{N}\, \sum_{i=1}^N |y_i-x_i|^2.
\end{equation}
This inequality also holds when the  $(x_i)$ are not necessarily distinct by approximation.  Moreover,  inequality~\eqref{eq:utilechaos} becomes, with $[X]=\mu^N_x$ and for any $\nu\in {\mathcal P}_2(\R^n)$,
\begin{equation}\label{eq:UC}
	\e\left[|\ld H(\mu^N_x)(X)-\ld H(\nu)(X)|^2\right] = \frac{1}{N}\sum_{i=1}^N |\ld H(\mu^N_x)(x_i)-\ld H(\nu)(x_i)|^2|\leq C\, W_2^2(\mu^N_X,\nu).
\end{equation}
We will also need below a few points on the L-derivative. For ${\bf x}=(x_1, \dots, x_N)\in (\R^n)^N$, we set $H^N({\bf x})=H(\frac{1}{N}\sum_{i=1}^N \delta_{x_i})$. We know from \cite{CDLL} that 
	\begin{align*}
		&D_{x_i} H^N({\bf x})= \frac{1}{N} \ld H(\frac{1}{N}\sum_{i=1}^N \delta_{x_i}) (x_i), \\
		& D^2_{x_ix_i}  H^N({\bf x})= \frac{1}{N} \partial_y\ld H(\frac{1}{N}\sum_{i=1}^N \delta_{x_i})(x_i)+\frac{1}{N^2} D^2_{\mu\mu}H(\frac{1}{N}\sum_{i=1}^N \delta_{x_i})(x_i,x_i),\\
		&D^2_{x_ix_j}  H^N({\bf x})= \frac{1}{N^2} D^2_{\mu\mu}H(\frac{1}{N}\sum_{i=1}^N \delta_{x_i})(x_i,x_j)\qquad i\neq j.
	\end{align*}
	So, if $(X^i)$ and $\mu^N$ satisfies the reflected SDE \eqref{eq:mainps_cond}, we have, by It\^{o}'s formula, 
	\begin{align}
		 dH(\mu^N_s)\; = & \frac{1}{N}\sum_{i=1}^N  \ld H(\mu^N_s)(X^i_s)\cdot b(s, X^i_s) ds \notag\\
		&\qquad +\frac{1}{2N}\sum_{i=1}^N {\rm Tr}\left((\sigma_0\sigma_0^*(s,X^i_s)+\sigma_1\sigma_1^*(s,X^i_s)) \partial_y\ld H(\mu^N_s)(X^i_s)\right)ds \notag\\
		&\qquad + \frac{1}{2N^2} \sum_{i=1}^N {\rm Tr}\left(\sigma_0\sigma_0^*(s,X^i_s)D^2_{\mu\mu}H(\mu^N_s)( X^i_s, X^i_s)\right)ds\notag\\
		& \qquad + \frac{1}{2N^2} \sum_{i,j=1}^N  {\rm Tr}\left(\sigma_1(s,X^j_s)\sigma_1^*(s,X^i_s)D^2_{\mu\mu}H(\mu^N_s)( X^j_s, X^i_s)\right)ds \notag\\
		&\qquad  + \frac{1}{N}\sum_{i=1}^N \ld H(\mu^N_s)(X^i_s)\cdot \sigma_0(s, X^i_s)dB^i_s  \label{I5}\\
		&\qquad  + \frac{1}{N}\sum_{i=1}^N   \ld H(\mu^N_s)(X^i_s)\cdot \sigma_1(s, X^i_s)dW_s 
		+ \frac{1}{N}\sum_{i=1}^N  \left| \ld H(\mu^N_s)(X^i_s)\right|^2 dK^N_s.\notag
	\end{align}
We also note that, by {\yin A}ssumption \eqref{eq:bilip0Cond},  
\begin{equation}\label{dskfh1}
\frac{1}{N}\sum_{i=1}^N |\ld H(\mu^N_s)(X^i_s)| = \int_{\R^n} |\ld H(\mu^N_s)(x)|\mu^N_s(dx) \leq M{\yin ,}
\end{equation}
and 
\begin{equation}\label{dskfh2}
 \frac{1}{N^2} \sum_{i,j=1}^N  \left|D^2_{\mu\mu}H(\mu^N_s)( X^j_s, X^i_s)\right| = 
 \int_{\R^{2n}} \left|D^2_{\mu\mu}H(\mu^N_s)(x,y)\right| \mu^N_s(dx)\mu^N_s(dy) \leq M.
\end{equation}

\begin{proof}[Proof of Theorem \ref{thm:chaosS_cond}] We work on $\Omega_N$. 
	Let us set $\Delta X^i_t=X^i_t-\bar X^i_t$. 
For $\alpha,\delta \geq 1$ to be chosen below, we set $\gamma_t=\exp\{-\alpha(\delta t +H(\mu^N_s)+H(\bar \mu_s))\}$. We note that $\gamma_t$ is bounded above and below by positive constant on bounded time intervals. {\color{black} Choosing $\alpha$ and $\delta$ large enough, we} have, by It\^{o}'s formula \eqref{I4} and \eqref{I5} and by  the Lipschitz continuity of  $b$, {\yin $\sigma_0$ and $\sigma_1$} as well as the bounds on the coefficients (see also \eqref{dskfh1} and \eqref{dskfh2}):
	\begin{align}
		& \gamma_t |\Delta X^i_t|^2   \leq 
		 2\, \int_0^t \gamma_s \Delta X^i_s\cdot\ld H\left(\mu^N_s\right)(X^i_s)\, dK^N_s 
		 - 2\, \int_0^t  \gamma_s \Delta X^i_s\cdot \ld H\left(\bar \mu_s\right)(\bar X^i_s)\, dK_s \notag\\
		 & \qquad -\frac{\alpha \delta}{2} \int_0^t  \gamma_s  |\Delta X^i_s|^2\d s + 2 \, \int_0^t \gamma_s \Delta X^i_s \cdot \left(\sigma_0(s,X^i_s)-\sigma_0(s,\bar X^i_s)\right) dB^i_s\label{eq:di_Cond}
 \\		
		& \qquad + 2 \, \int_0^t \gamma_s \Delta X^i_s \cdot \left(\sigma_1(s,X^i_s)-\sigma_1(s,\bar X^i_s)\right) dW_s 
				-  \frac{\alpha}{N}\sum_{j=1}^N \int_0^t \gamma_s |\Delta X^i_s|^2\ld H(\mu^N_s)(X^j_s)\cdot \sigma_0(s, X^j_s)dB^j_s \notag \\
		&\qquad  - \alpha \int_0^t \gamma_s |\Delta X^i_s|^2 \left( \frac{1}{N}\sum_{j=1}^N  \sigma_1^*(s, X^j_s)\ld H(\mu^N_s)(X^j_s)
		+ \E^0\left[\sigma_1^*(s, \bar X^i_s) \ld H(\bar \mu_s)(\bar X^i_s)\right]\right) \cdot dW_s 
\notag		 \\
		&\qquad 
		- \alpha \int_0^t \gamma_s |\Delta X^i_s|^2 \left(\frac{1}{N}\sum_{j=1}^N   \left| \ld H(\mu^N_s)(X^j_s)\right|^2\right)  dK^N_s
		-\alpha \int_0^t \gamma_s |\Delta X^i_s|^2 \E^0\left[\left| \ld H(\bar \mu_s)(\bar X^i_s)\right|^2\right]\d K_s .\notag
	\end{align}
	Using \eqref{eq:LCD}, we have, since $H\left(\mu^N_s\right)=0$ $dK$-a.e. and $H[\bar \mu_s]= H([\bar X^i_s|W])\geq 0$,
	\begin{align*}
		& \frac{2}{N} \sum_{i=1}^N \int_0^t \gamma_s \Delta X^i_s\cdot\ld H\left(\mu^N_s\right)(X^i_s)\, dK^N_s \\ 
		& \qquad \leq 2\, \int_0^t \gamma_s\left(H\left(\mu^N_s\right) - H\left(\bar \mu^N_s\right)\right) dK^N_s +\frac{C}{N}\sum_{i=1}^N \int_0^t \gamma_s  |\Delta X^i_s|^2\, dK^N_s\\
		& \qquad \leq 2\, \int_0^t\gamma_s \left(H(\bar \mu_s)- H\left(\bar \mu^N_s\right)\right) dK^N_s+\frac{C}{N}\sum_{i=1}^N \int_0^t \gamma_s  |\Delta X^i_s|^2\, dK^N_s.
	\end{align*}
	Since $H$ is $M$-Lipschitz for $W_2$, we get
	\begin{align}\label{eq:one_cond}
		&\frac{2}{N} \sum_{i=1}^N \int_0^t \gamma_s \Delta X^i_s\cdot\ld H\left(\mu^N_s\right)(X^i_s)\, dK^N_s\notag \\ 
		&\qquad  \leq 2M \, \sup_{0\leq s\leq t} W_2\left(\bar\mu^N_s,\bar\mu_s\right) \, K^N_t + \frac{C}{N}\sum_{i=1}^N \int_0^t  \gamma_s  |\Delta X^i_s|^2\, dK^N_s.
	\end{align}	
	We split the term $\displaystyle
		 - \frac{2}{N}\sum_{i=1}^N \int_0^t \gamma_s \Delta X^i_s\cdot \ld H\left(\bar\mu_s\right)(\bar X^i_s)\, dK_s
	$
	into two parts:
	\begin{gather*}
		A:=- \frac{2}{N}\sum_{i=1}^N \int_0^t \gamma_s \Delta X^i_s\cdot \ld H\left(\bar\mu^N_s\right)(\bar X^i_s)\, dK_s,\\
		B:=- \frac{2}{N}\sum_{i=1}^N \int_0^t \gamma_s \Delta X^i_s\cdot \left\{\ld H\left(\bar\mu^N_s\right)(\bar X^i_s)-\ld H\left(\bar\mu_s\right)(\bar X^i_s)\right\} dK_s.
	\end{gather*}
	We use \eqref{eq:LCD} for the first one: since $H(\mu^N_s)\geq 0$ and $H(\bar \mu_s)=0$ $dK_s$-a.e.	
	\begin{align*}
		& A = \frac{2}{N}\sum_{i=1}^N \int_0^t \gamma_s \left(\bar X^i_s-X^i_s\right)\cdot \ld H\left(\bar\mu^N_s\right)(\bar X^i_s)\, dK_s \\
		& \qquad  \leq  2\, \int_0^t \gamma_s \left(H\left(\bar \mu^N_s\right)-H\left(\bar\mu_s\right)\right) dK_s+\frac{C}{N}\sum_{i=1}^N \int_0^t \gamma_s  |\Delta X^i_s|^2\, dK_s,
	\end{align*}
	and using the Lipschitz continuity of $H$, we get
	\begin{equation}\label{eq:two_cond}
		A \leq 2M\, \sup_{0\leq s\leq t} W_2(\bar\mu^N_s,\bar\mu_s)\, K_t+  \frac{C}{N}\sum_{i=1}^N \int_0^t \gamma_s   |\Delta X^i_s|^2\, dK_s.
	\end{equation}
	Finally, we have, using Cauchy-Schwarz inequality, the global bound on $\int_{\R^n}|\ld H(\mu)|d\mu$ (given by assumption \eqref{eq:bilip0Cond}) and \eqref{eq:UC}, 
	\begin{align*}
		B & \leq 2\, \int_0^t\gamma_s  \left(\frac{1}{N} \sum_{i=1}^N |\Delta X^i_s|^2\right)^{1/2} \left(\frac{1}{N} \sum_{i=1}^N \left|\ld H\left(\bar\mu^N_s\right)(\bar X^i_s)-\ld H\left(\bar\mu_s\right)(\bar X^i_s)\right|^2\right)^{1/2} dK_s \\
		 & \leq C\, \int_0^t\gamma_s  \left(\frac{1}{N} \sum_{i=1}^N |\Delta X^i_s|^2\right)^{1/2} \left(\frac{1}{N} \sum_{i=1}^N \left|\ld H\left(\bar\mu^N_s\right)(\bar X^i_s)-\ld H\left(\bar\mu_s\right)(\bar X^i_s)\right|^2\right)^{1/4} dK_s \\
		& \leq C\, \int_0^t\gamma_s  \left(\frac{1}{N} \sum_{i=1}^N |\Delta X^i_s|^2\right)^{1/2} W_2^{1/2}\left(\bar\mu^N_s,\bar\mu_s\right) dK_s,
	\end{align*}
	from which we deduce
	\begin{equation}\label{eq:three_cond}
		B \leq \frac{C}{N} \sum_{i=1}^N\int_0^t \gamma_s  |\Delta X^i_s|^2\, dK_s + C \, \sup_{0\leq s\leq t} W_2(\bar\mu^N_s,\bar\mu_s)\, K_t.
	\end{equation}

	Inserting {\yin A}ssumption \eqref{eq:condCond}, estimates~\eqref{eq:one_cond}, \eqref{eq:two_cond} and \eqref{eq:three_cond} into \eqref{eq:di_Cond}, we get, for $\alpha$ and $\delta$ large enough,  
	\begin{align}
		& \frac{\gamma_t}{N} \, \sum_{i=1}^N |\Delta X^i_t|^2   
		 \leq C\, \left(K^N_t + K_t\right)\,\sup_{0\leq s\leq t} W_2(\bar\mu^N_s,\bar\mu_s) 
		+ \frac{2}{N} \sum_{i=1}^N \, \int_0^t \gamma_s \Delta X^i_s \cdot \left(\sigma_0(X^i_s)-\sigma_0(\bar X^i_s)\right) dB^i_s \notag\\
		&  + \frac{2}{N} \sum_{i=1}^N\, \int_0^t \gamma_s \Delta X^i_s \cdot \left(\sigma_1(s,X^i_s)-\sigma_1(s,\bar X^i_s)\right) dW_s 
				-  \frac{\alpha}{N^2}\sum_{i,j=1}^N \int_0^t \gamma_s |\Delta X^i_s|^2\ld H(\mu^N_s)(X^j_s)\cdot \sigma_0(s, X^j_s)dB^j_s \notag \\
		&  - \frac{\alpha}{N}\sum_{i=1}^N \int_0^t \gamma_s |\Delta X^i_s|^2 \left( \frac{1}{N}\sum_{j=1}^N  \sigma_1^*(s, X^j_s)\ld H(\mu^N_s)(X^j_s)
		+ \E^0\left[\sigma_1^*(s, \bar X^i_s) \ld H(\bar \mu_s)(\bar X^i_s)\right]\right) \cdot dW_s.  \label{redred}
	\end{align}
		We prove below the following exponential moment estimate on $K^N$: 
	\begin{equation}\label{eq:momentestiKN_cond}
		\sup_{N\geq 1}\e\left[ \exp\{\theta K^N_t\}{\bf 1}_{\Omega_N}\right]\leq C_\theta(t){\yin ,} \qquad \forall \theta>0.
	\end{equation}
	Then, taking expectation in \eqref{redred} and using {\yin the} Cauchy-Schwarz inequality, \eqref{eq:momentestiKN_cond} and the fact that $(K_t)$ has also exponential moments (Theorem \ref{thm:Cond}),
	we get, as $\Omega_N$ is independent of $B^i$ and $W$, 
	\begin{align}
		\E\left[\frac{1}{N} \, \sum_{i=1}^N |\Delta X^i_t|^2 {\bf 1}_{\Omega_N} \right] &
		\leq C\, 
		(\E^{1/2}\left[(K_t^N)^2{\bf 1}_{\Omega_N}\right]+\E^{1/2}\left[K_t^2\right])\,\E^{1/2}\left[\sup_{0\leq s\leq t} W_2^{2}(\bar\mu^N_s,\bar\mu_s)\right]\notag \\		
		&   \leq C(t)\,
		\E^{1/2}\left[\sup_{0\leq s\leq t} W_2^{2}(\bar\mu^N_s,\bar\mu_s)\right]. \label{kqbshdfjhb}
	\end{align}
Note that, by the exchangeability of the $(X^i)$, this implies that, for any $T>0$,  
	\begin{align}
		\sup_{t\in [0,T]} \sup_{i=1,\dots, N} \E\left[ |\Delta X^i_t|^2 {\bf 1}_{\Omega_N} \right]  \leq C(T)\,
		\E^{1/2}\left[\sup_{0\leq s\leq {\yin T}} W_2^{2}(\bar\mu^N_s,\bar\mu_s)\right]. \label{kqbshdfjhb1}
	\end{align}
In order to improve this inequality and have the sup in time into the expectation, we come back to \eqref{redred} and obtain, by the BDG inequality,
	\begin{align*}
		& \E\left[ \sup_{0\leq t\leq T}  \frac{1}{N} \, \sum_{i=1}^N \E^0\left[|X^i_t-\bar X^i_t|^2   \right]{\bf 1}_{\Omega_N} \right] \\
		& \qquad \leq  C\, \E\left[ \left(K^N_T{\bf 1}_{\Omega_N} + K_T\right)\,\sup_{0\leq s\leq T} W_2(\bar\mu^N_s,\bar\mu_s) \right] 
		+ C \E\left[\left(\int_0^T \frac{1}{N} \sum_{i=1}^N\left(\E^0\left[|\Delta X^i_s|^2 \right]\right)^2\d s\right)^{1/2}{\bf 1}_{\Omega_N}\right] ,
	\end{align*}
from which we infer by the estimate on $K$ in Theorem \ref{thm:Cond} and on $K^N$ in \eqref{eq:momentestiKN_cond}, by \eqref{kqbshdfjhb} and by the usual argument, that 
	\begin{align}
		& \E\left[ \sup_{0\leq {\yin t}\leq T}  \frac{1}{N} \, \sum_{i=1}^N \E^0\left[|X^i_t-\bar X^i_t|^2   \right]{\bf 1}_{\Omega_N} \right]  \leq  C(T)\, \E^{1/2}\left[ \sup_{0\leq s\leq T} W_2^2(\bar\mu^N_s,\bar\mu_s) \right] . \label{kqbshdfjhb2}
	\end{align}
		
Then we come back to {\pe the corresponding version of \eqref{eq:di_Cond} with $\gamma_t = \exp(-\delta t)$ therein}, take the sup in time and use the BDG inequality, the bound on $\ld H$ (through \eqref{dskfh1} in particular) to get 
	\begin{align*}
		 \E\left[\sup_{0\leq s\leq t} \gamma_t |\Delta X^i_s|^2{\bf 1}_{\Omega_N} \right]   \leq   C
		 \E\left[ \int_0^t (\E^0\left[|\Delta X^i_s|^2\right])^{1/2}  \{ \d K^N_s + \d K_s\}{\bf 1}_{\Omega_N}\right] 
		 + C \, \E\left[ \left(\int_0^t \gamma_s |\Delta X^i_s|^{{\color{black} 2}}\d s\right)^{1/2}{\bf 1}_{\Omega_N}\right]. 
	\end{align*}
By exchangeability, the first term in the right-hand side does not depend on $i$. Thus 	
\begin{align*}
& \E\left[ \int_0^t (\E^0\left[|\Delta X^i_s|^2\right])^{1/2}  \{ \d K^N_s + \d K_s\}{\bf 1}_{\Omega_N}\right] =
\E\left[ \int_0^t \frac{1}{N} \sum_{{\yin i}=1}^N (\E^0\left[|\Delta X^i_s|^2\right])^{1/2}  \{ \d K^N_s + \d K_s\}{\bf 1}_{\Omega_N}\right] \\
&\qquad  \leq \E\left[ \int_0^t \left(\frac{1}{N} \sum_{{\yin i}=1}^N \E^0\left[|\Delta X^i_s|^2\right]\right)^{1/2}  \{ \d K^N_s + \d K_s\}{\bf 1}_{\Omega_N}\right] \\
&\qquad  \leq \E^{1/2}\left[\sup_{0\leq s\leq t}\frac{1}{N} \sum_{{\yin i}=1}^N \E^0\left[|\Delta X^i_s|^2\right]{\bf 1}_{\Omega_N}\right]  \E^{1/2}\left[(K^N_t {\bf 1}_{\Omega_N}+  K_t)^2\right] .
\end{align*}
By \eqref{kqbshdfjhb1}, \eqref{kqbshdfjhb2} and the bounds on $K^N$ and $K$, this finally implies that 
	\begin{align*}
		 \E\left[\sup_{0\leq s\leq T}  |\Delta X^i_s|^2{\bf 1}_{\Omega_N}\right]   \leq  C(T)\, \left(\E^{1/4}\left[\sup_{0\leq s\leq T} W_2^2 (\bar\mu^N_s,\bar\mu_s) \right]
		 +\E^{1/2}\left[\sup_{0\leq s\leq T} W_2^2 (\bar\mu^N_s,\bar\mu_s) \right]\right). 
	\end{align*}

	To complete the proof, it remains to show that \eqref{eq:momentestiKN_cond} holds. We have, by It\^{o}'s formula \eqref{I5}, 
	\begin{align*}
		 H(\mu^N_T)\; = & H(\mu^N_0)+ \frac{1}{N}\sum_{i=1}^N \int_0^T  \ld H(\mu^N_s)(X^i_s)\cdot b(s, X^i_s) ds \\
		&\qquad +\frac{1}{2N}\sum_{i=1}^N \int_0^T {\rm Tr}\left((\sigma_0\sigma_0^*(s,X^i_s)+\sigma_1\sigma_1^*(s,X^i_s)) \partial_y\ld H(\mu^N_s)(X^i_s)\right)ds \\
		&  \qquad + \frac{1}{2N^2} \sum_{i=1}^N \int_0^T {\rm Tr}\left(\sigma_0\sigma_0^*(s,X^i_s)D^2_{\mu\mu}H(\mu^N_s)( X^i_s, X^i_s)\right)ds\\
		& \qquad + \frac{1}{2N^2} \sum_{i,j=1}^N \int_0^T {\rm Tr}\left(\sigma_1(s,X^j_s)\sigma_1^*(s,X^i_s)D^2_{\mu\mu}H(\mu^N_s)( X^j_s, X^i_s)\right)ds  \\
		&\qquad  + \frac{1}{N}\sum_{i=1}^N \int_0^T  \ld H(\mu^N_s)(X^i_s)\cdot \sigma_0(s, X^i_s)dB^i_s  \\
		&\qquad  + \frac{1}{N}\sum_{i=1}^N \int_0^T  \ld H(\mu^N_s)(X^i_s)\cdot \sigma_1(s, X^i_s)dW_s 
		+ \frac{1}{N}\sum_{i=1}^N \int_0^T \left| \ld H(\mu^N_s)(X^i_s)\right|^2 dK^N_s.
	\end{align*}
	Note that, by \eqref{eq:condCond} and on $\{H(\mu^N_s)=0\}$, one has 
	\begin{equation*}
	\frac{1}{N}\sum_{i=1}^N \left| \ld H(\mu^N_s(X^i_s)\right|^2=\int_{\R^n} \left| \ld H(\mu^N_s)(x)\right|^2 \mu^N_s(dx) \geq \beta^2.
	\end{equation*}
On the other hand, by \eqref{eq:bilip0Cond} and the $L^\infty$ bound on $b$,  
	\begin{align*}
		\left| \frac{1}{N}\sum_{i=1}^N \ld H(\mu^N_s)(X^i_s)\cdot b(s, X^i_s)\right| 
		&  \leq \frac{C}{N}\sum_{i=1}^N \left| \ld H(\mu^N_s)(X^i_s)\right|\leq C.
	\end{align*}
As {\yin $\sigma_0$ and $\sigma_1$} are bounded and  \eqref{eq:bilip0Cond} holds, we obtain in the same way {\yin
	\begin{align*}
	& \left| \frac{1}{N}\sum_{i=1}^N  \ld H(\mu^N_s)(X^i_s)\cdot b(s, X^i_s)\right| \\
		+& \left| \frac{1}{2N}\sum_{i=1}^N  {\rm Tr}\left(\sigma_0\sigma_0^*(s,X^i_s)+\sigma_1\sigma_1^*(s,X^i_s))\partial_y\ld H(\mu^N_s)(X^i_s)\right)\right|\\+ 
		&  \left| \frac{1}{2N^2} \sum_{i=1}^N {\rm Tr}\left(\sigma_0\sigma_0^*(s,X^i_s)D^2_{\mu\mu}H(\mu^N_s)( X^i_s, X^i_s)\right)\right| \\
		+ &  \left| \frac{1}{2N^2} \sum_{i, j=1}^N \int_0^T {\rm Tr}\left(\sigma_1(s,X^j_s)\sigma_1^*(s,X^i_s)D^2_{\mu\mu}H(\mu^N_s)( X^j_s, X^i_s)\right)ds \right|  \leq C. \\ 		
	\end{align*}
	}
As, in addition, $H$ is bounded, we finally have 
	\begin{align*}
		\beta^2 K^N_T \leq  C-   \frac{1}{N}\sum_{i=1}^N \int_0^T  \ld H(\mu^N_s)(X^i_s)\cdot \sigma_0(s, X^i_s)dB^i_s
		 - \frac{1}{N}\sum_{i=1}^N \int_0^T  \ld H(\mu^N_s)(X^i_s)\cdot \sigma_1(s, X^i_s)dW_s  , 
	\end{align*}
	which yield \eqref{eq:momentestiKN_cond}. 	
	\end{proof}

\subsection{{\pe Neumann Problem on Wasserstein space and }Feynman-Kac formula}
We connect here our reflected process with a PDE on the Wasserstein space. For this we assume {\pee in addition to \A{Hcl}} that  $b$ and $\sigma$ are deterministic {\pee, continuous in time} and let  ${\mathcal O}:= \{\mu\in \m P^2(\R^n), \; H(\mu)> 0\}$. 
Given a bounded, continuous map $G:\overline{\m O}\to \R$ we consider the map $u:[0,T]\times \overline{\mathcal O}\to \R$ defined by
\begin{equation}\label{eq.FKbis}
u(t_0,\mu_0)=\E\left[ G([X^{t_0,X_0}_T|W])\right],
\end{equation}
where $(X^{t_0,X_0}_s,K^{t_0,X_0}_s)$ solves the reflected SDE  \eqref{eq:mainsCond} on $[t_0,T]$ with initial condition $X^{t_0,X_0}_{t_0}=X_0$ and $[X_0]=\mu_0$. {\color{black} Thanks to Remark \ref{rem:uniquelaw}, the uniqueness in law holds for \eqref{eq:mainsCond} so that \eqref{eq.FKbis} is defined without ambiguity.} By the semi-group property, we have, for any {\yin ${\mathcal F}^W$-}stopping time $\tau\geq t_0$, 
\begin{equation}\label{eq.DDP}
u(t_0,\mu_0)=\E\left[ u(\tau, [X^{t_0,X_0}_{\tau}|W])\right]. 
\end{equation}
So we can expect that $u$ is, in a suitable sense, a solution to the following Neumann problem on the Wasserstein space:
\begin{equation}\label{lqbzlsdCond}
	\left\{ 
		\begin{split} 
		(i)\;\;\; &(\partial_t +\mathcal{A})u(t,\mu)  =0\qquad		{\rm in }\; (0,T)\times {\mathcal O},\\
		(ii)\;\;&\int_{\R^n} \ld u(t,\mu)(y)\cdot \ld H(\mu)(y)\mu(dy)= 0, \qquad  {\rm in }\; (0,T)\times \partial{\mathcal O},\\
		(iii)\;& u(T,\mu)=G(\mu), \qquad {\rm in }\; {\mathcal O},
		 \end{split}
	\right. 
\end{equation}
where the operator $\mathcal{A}$ is given, for any smooth function $\phi: [0,T]\times \mathcal{P}^2(\R^n) \to \R$, by
\begin{equation}\label{def:def_de_A}
\begin{array}{rl}
\displaystyle \mathcal{A}\phi(t,\mu)  := & \displaystyle \int_{\R^n}  b(t,z)\cdot \ld \phi(t,\mu)(z)d\mu(z) + \frac 12 \int_{\R^n} {\rm Tr}[(a_{{\yin 0}}(t,z)+a_{{\yin 1}}(t,z))\partial_z \ld \phi(t,\mu)(z)]d \mu(z)\\
&\displaystyle  \qquad +  \frac12 \int_{\R^n\times \R^n} {\rm Tr}\left(D^2_{\mu\mu} u(t,\mu)(z,z')\sigma_1(t,z)\sigma_1^{{\yin *}}(t,z')\right)\mu(\d z)\mu(\d z'),
\end{array}
\end{equation}
with  $a_i=\sigma_i\sigma_i^{{\yin *}}$ {\yin ($i=0,1$)}. 

Our  results are the following : we first show that the map $u$ given by \eqref{eq.FKbis} satisfies \eqref{lqbzlsdCond} in the viscosity sense.
Then we provide a Feynman-Kac formula showing that any classical solution of \eqref{lqbzlsdCond}, if it exists, is given by \eqref{eq.FKbis}. Note that the  question of the uniqueness of the viscosity solution of \eqref{lqbzlsdCond} is not considered in this work.

In order to show that the map $u$ given by \eqref{eq.FKbis} satisfies \eqref{lqbzlsdCond} in a viscosity sense, let us introduce the following definition (inspired from Definition 11.17 of \cite{CD17-2}).
\begin{df}\label{def:viscoBis}
We say that a continuous function $u$ is a viscosity solution of \eqref{lqbzlsdCond} if
\begin{enumerate}
\item[(i)] $u$ is continuous on $[0,T]\times \overline{ {\mathcal O}}$;
\item[(ii)] for any $(t,\mu)$ in $(0,T) \times \mathcal{P}^2(\R^n)$, for any test function $\varphi : [0,T] \times \mathcal{P}^2(\R^n) \to \R$ (see Definition 11.17 of \cite{CD17-2} for the class of test functions) such that $u-\varphi$ has a local minimum (resp. max) in $(t,\mu)$ we have
\begin{equation}\label{eq:linPDE_viscoBis}\left\lbrace\begin{array}{llll}
\displaystyle (\partial_t + \mathcal{A})\varphi(t,\mu) \leq  0,\ ({\rm resp.} \geq0) \text{ in } \mathcal{O},\\

\displaystyle \min\bigg\{(\partial_t +  \mathcal{A})\varphi(t,\mu)\ ;\  \int_{\R^n}  \ld H(\mu)(z) \cdot \ld\varphi(t,\mu)(z) \d \mu(z)\bigg\} \leq  0,\ ({\rm resp.} \geq0)  \text{ in } \partial \mathcal{O};\\
\end{array}\right.
\end{equation}
\item[(iii)] $\displaystyle u(T,\mu)= G(\mu)$ in $\mathcal{O}$.
\end{enumerate}
\end{df}

\begin{thm}\label{thm.viscsol}
If \A{Hcl} hold{\yin s} and if, in addition,  the coefficients $b$ and $\sigma$ are continuous in time and deterministic, then the map $u$ defined by \eqref{eq.FKbis} is a viscosity solution of \eqref{lqbzlsdCond} in the sense of Definition~\ref{def:viscoBis}.
\end{thm}

\begin{proof} Let us first check that $u$ is continuous. 
Let $\mu_0,\, \nu_0\in \mathcal{P}_2(\R^n)$ and let $[X_0]=\mu_0$ and $[Y_0]=\nu_0$. By Proposition \ref{prop.CondInitCN}, there exists a $C>0$ such that:
$$
\E[|X_t^{t_0,X_0} - X_{t}^{t_0,Y_0}|^2] \leq C \E\left[|X_0-Y_0|^2\right].
$$
Combining this inequality with the dynamic programming principle in \eqref{eq.DDP}, it is easy to see that the map $u$ is a continuous function. 

Let us now prove that $u$ solves  \eqref{lqbzlsdCond} in the viscosity sense.  Suppose that $\varphi : [0,T] \times \mathcal{P}^2(\R^n) \to \R$ is a test function such that $u-\varphi$ have a local minimum in $(t,\mu)$. We can assume, without loss of generality, that: $\varphi(t,\mu) = u(t,\mu)$ (which can be done by translating $\varphi$). Let $r>0$ be such that $u(s,\nu)\geq \varphi(s,\nu)$ for $s\in [t,t+r]$ and $W_2(\nu,\mu)\leq r$. Let us set $\mu_s:= [X_s^{t,\mu}|W]$ and let $\tau$ be the {\yin $\mathcal{F}^W$-}stopping time 
$$
\tau = \inf\{s\geq t, \; W_2(\mu_s,\mu)\leq r\}.
$$
Finally, for $h\in (0, r]$, we set $\tau_h= \tau\wedge (t+h)$. Note that, in view of the exponential estimate in Theorem \ref{thm:Cond} and our assumptions on the coefficients of the reflected SDE, it is not difficult to check that $\E[\tau_h-t)/h]\to 1$. On the one hand we have from \eqref{eq.DDP}:
\begin{equation}\label{kjnarezrned}
\varphi(t,\mu) = u(t,\mu)=\E\left[ u(\tau_h, \mu_{\tau_h})\right] \geq \E\left[ \varphi(\tau_h, \mu_{\tau_h})\right]. 
\end{equation}

On the other hand, by applying It\^{o}'s formula \eqref{I4}, 
we get:
\begin{eqnarray*}
\varphi(\tau_h,\mu_{\tau_h})  &= & {\yin \varphi(t,\mu) + } \int_t^{\tau_h} (\partial_t + \mathcal{A})\varphi(r,\mu_r) d r 
\\
&& 
+  \int_t^{\tau_h} \int_{\R^n} \ld H(\mu_r)(z) \cdot \ld \varphi(r,\mu_r)(z) \mu_r(\d z) d K_r  
\\
&& 
+ \int_{t_0}^{\tau_h} \int_{\R^n} \sigma_1^T(s,z)\ld \varphi(r,\mu_r)(z) \mu_r(\d z)\cdot \d W_r.
\end{eqnarray*}
Plugging this {\color{black} equality} into \eqref{kjnarezrned} and taking expectation we find: 
\begin{eqnarray*}
0  &\geq &  \E\Bigl[ \int_t^{\tau_h} (\partial_t + \mathcal{A})\varphi(r,\mu_r) d r 
+  \int_t^{\tau_h} \int_{\R^n} \ld H(\mu_r)(z) \cdot \ld \varphi(r,\mu_r)(z) \mu_r(\d z) d K_r  \Bigr] .
\end{eqnarray*}

Assume now that $\mu$ is such that $H(\mu) > 0$. Then choosing $r>0$ small enough we have $H(\mu_s) >0$ a.s. in $[t, \tau]$, so that $d K([t,\tau_h])=0$. Hence
\begin{eqnarray*}
&& \E\Bigl[ \int_t^{\tau_h} (\partial_t + \mathcal{A})\varphi(r,\mu_r) d r\Bigr] \leq 0.
\end{eqnarray*}
Dividing by $h>0$,    letting $h\to 0$ and using that $\E[\tau_h-t)/h]\to 1$, we eventually have 
\begin{eqnarray}
(\partial_t +\mathcal{A})\varphi(t,\mu)   \leq 0.
\end{eqnarray}

Assume now that $\mu$ is such that $H(\mu)=0$ and that
$$\min\bigg\{(\partial_t +  \mathcal{A})\varphi(t,\mu)\ ;\   \int \ld\varphi(t,\mu)(z) d \mu(z)\bigg\} >0.$$
Changing $r>0$ if necessary, there exists $\alpha>0$ such that, for any $s\in [t,t+r]$ and $W_2(\nu,\mu)\leq r$, we have 
$$
\min\bigg\{(\partial_t +  \mathcal{A})\varphi(s,\nu)\ ;\   \int \ld\varphi(s,\nu)(z) d \nu(z)\bigg\} \geq \alpha. 
$$
Then 
\begin{eqnarray*}
0  &\geq &  \E\Bigl[ \int_t^{\tau_h} (\partial_t + \mathcal{A})\varphi(r,\mu_r) d r 
+  \int_t^{\tau_h} \int_{\R^n} \ld H(\mu_r)(z) \cdot \ld \varphi(r,\mu_r)(z)\d \mu_r(z) d K_r  \Bigr] \geq \alpha \E\left[(\tau_h-t)\right].
\end{eqnarray*}
As $\E[\tau_h-t)/h]\to 1$, we find a contradiction.  
\end{proof}

In order to establish a kind of reverse statement, we now prove a Feynman-Kac formula.

\begin{prop}\label{prop.FK.cond} Assume that $u$ is a classical solution to \eqref{lqbzlsdCond}. Then $u$ is given by \eqref{eq.FKbis}. 
\end{prop}

\begin{proof} Let us set $\mu_s:= [X_s^{t_0,X_0}|W]$. Then, by It\^{o}'s formula \eqref{I4}: 
\begin{align*} 
 u(T,\mu_T)& = u(t_0,\mu_{t_0})+\int_{t_0}^T ( \partial_t +{\mathcal A}) u(s,\mu_s)  
+ \int_{t_0}^T \int_{\R^n}  \ld u(s,\mu_s)(y)\cdot \ld H(\mu_s)(y)\mu_s(dy) dK_s\\
& \quad + \int_{t_0}^T \int_{\R^n} \sigma_1^T(s,x)\ld H(\mu_s)(x) \mu_s(\d x)\cdot \d W_s, 
\end{align*}
where, by \eqref{lqbzlsdCond}-(ii),  
\begin{align*} 
& \int_{t_0}^T \int_{\R^n}  \ld u(s,\mu_s)(y)\cdot \ld H(\mu_s)(y)\mu_s(dy) dK_s\\
& \qquad = 
\int_{t_0}^T {\bf 1}_{\{\mu_s\in\partial {\m O} \}}  \left(\int_{\R^n}  \ld u(s,\mu_s)(y)\cdot \ld H(\mu_s)(y)\mu_s(dy)\right) dK_s=0. 
\end{align*}
From the equation satisfied by $u$, we obtain 
\begin{align*} 
& u(t_0,\mu_0)= u(t_0,\mu_{t_0})= \E\left[u(T,\mu_T)\right]= \E\left[G(\mu_T)\right]. 
\end{align*} 
\end{proof}

\section{Backward SDEs with normal reflexion in law} \label{sec:the_backward_case}

In this part, we are interested in Backward SDE constrained in law with normal reflexion. Namely, on $[0,T]$, $T>0$ we consider the following problem:
\begin{equation}\left \lbrace \begin{array}{ll}\label{eq:mainb}	
&\displaystyle  Y_t  =\xi+\int_t^T f(s,Y_s,Z_s)\, ds - \int_t^T Z_s\,dB_s + \int_t^T \ld H([Y_s])(Y_s) dK_s,\quad 0\leq t\leq T, \\
&\displaystyle  H([Y_t]) \geq 0, \quad 0\leq t\leq T, \qquad \int_0^T H([Y_s])  dK_s = 0,
\end{array}\right.
\end{equation}
{\pee where $B$ is a Brownian motion on a given probability space $(\Omega, \mathcal{F},\mathbb{P})$ endowed with the natural filtration $\mathcal F^B$,} the processes $Y,Z$ are {\color{black} of dimension $n$ and $n \times d$ respectively}, $f:\Omega \times \rset_+\times\rset^n \times \rset^{ n \times d}  \fl \rset^n $ and $H : \mathcal{P}^2(\R^n) \fl \R$.\\

We call solution of \eqref{eq:mainb} a triple of progressively measurable processes $(Y,Z,K)$ taking values in $\rset^n\times\rset^{n\times d}\times\rset$ such that $K$ is deterministic, continuous and {\yin nondecreasing} with $K_0=0$. We study the system under the following assumptions, referred as assumptions \A{A} in the following:

\begin{trivlist} 
	\item[\A{A1}] The functions $f:\Omega \times \rset_+\times\rset^n \times \rset^{ n \times d}  \fl \rset^n $ is Lipschitz {\yin w.r.t. $(y,z)$} uniformly in time and $\omega$ and adapted to the filtration $\m F^B$ {\yin for fixed $(y,z)$} and:
	$$\e \left[ \int_0^T |f(s,0,0)|^2 ds \right] <+\infty.$$
	\item[\A{A2}] The function $H: \m P(\rset^n)\fl \rset$ is fully $\m C^2$ (see {\color{black} the introduction for the definition}) and 
	\begin{itemize}
		\item there exist $0<\beta\leq M$ and $\eta>0$ such that
				\begin{align}\label{boundHback}
			{\color{black}\forall \mu \in \m P^2(\rset^n){\yin ,}\quad  \int_{\rset^n} |\ld H(\mu)|^2(x)\, \mu(dx)\leq M^2, } 
		\end{align}
		and 
		\begin{align}
			\forall \mu \in \m P^2(\rset^n)\; \mbox{\rm with} \; -\eta\leq H(\mu)\leq 0, \; \qquad \beta^2\leq  \int_{\rset^n} |\ld H(\mu)|^2(x)\, \mu(dx),  \label{eq:bilip_CV} 
		\end{align}
		\item there exists $C\geq 0$ such that
		\begin{equation}\label{eq:DFL_CV}
			\forall X \in L^2, \qquad \e\left[|\ld H([X])(X) -\ld H([Y])(Y)|^2\right] \leq C\, \e\left[|X-Y|^2\right].
		\end{equation}
	\end{itemize}
	\item[\A{A3}] The function $H$ is concave: for all $\mu, \nu$ in $\mathcal{P}_2(\R^n)$ 
		\begin{equation}\label{eq:LC_CV}
		\forall X\sim \mu,\ Y\sim \nu,\quad	H(\nu)-H(\mu) - \e\left[\ld H([X])(X)\cdot(Y-X)\right] \leq 0.
		\end{equation}
\item[\A{A4}] The terminal value $\xi$ is $\mathcal F_T$-measurable and 
\begin{equation*}
		H([\xi]) \geq 0, \quad \M_2([\xi]) < +\infty.
\end{equation*}
\end{trivlist}


{\pe 
\begin{rem}\label{point_inside_O}
{\color{black} Let us note that:}
\begin{trivlist} 
\item[(i)] It follows from \eqref{eq:bilip_CV} and \A{A4} that there exists $\tl$ in $\mathcal{P}^2(\R^n)$ such that $H(\tl) >0 $; 
\item[(ii)] ``Conversely'', if $\{H\geq -\eta\}$ is bounded and if there exists $\tl \in {\mathcal P}^2(\R^n)$ such that $H(\tl)>0$ then \eqref{eq:bilip_CV} holds.
\end{trivlist}
\end{rem}
{\pee
\begin{rem}\label{itoprod} In the following, we often use a r.v. $\TL$ with law $\tl$ as a pivotal point to obtain certain estimates. This leads us to use several times Itô's formula on quantities like $|Y_t-\TL|^2$. To do so, we work as follows:  we pick a r.v. $\tilde \Lambda \sim \tl$ independent of $\mathcal F^B$ on an atomless probability space $(\tilde \Omega,\tilde{\mathcal{F}},(\tilde{\mathcal{F}}_t)_{t\geq 0},\tilde{\mathbb{P}})$. Then, we apply Itô's formula on the space $(\bar \Omega,\bar{\mathcal{F}},(\bar{\mathcal{F}}_t)_{t\geq 0},\bar{\mathbb{P}}) = (\Omega\times \tilde \Omega, \mathcal{F}\otimes \tilde{\mathcal{F}},(\mathcal{F}_t\otimes \tilde{\mathcal{F}}_t)_{t\geq 0}, \mathbb{P} \otimes \tilde{\mathbb{P}})$ and denote by $\tilde \E$ and $\bar \E$ the associated expectations.

In the following, this framework will be referred as \emph{we work on the enlarged filtered probability space of Remark \ref{itoprod}}.
\end{rem}
}
\begin{proof}[Proof of Remark \ref{point_inside_O}] {\pe The first assertion is obvious pushing forward the measure $[\xi]$ along $\ld H$. Concerning the second assertion,} let $\lambda$ be such that $-\eta\leq H(\lambda)\leq 0$. Let $X,X_0\in L^2$ be such that $[X]=\lambda$, $[X_0]=\tl$. Then, {\color{black} by the concavity} of $H$, 
\begin{align*}
H(\tl)& \leq H(\lambda)+\E\left[ \partial_\mu H(\lambda)(X)\cdot (X_0-X)\right] \\ 
& \leq \E^{1/2}\left[ \left|\partial_\mu H(\lambda)(X)\right|^2\right] \E^{1/2}\left[ |X_0-X|^2\right].
\end{align*}
As $X$ is bounded in $L^2$ by some constant $C$ and $H(\tl)>0$, this proves \eqref{eq:bilip_CV} with $\beta=H(\tl) /C>0$.\\

Next, let us come back to the concavity assumption \A{A3}. Let now $\bx = (x_1, \ldots, x_N)\in (\R^{n})^N$ and $H^N$ be the finite dimensional projection of $H$:  
$$H^N(\bx):=H(\frac{1}{N}\sum_{i=1}^N \delta_{x_i}).$$
 We have that for all $\bx,\by$ in $(\R^{n})^N$:
\begin{eqnarray}
H^N(\by) - H^N(\bx) - \frac{1}{N} \sum_{i=1}^N D_\mu H(\mu_\bx^N)(x_i) (x_i-y_i) \leq 0,
\end{eqnarray}
since 
\begin{eqnarray}
\partial_{x_i} H^N(\bx)= \frac{1}{N} D_\mu H\left(\frac{1}{N}\sum_{\ell=1}^N \delta_{x_i}\right) (x_i), 
\end{eqnarray}
this means that the mapping $\bx \mapsto H^N(\bx)$ is concave in the classical sense. Hence, $D^2_\bx  H^N({\bx})$ is non-positive, where 
\begin{eqnarray}
\partial^2_{x_ix_j}  H^N({\bx})= \frac{1}{N} \partial_z D_\mu H\left(\frac{1}{N}\sum_{\ell=1}^N \delta_{x_i}\right)(x_i)\delta_{i,j}+\frac{1}{N^2} D^2_{\mu}H(\frac{1}{N}\sum_{\ell=1}^N \delta_{x_i})(x_i,x_j). \label{concecava}
\end{eqnarray}
\end{proof}

\subsection{Existence and uniqueness of the solution}

\begin{thm}\label{thm:existence_uniqueness}
Under \A{A} the BSDE with normal reflexion in law \eqref{eq:mainb} has a unique square integrable solution such that $K$ is deterministic.
\end{thm}
{\pe
\begin{prop}\label{lemmemomenY} Let $p\geq 2$. Assume that assumptions \A{A} hold true and assume in addition that $\M_p([\xi]) + \H_p$ is finite. {\pee Suppose further that $\E\left[\int_0^T|f({\yin s},0,0)|^pds\right] < + \infty$.}  Then, there exists a {\yin constant} $C_{p,T}:=C_{p,T}\left(\M_p([\xi]),\H_p,\E\left[\int_0^T|f({\yin s},0,0)|^pds\right] \right)>0$ such that
\begin{equation}
\E\left[ \sup_{0\leq t\leq T}|Y_t|^p\right] \leq C_{p,T}. 
\end{equation}
\end{prop}
}

\begin{proof} Let us start with the uniqueness result.
 Set $(Y,Z,K)$ and $(Y',Z',K')$ two solutions of \eqref{eq:mainb} and let us denote $\Delta Y_t = Y_t-Y'_t$, $\Delta Z_t = Z_t-Z'_t$ and $\Delta K_t = K_t-K'_t$. Let $\alpha \in \R$, applying It\^o's formula on $e^{\alpha t}|\Delta Y_t|^2$ one obtains
\begin{eqnarray*}
e^{\alpha t}|\Delta Y_t|^2 &=& \int_t^T \left(-\alpha e^{\alpha s}|\Delta Y_s|^2 +2  e^{\alpha s} (\Delta Y_s) \cdot \Delta f(s,Y_s,Z_s)\right) ds - 2 \int_t^{T} e^{\alpha s} (\Delta Y_s) \cdot (\Delta Z_s d B_s)\\
&&- \int_t^T e^{\alpha s} |\Delta Z_s|^2 ds + 2 \int_t^T e^{\alpha s}  (\Delta Y_s) \cdot (\Delta [D_\mu H([Y_s])(Y_s) d K_s]),
\end{eqnarray*}
where $\Delta[D_\mu H([Y_s])(Y_s)d K_s] =D_\mu H([Y_s])(Y_s)d K_s-D_\mu H([Y'_s])(Y'_s)d K_s'$.  By using classical arguments and assuming that $\alpha \geq 2||f||_{\rm Lip} + 2 ||f||_{\rm Lip}^2$, we derive:

\begin{eqnarray*}
e^{\alpha t}|\Delta Y_t|^2 + \frac{1}{2} \int_t^T e^{\alpha s} |\Delta Z_s|^2 ds &\leq & - 2 \int_t^{T} e^{\alpha s} (\Delta Y_s) \cdot (\Delta Z_s d B_s)\\
&& +2 \int_t^T e^{\alpha s}  (\Delta Y_s) \cdot (\Delta [D_\mu H([Y_s])(Y_s) d K_s]).
\end{eqnarray*}
From the concavity on $H$, thanks to the Skorokhod condition and since $\forall s \in [0,T],\ H([Y_s']) \geq 0$ we have 
\begin{eqnarray*}
 \int_t^T e^{\alpha s} \E  [ (\Delta Y_s) \cdot D_\mu H([Y_s])(Y_s) ] d K_s \leq   \int_t^T e^{\alpha s} ( H([Y_s]) - H([Y_s'])) d K_s \leq 0,
\end{eqnarray*}
and the same arguments lead to 
\begin{eqnarray*}
- \int_t^T e^{\alpha s} \E  [ (\Delta Y_s) \cdot D_\mu H([Y_s'])(Y_s') ] d K_s' \leq 0.
\end{eqnarray*}
Hence,
\begin{eqnarray*}
\E \left[ e^{\alpha t}|\Delta Y_t|^2 + \frac{1}{2} \int_t^T e^{\alpha s} |\Delta Z_s|^2 ds  \right]&\leq & 0,
\end{eqnarray*}
and uniqueness of $\{Y,Z\}$ follows. 

{\pee Let us now deal with the uniqueness of the processes $K$, $K'$. We aim at reproducing the approach implemented in the proof of Theorem \ref{en:eus}. {\color{black} However, we cannot use the chain rule on the Wasserstein space because of the lack of needed integrability of the processes $Z$ and $Z'$}. To overcome this problem, we are lead to apply classical It\^o's formula on i.i.d. copies of $(Y,Z)$.
Define $\{(\bar Y)^i, (\bar Z)^i\}_{1 \leq i \leq N}$ as $N$ copies of $(Y,Z)$. {\color{black} Writing $\mathbf{\bar Y} = (\bar Y^1,\ldots,\bar Y^N)^*$ {\color{black}, ${\bf\bar Z}  = (\bar Z^1,\ldots,\bar Z^N)^*$ }} and $\bar\mu_t^{N} = N^{-1}\sum_{i=1}^N \delta_{ (\bar Y_t)^i}$ we have, from classical It\^o's formula
\begin{eqnarray*}
 \frac{1}{N}\sum_{i=1}^N \int_s^t \left| D_\mu H (\bar \mu^{N}_u)(\bar Y_u^i)\right|^2 dK_u &= & H(\bar \mu^{N}_s)  - H(\bar \mu^{N}_t)- \frac{1}{N}\sum_{i=1}^N \int_s^t  D_\mu H(\bar \mu^{N}_u)(\bar Y_u^i)\cdot f(u,Y_u,Z_u) du\notag \\
&&+\frac{1}{2}\int_s^t {\rm Tr}\left[D^2_\bx H^N({\bf \bar Y_u}){\bf\bar Z_u  (\bar Z_u)^*}\right]du \notag\\
&&  + \frac{1}{N}\sum_{i=1}^N \int_s^t  D_\mu H(\bar \mu^{N}_u)(\bar Y_u^i)\cdot (\bar Z_u)^i dB^i_u  \notag\\
&& = \frac{1}{N}\sum_{i=1}^N \int_s^t \left| D_\mu H (\bar \mu^{N}_u)(\bar Y_u^i)\right|^2 dK_u'.
\end{eqnarray*}
We have, passing to the limit $N\to +\infty$, that
\begin{equation*}
\int_s^t \E\left[\left| D_\mu H ( \mu_u)(Y_u)\right|^2\right] dK_u = \int_s^t \left[\left| D_\mu H ( \mu_u)(Y_u)\right|^2\right] dK_u'.
\end{equation*}
We can thus repeat the end of the proof of uniqueness in Theorem \ref{en:eus} to deduce that $K=K'$. The result follows.
}

Let us now handle the existence part. The proof is divided into three parts: we first assume that the driver $f$ is space independent and  bounded and construct a solution thanks to a penalization approximation; we then extend the result thanks to a truncation argument when it is in $\mathbb{L}^2(\Omega, \mathbb{L}^2\left([0,T],\R^{n})\right)$; we finally show, thanks to a Picard iteration, that the result holds true under \A{A1}.\\

\emph{Step 1 : Existence for bounded and space independent generator.}  In what follows we start by assuming that 
$$\A{T1},\quad \forall (\omega,s,y,z) \in \Omega \times [0,T] \times \R^n \times \R^{n\times d},\, f(\omega,s,y,z) = f(\omega,s),\qquad \exists \kappa>0,\, |f(\omega,s)| \leq \kappa,\,\mathbb{P}-a.s.. $$

In this case, we construct a solution through a penalization approach.  For $k\geq 1$, let $\psi_k : \rset \fl \rset_+$ be the function defined by
\begin{equation*}
\psi_k(x)= r \text{ if } x\leq -1/k, \quad \psi_k(x) = -krx, \text{ if } -1/k\leq x\leq 0, \quad \psi_k(x) = 0, \text{ if } x \geq 0.
\end{equation*}
Note that the function $\psi_k$ depends on the constant $r>0$ which will be chosen later. Let $\left(Y^k,Z^k\right)$ be the solution to the following BSDE:

\begin{equation}\label{eq:bsdem}
Y^k_t = \xi + \int_t^T f(s) ds - \int_t^T Z^k_s dB_s + \int_t^T D_\mu H([Y^k_s])(Y_s^k) \psi_k(H([Y^k_s])) ds, \quad 0\leq t\leq T.
\end{equation}
We have {\yin the following Proposition whose proof will be given at the end of this Step.}

\begin{prop}\label{prop:penalized} Under \A{A}, for any $k \geq 1$ there exists a unique solution of \eqref{eq:bsdem} satisfying
\begin{equation*}
\e\left[\sup_{0\leq t\leq T}|Y^k_t|^2 + \int_0^T |Z^k_s|^2\, ds \right] \leq C_{r}.
\end{equation*}
Moreover, under \A{T1}, {\color{black} there exists $r>0$ such that} for all $k \geq 1$ and $0\leq t \leq T$ we have
\begin{equation}\label{esti:Hpourpenalisee}
H([Y_t^k]) \geq -1/k.
\end{equation}
\end{prop}

Introducing 
\begin{equation*}
K^k_t = \int_0^t \psi_k(H([Y^k_s])) ds,
\end{equation*}
we rewrite the previous BSDE as

\begin{equation*}
Y^k_t = \xi + \int_t^T f(s) ds - \int_t^T Z^k_s dB_s + \int_t^T  D_\mu H([Y^k_s])(Y_s^k) dK^k_s, \quad 0\leq t\leq T.
\end{equation*}

Let $ k, \ell $ in $\mathbb{N}^*$ be fixed and set $\Delta Y_t : = Y^k_t - Y^\ell_t$ and $\Delta Z_t = Z^k_t-Z^\ell_t$. {\pee Applying Itô's formula to $|\Delta Y_t|^2$, and using Young's inequality, we obtain:}

\begin{eqnarray}\label{eq:cauchy}
 |\Delta Y_t|^2 + \int_t^T  |\Delta Z_s|^2 ds &\leq  &2 \int_t^T  \Delta Y_s \cdot D_\mu H([Y_s^k])(Y_s^k) dK^k_s \\
&&- 2 \int_t^T \Delta Y_s \cdot D_\mu H([Y_s^\ell])(Y_s^\ell) dK^\ell_s - 2 \int_t^T \Delta Y_s \cdot \Delta Z_s dB_s. \notag
\end{eqnarray}
From the $L$-concavity of $H$, Proposition \ref{prop:penalized} and the definition of $K^k$ we have
\begin{eqnarray*}
\e \left[\int_t^T \Delta Y_s \cdot D_\mu H([Y_s^k])(Y_s^k) dK^k_s \right] &\leq& \e \left[\int_t^T (H([Y_s^k]) - H([Y_s^\ell]) )\psi_k(H([Y_s^k]) ds\right]\\
&\leq & 0 + \frac{r}{\ell} (T-t),
\end{eqnarray*}
since $x\psi_k(x)\leq 0$ and $\psi_k$ is bounded by $r$. Arguing similarly we have
\begin{equation*}
-\e\left[\int_t^T  \Delta Y_s \cdot D_\mu H([Y_s^\ell])(Y_s^\ell) dK^\ell_s\right] \leq  \frac{r}{k} (T-t).
\end{equation*}
Hence we deduce that {\pee for some $C_{T,r}>0$:}
\begin{equation*}
	\sup_{0\leq t\leq T} \e\left[|\Delta Y_t|^2 + \int_0^T |\Delta Z_s|^2 ds \right] \leq \pee {C_{T,r}}\left(\frac{1}{k}+\frac{1}{\ell}\right),
\end{equation*}
and coming back to~\eqref{eq:cauchy}, we get from BDG's inequality together with the fact that $\sup_k \psi_k \leq r $ that {\pee there exists $C_{T,r}'>0$ such that}
\begin{equation*}
\e\left[\sup_{0\leq t\leq T} |\Delta Y_t|^2 + \int_0^T |\Delta Z_s|^2 ds\right] \leq {\pee C_{T,r}' }\left(\frac{1}{\sqrt k} + \frac{1}{\sqrt \ell}\right).
\end{equation*}
Thus, $\{(Y^k,Z^k)\}_{k\geq 1}$ is a Cauchy sequence in $\ys\times\zs$. Let us denote by $(Y,Z)$ its limit. Since $\psi_k$ is bounded by $r$ for all $k$, $K^k$ is Lipschitz with $| K^k|_{\text{Lip}}\leq r$. Hence, by Ascoli-Arzela theorem, up to a subsequence, $(K^k)$ converges towards a non decreasing, Lipschitz continuous function $K$ in $\mathcal{C}([0,T],\R)$.

It is straightforward to check that $(Y,Z,K)$ solves~\eqref{eq:mainb}. Indeed, 
\begin{itemize}
	\item $H([Y_t])= \lim_{k\to\infty} H([Y_t^k]) \geq 0$ by Proposition \ref{prop:penalized}.
	\item the Skorokhod condition is also satisfied : since $x\psi_k(x)\leq 0$,
	\begin{equation*}
		0\leq \int_0^T H([Y_s])dK_s = \lim_{k\to\infty} \int_0^T H([Y_s^k]) dK^k_s = \lim_{k\to\infty} \int_0^TH([Y_s^k]) \psi_k(H([Y_s^k])) ds \leq 0.
	\end{equation*}
\end{itemize}

\begin{proof}[Proof of Proposition \ref{prop:penalized}]
First note that existence and uniqueness of a solution follows from \cite{BLP09}. We \phrase. Arguing as in the proof of uniqueness, for $\alpha$ large enough, we have, for $0\leq t\leq T$,
	\begin{multline}\label{eq:borne}
		e^{\alpha t} |Y^k_t-\TL|^2 + \frac 12 \int_t^T e^{\alpha s} |Z^k_s|^2 ds  \\ \leq  e^{\alpha T}|\xi-\TL|^2 + \int_t^T e^{\alpha s}|f(s,\TL,0)|^2\, ds + 2 \int_t^T e^{\alpha s}(Y^k_s-\TL) \cdot D_\mu H([Y^k_s])(Y_s^k) dK^k_s  - 2 \int_t^T e^{\alpha s} (Y^k_s-\TL) \cdot Z^k_s dB_s.
	\end{multline}
	
Using the $L$-concavity of $H$ together with the fact that $H({\pe \tilde \lambda})\geq 0$ we have,
\begin{align*}
	{\pee \bar \e}\left[\int_t^T e^{\alpha s} (Y^k_s-\TL) \cdot D_\mu H([Y^k_s])(Y_s^k) dK^k_s\right] \leq \int_t^T e^{\alpha s}\left(H([Y^k_s])-H(\tl)\right) \psi_k(H([Y^k_s])) ds \leq 0,
\end{align*}
since $x\psi_k(x)\leq 0$. It follows that there exists $C_{\tl} : =C(\M_2(\tl),T,\A{A}) >0$ independent of $k$ and $r$ such that

\begin{equation}\label{eq:bound}
		\sup_{0\leq t\leq T} \e\left[|Y^k_t|^2 + \int_0^T |Z^k_s|^2\, ds\right] \leq C_{\tl}. 
\end{equation}

Coming back to the estimate~\eqref{eq:borne} and using BDG's inequality, we deduce that, for some constant $C_{a,r}: =C(a,T,r,\A{A}) >0$ independent of $k$,
\begin{equation*}
\e\left[\sup_{0\leq t\leq T}|Y^k_t|^2 + \int_0^T |Z^k_s|^2\, ds \right] \leq C_{a,r},
\end{equation*}
in other words, the sequence $\{(Y^k,Z^k)\}_{k\geq 1}$ is bounded in $\ys\times {\yin H^2}$.\\

{\color{black}Let us now prove \eqref{esti:Hpourpenalisee}. Our approach here is in the same spirit as the proof of estimate \eqref{eq:rigole} in paragraph \ref{subsec:Existence_and_uniqueness_of_the_solution}. {\color{black} Again, due to the lack of needed integrability of the process $Z^k$, we apply classical It\^o's formula and then use the concavity property of the function $H$, see the computations below.}\\}

Assume that there exists $0<t_0<T$ such that $H([Y_{t_0}^k]) < -1/k$. Define $t=\inf\{ u \geq t_0:\ H([Y_u^k]) \geq -1/k\}$ and $s=\sup\{ u \leq t:\ H([Y_u^k]) \leq H([Y_{t_0}^k]) \vee -\eta\}$. Since $H([\xi]) \geq 0$, and $ -\eta <-1/k$ we have that $0<s<t<T$ and $-\eta \leq H([Y_u^k]) \leq -1/k$ on $[s,t]$. 

{\pe Let $\{(\bar Y^k)^i, (\bar Z^k)^i, \bar f^i\}_{1 \leq i \leq N}$ be $N$ copies of $(Y^k,Z^k,f)$. {\color{black} Writing $\mathbf{\bar Y^k} = ((\bar Y^k)^1,\ldots,(\bar Y^k)^N)^*$ {\color{black} and ${\bf\bar Z^k}  = ((\bar Z^k)^1,\ldots,(\bar Z^k)^N)^*$ }} and $\bar\mu_t^{N,k} = N^{-1}\sum_{i=1}^N \delta_{ (\bar Y^k_t)^i}$} we have, from classical It\^o's formula
\begin{eqnarray*}
 H(\bar \mu^{N,k}_s) &= & H(\bar \mu^{N,k}_t)+ \frac{1}{N}\sum_{i=1}^N \int_s^t  D_\mu H(\bar \mu^{N,k}_u)((\bar Y^k_u)^i)\cdot f^i(u) du\notag \\
&&-\frac{1}{2}\int_s^t {\rm Tr}\left[D^2_\bx H^N({\bf \bar Y_u^k}){\bf\bar Z_u^k  (\bar Z_u^k)^*}\right]du \notag\\
&&  - \frac{1}{N}\sum_{i=1}^N \int_s^t  D_\mu H(\bar \mu^{N,k}_u)((\bar Y^k_u)^i)\cdot (\bar Z_u^k)^i dB^i_u  \notag\\
&& + \frac{1}{N}\sum_{i=1}^N \int_s^t \left| D_\mu H (\bar \mu^{N,k}_u)((\bar Y^k_u)^i)\right|^2 dK_u.
\end{eqnarray*}
Hence we have from \eqref{concecava}:

\begin{eqnarray}\label{ressertpourmoments}
 H(\bar \mu^{N,k}_s) &\geq&  H(\bar \mu^{N,k}_s) +\frac{1}{2}\int_s^t {\rm Tr}\left[D^2_\bx H^{N,k}({\bf \bar Y_u^k}){\bf\bar Z_u^k  (\bar Z_u^k)^*}\right]du \notag\\ 
 &= & H(\bar \mu^{N,k}_t)+ \frac{1}{N}\sum_{i=1}^N \int_s^t  D_\mu H(\bar \mu^{N,k}_u)((\bar Y^k_u)^i)\cdot f^i(u) du\notag \\
&&  - \frac{1}{N}\sum_{i=1}^N \int_s^t  D_\mu H(\bar \mu^{N,k}_u)((\bar Y^k_u)^i)\cdot (\bar Z_u^k)^i dB^i_u  \notag\\
&& + \frac{1}{N}\sum_{i=1}^N \int_s^t \left| D_\mu H (\bar \mu^{N,k}_u)((\bar Y^k_u)^i)\right|^2 dK_u^k.
\end{eqnarray}

Taking the expectation and then the limit over $N$ on both sides leads, see the proof of Theorem 5.98 of \cite{CD17-1}, to 

\begin{eqnarray*}
H([Y_s^k]) &\geq & H([Y_t^k]) + \int_s^t \e\left[D_\mu H([Y_u^k])(Y_u^k) \cdot f(u) \right]{\color{black} du} \\
&&+\int_s^t \e {\yin [}|D_\mu H([Y_u^k])(Y_u^k) |^2{\yin ]} \psi_k(H([Y_u^k])) du.
\end{eqnarray*}

From \A{A} we get that there exi{\yin s}ts a constant $C:=C(a,T,\A{A}) >0$ such that

\begin{eqnarray*}
H([Y_s^k])  &\geq & H([Y_t^k]) - C(t-s)+ \int_s^t \e \left[|D_\mu H([Y_u^k])(Y_u^k) |^2\right] \psi_k(H([Y_u^k])) du.
\end{eqnarray*}
Using now \A{T1} and the definition of $\psi_k$ we get 

\begin{eqnarray*}
H([Y_s^k])  &\geq & H([Y_t^k]) - C(t-s) + \beta^2 (t-s) r.
\end{eqnarray*}
Since $H([Y_s^k])\leq H([Y_{t_0}^k]) \vee -\eta < -1/k \leq H([Y_t^k])$ this implies $- C + \beta^2 r < 0$ which {\color{black} cannot hold for all $r$ large enough and then leads} to a contradiction. 
\end{proof}

{\pee 
\begin{rem}\label{constructKpart} Note that from the proof of the above Proposition and Step1 remain true for a generator $f$ satisying $\sup_{t \leq T}\E[|f(t)|^2] <+\infty$.

\end{rem}

}

\emph{Step 2 : from space independent generator in $\mathbb{L}^\infty$ to {\rm {\yin H}}$^2$.} We now prove, using a truncation argument, that the BSDE \eqref{eq:mainb} admits a solution when 

$$\A{T2},\quad \forall (\omega,s,y,z) \in \Omega \times [0,T] \times \R^n \times \R^{n\times d},\quad f(\omega,s,y,z) = f(\omega,s),\qquad \E \left[ \int_0^T| f|^2(s) ds \right] < + \infty.$$

Let $Y_t^m$ be given by
\begin{eqnarray*}
Y_t^m = \xi + \int_t^T f(s)\mathbf{1}_{|f(s)| \leq m} ds -\int_t^T Z_s^m dB_s +{\color{black}  \int_t^T D_\mu H([Y_s^m])(Y_s^m) d K_s^m}.
\end{eqnarray*}
For each $m$ in $\mathbb{N}$, this BSDE admits a unique solution from \emph{step 1}. Moreover we have for any $m,\ell \geq 0$, {\pee by applying Itô's formula on $e^{\alpha t}|Y_t^\ell - Y_t^m|^2$, $\alpha \geq 1$, that}:
\begin{eqnarray}
&& e^{\alpha t}|Y_t^\ell-Y_t^m|^2 + \int_t^T e^{\alpha s} |Z_s^\ell-Z_s^m|^2 ds \label{eq:bsde:fdansM2}\\
&{\yin =} &  {\yin 2\int_t^T  e^{\alpha s}(Y_s^\ell - Y_s^m)\cdot (f(s)\mathbf{1}_{|f(s)| \leq \ell} - f(s)\mathbf{1}_{|f(s)| \leq m})} ds -2\int_t^T  e^{\alpha s}(Y_s^\ell-Y_s^m) \cdot (Z_s^\ell-Z_s^m) d B_s\notag\\
&& + 2\int_t^T e^{\alpha s} (Y_s^\ell-Y_s^m)  \cdot D_\mu H([Y_s^\ell])(Y_s^\ell) d K_s^\ell -  2\int_t^T e^{\alpha s} (Y_s^\ell-Y_s^m)\cdot D_\mu H([Y_s^m])(Y_s^m) d K_s^m\notag\\
&& - \alpha \int_t^T e^{\alpha s} |Y_s^\ell-Y_s^m|^2 ds.\notag
\end{eqnarray}
Taking now the expectation, using $L$-concavity, Skorokhod property and the fact that for each $k$ the marginals $\{[Y_t^k]\}_{0 \leq t \leq T}$ satisfy the constraint, we deduce that:
\begin{eqnarray}\label{eq:l-concplusskoro}
&&\int_t^T  \E\left[e^{\alpha s}(Y_s^\ell-Y_s^m)  \cdot D_\mu H([Y_s^\ell])(Y_s^\ell) \right] d K_s^\ell -\int_t^T e^{\alpha s} \e \left[(Y_t^\ell-Y_t^m) \cdot D_\mu H([Y_s^m])(Y_s^m) \right]d K_s^m\notag\\
&\leq & 2 \E\left[\int_t^T e^{\alpha s} \left(H([Y_s^\ell]) - H([Y_s^m])\right) (d K_s^\ell-d K_s^m) \right]\leq 0,
\end{eqnarray}
so that, 
\begin{eqnarray*}
&& \sup_{t \leq T}\e\left[ e^{\alpha t} |Y_t^\ell-Y_t^m|^2 + \int_t^T e^{\alpha s} |Z_s^\ell-Z_s^m|^2 ds \right] \\
&\leq & {\yin C}\int_0^T \e\left[ |f(s)\mathbf{1}_{|f(s)| \leq \ell} - f(s)\mathbf{1}_{|f(s)| \leq m}|^2 \right] ds.
\end{eqnarray*}

Let us now show that  $K_T^m$ is bounded, uniformly in $m$. To do so, we \phrase. One has{\yin ,}
 \begin{eqnarray}\label{eqKKKK}
 e^{\alpha t}|Y_t^m-\TL|^2 +  \int_t^T e^{\alpha s} |Z_s^m|^2 ds &=& e^{\alpha T}|\xi-\TL|^2 +  2\int_t^T e^{\alpha s}(Y_s^m-\TL) \cdot f(s)\mathbf{1}_{|f(s)| \leq m}  ds\\
&&   -2\int_t^T e^{\alpha s}(Y_s^m-\TL) \cdot Z_s^m d B_s  + 2\int_t^T  e^{\alpha s}(Y_s^m-\TL) \cdot D_\mu H([Y_s^m])(Y_s^m) d K_s^m\notag\\
&& - \alpha \int_t^T e^{\alpha s} |Y_s^m-\TL|^2 ds.\notag
\end{eqnarray}
Taking the expectation on both side, using Young's inequality and $L$-concavity together with the Skorokhod condition, we obtain for $\alpha$ large enough that
 \begin{eqnarray*}
\sup_{t\leq T} {\pee \bar \e} \left[ e^{\alpha t} |Y_t^m-\TL|^2\right] + {\yin \e}\left[ \int_0^T e^{\alpha s} |Z_s^m|^2 ds \right] + e^{\alpha T}H(\tl)(K_T^m-K_t^m)  \leq   {\yin C\Bigg(}{\pee \bar \e} \left[ |\xi-\TL||^2 \right] +\e \left[ \int_0^T |f(s) |^2ds \right] {\yin \Bigg)}.
\end{eqnarray*}

Hence, coming back to \eqref{eq:bsde:fdansM2}, taking the supremum in time, the expectation and choosing $\alpha$ large enough lead to 
\begin{eqnarray}\label{eq:bsdegen1}
&& \e \left[ \sup_{t\leq T } e^{\alpha t}|Y_t^\ell-Y_t^m|^2 + \int_0^Te^{\alpha s} |Z_s^\ell-Z_s^m|^2 ds \right] \\
&\leq & {\yin C \Bigg(}  \e \left[\int_0^T |f(s)\mathbf{1}_{|f(s)| \leq \ell} - f(s)\mathbf{1}_{|f(s)| \leq m}|^2 ds \right]+ \e\left[ \sup_{t\leq T}\left| \int_t^T e^{\alpha s} (Y_t^\ell-Y_t^m) \cdot (Z_s^\ell-Z_s^m) d B_s \right| \right]\notag\\
&& + \int_0^T  \e^{1/2}\left[\left| Y_s^\ell-Y_s^m\right|^2\right] \e^{1/2}\left[\left| D_\mu H([Y_s^\ell])(Y_s^\ell) \right|^2 \right]d K_s^\ell \notag\\
&&+  \int_t^T  \e^{1/2} \left[|Y_t^\ell-Y_t^m|^2\right] \e^{1/2}\left[ |D_\mu H([Y_s^m])(Y_s^m)|^2\right] d K_s^m{\yin \Bigg)},\notag
\end{eqnarray}
thanks to Cauchy-Schwarz's inequality and since $K^m, K^\ell$ are deterministic. Since from BDG's inequality and Young's inequality one has
\begin{eqnarray}
&&\e\left[ \sup_{t\leq T}\left| \int_t^T e^{\alpha s} (Y_t^\ell-Y_t^m) \cdot (Z_s^\ell-Z_s^m) d B_s \right| \right]\label{eq:bsdegen2}\\
 & \leq & \epsilon  \e\left[ \sup_{t\leq T}e^{\alpha t} |Y_t^\ell-Y_t^m|^2 \right] + \frac{C}{\epsilon} \e \left[ \int_0^T e^{\alpha s}|Z_s^\ell-Z_s^m|^2 ds\right],\notag
\end{eqnarray}
we deduce that $\{Y^m,Z^m\}$ is a Cauchy sequence in $\mathcal{S}^2 \times {\rm {\yin H}}^2$ that converges to some limit $\{Y,Z\}$.

{\pee Let us now deal with the convergence of the process $K^m$. By the convergence of $\{Y_t^m,Z_t^m\}$ it is clear that the sequence of processes $(L^m)_{m\geq 0}$ defined by
\begin{equation*}
\forall t \geq 0, \qquad L_t^m := \int_0^t \ld H(\mu_s^m)(Y_s^m) dK_s^m,
\end{equation*}
converges in $\mathcal{S}^2$ to some process $L$. Set $h(\mu) := \left(\int | \ld H(\mu)(\cdot)|^2 d\mu \vee \beta^2\right)^{-1}$, for all $t \geq 0$, $\varphi_t := [\ld H(\mu_t)(Y_t)]^*$. Let us denote by $\big(\varphi^M_t,h^M(\mu_t)\big)_{t \geq 0}$ the discretized version of $\big(\varphi_t,h(\mu_t)\big)_{t \geq 0}$ along a subdivision $(t_k)_{0\leq k \leq M}$ of $[0,T]$ of stepsize $1/M$. Let $M>0$, for $l,m\geq 0$ we have
\begin{eqnarray*}
|K_t^m - K_t^l| &=& \left|\E\left[\int_0^t h(\mu_s^m)\varphi_s^m  dL_s^m\right] - \E\left[\int_0^t h(\mu_s^l)\varphi_s^l dL_s^l\right]\right|\\
&= & \Bigg|\E\left[\int_0^t \{ h(\mu_s^m)\varphi_s^m - h(\mu_s^l)\varphi_s^l\} dL_s^m\right] + \E\left[\int_0^t \{h(\mu_s^l)\varphi_s^l - h(\mu_s)\varphi_{s}\} dL_s^m\right] \\
&&+ \E\left[\int_0^t \{h(\mu_s)\varphi_s - h^M(\mu_s)\varphi_{s}^M\} dL_s^m\right] \\
&&+ \E\left[\int_0^t \{h^M(\mu_s)\varphi_s^M\} (dL_s^m-dL_s^l)\right] + \E\left[\int_0^t\{h^M(\mu_s)\varphi_s^M - h(\mu_s^l)\varphi_s^l\} dL_s^l\right]\Bigg|.
\end{eqnarray*}
Note now that for all $k\geq 0$, we have $dL_s^k = \ld H(\mu_s^k)(Y_s^k) dK_s^k$ with $K^k$ deterministic. Writing $L^m$ and $L^l$ as this, one can inverse the integration and expectation operators in the first, second, third and fifth terms of the above r.h.s. We can thus deduce, taking first the limit superior in $m,l$ and then the limit superior over $M$, that the process $K^m$ converges to some deterministic continuous process $K$.

 It is not hard to see that this limit $(Y,Z,K)$ is a solution of \eqref{eq:mainb} under the standing assumption on the generator $f$ assumed in this part.
}

\emph{Step 3: from space independent generator to space dependent generator.} We now only assume that the generator $f$ satisfies our assumptions \A{A}. We are now going to construct a solution to BSDE \eqref{eq:mainb} by using a Picard iteration. Set $(Y^0,Z^0)=(0,0)$ and define recursively $(Y^m,Z^m)$ as the unique solution of
\begin{eqnarray*}
&&Y_t^m = \xi + \int_t^T f(s,Y_s^{m-1},Z_s^{m-1}) ds - \int_t^T Z_s^m d B_s + \int_t^T D_\mu H([Y_s^m])(Y_s^m) dK_s^m,\\
&&\forall t \geq 0,\quad H([Y_t^m])\geq 0, \quad \int_t^T H([Y_s^m])d K_s^m =0, \qquad \quad m\geq 1.
\end{eqnarray*}
Let us denote for any $m\geq 0$: $\Delta Y^{m+1} = Y^{m+1}-Y^m$ and $\Delta Z^{m+1} = Z^{m+1}-Z^{m}$.

On the one hand, applying Itô's formula on $e^{\alpha t} |\Delta Y_t^{m+1}|^2$, $\alpha \geq 1$, we have
\begin{eqnarray}
&& e^{\alpha t}|\Delta Y_t^{m+1}|^2 + \int_t^T e^{\alpha s} |\Delta Z_s^{m+1}|^2 ds \notag\\
& = & -\alpha \int_t^T e^{\alpha s} |\Delta Y_s^{m+1}|^2 ds  + 2 \int_t^Te^{\alpha s} (\Delta Y_s^{m+1})\cdot(f(s,Y_s^{m},Z_s^{m}) - f(s,Y_s^{m-1},Z_s^{m-1}))ds\notag\\
&& -2  \int_t^T e^{\alpha s} (\Delta Y_s^{m+1}) \cdot (\Delta Z_s^{m+1}) d B_s
+ \int_t^T e^{\alpha s} (\Delta Y_s^{m+1}) D_\mu H([Y_s^{m+1}])(Y_s^{m+1}) d K_s^{m+1}\notag\\
&&  - \int_t^T e^{\alpha s} (\Delta Y_s^{m+1}) D_\mu H([Y_s^{m}])(Y_s^{m}) d K_s^{m}.\label{eq:dvpIto}
\end{eqnarray}
Then, using Young's inequality and arguing as in \eqref{eq:l-concplusskoro} it can be deduced that for a suitable choice of $\alpha$ that
\begin{eqnarray}\label{eq:prep:cauchy}
\e\left[e^{\alpha t} |\Delta Y_t^{m+1}|^2 +  \int_t^T e^{\alpha s} |\Delta Z_s^{m+1}|^2 ds\right] &\leq &\frac{1}{4} \e \left[ \int_t^T e^{\alpha s} \left(|\Delta Y_t^{m}|^2 +|\Delta Z_s^{m}|^2\right) ds \right].
\end{eqnarray}

Let us now \phrase. We get that
 \begin{eqnarray*}
e^{\alpha t} |Y_t^m-\TL|^2 +  \int_t^T  e^{\alpha s} |Z_s^m|^2 ds &=& e^{\alpha T}|\xi-\TL|^2 +  2\int_t^Te^{\alpha s}  (Y_s^m-\TL) \cdot  f(s,Y_s^{m-1},Z_{s}^{m-1})  ds\\
&&-\alpha \int_t^T e^{\alpha s} |Y_s^m-\TL|^2 ds-2\int_t^T e^{\alpha s}(Y_s^m-\TL) \cdot  Z_s^m d B_s \\
&&    + 2\int_t^Te^{\alpha s}  (Y_s^m-\TL)\cdot  D_\mu H([Y_s^m])(Y_s^m) d K_s^m.
\end{eqnarray*}
By using the Lipschitz property of $f$ together with Young's inequality and the Skorokhod condition, we derive, for $\alpha$ large enough:

\begin{eqnarray}\label{eq:Kmisbounded}
&&\sup_{t\leq T} \bar \e \left[ e^{\alpha t}|Y_t^m-\TL|^2\right] + \e\left[ \int_0^T e^{\alpha s} |Z_s^m|^2 ds \right] + e^{\alpha T}H(\tl)(K_T^m-K_t^m) \notag \\
&\leq &   C\left(\bar \e \left[ |\xi-\TL|^2 \right] + \bar \E\left[\int_0^T |f(s,\TL,0) |^2 ds\right]\right),
\end{eqnarray}
so that $K^m_T$ is uniformly bounded. 

Finally, coming back to \eqref{eq:dvpIto}, we can argue as in \eqref{eq:l-concplusskoro} \eqref{eq:bsdegen1} to deduce from \eqref{eq:prep:cauchy} and \eqref{eq:Kmisbounded} that $\{Y_t^m,Z_t^m\}$ is a converging sequence in $\mathcal{S}^2 \times {\rm H}^2$ that converges to some limit $\{Y,Z\}$. {\pee We obtain that $K^m$ converges to some process $K$ by using the same scheme as we did at the end of \emph{step 2}. Again, it follows from standard computations that the limit  $\{Y,Z,K\}$ is a solution \eqref{eq:mainb}.  

}
\end{proof} 

{\color{black}

\begin{proof}[Proof of {\color{black} Proposition \ref{lemmemomenY}}]
Applying It\^o's formula to $e^{\alpha t}|Y_t|^p$ for some $p\geq 2$ we have
\begin{eqnarray*}
&&e^{\alpha t} |Y_t|^p +   \frac{p(p-1)}{2}\int_t^T  e^{\alpha s} |Y_s|^{p-2}  |Z_s|^2 ds \\
&= &   | \xi|^p +  p \int_t^T e^{\alpha s}  |Y_s|^{p-2}(Y_s) \cdot  f(s,Y_s,Z_s) ds -p \int_t^T  e^{\alpha s} |Y_s|^{p-2}(Y_s) \cdot  Z_s d B_s\\
&&     +p \int_t^T e^{\alpha s}  |Y_s|^{p-2}(Y_s) \cdot D_\mu H([Y_s])(Y_s) d K_s - \alpha \int_t^T e^{\alpha s }|Y_s|^p ds.\\
\end{eqnarray*}
{\pee On the one hand,} Lipschitz continuity assumption on $f$ together with Young's inequality yields
\begin{eqnarray*}
&& \int_t^T e^{\alpha s} |Y_s|^{p-2}(Y_s)\cdot f(s,Y_s,Z_s) ds\notag\\ 
 &\leq& {\pee C_{p} }\left\{ {\pee \int_t^T  |f(s,0,0)|^p ds } +  \int_t^T \left(1+\frac{1}{\epsilon} \right) e^{\alpha s}|Y_s|^{p}  + \epsilon e^{\alpha s}|Y_s|^{p-2}  |Z_s|^2ds\right\},
\end{eqnarray*}
for any $\epsilon>0$. {\pee On the other hand, still from Young's inequality we have
\begin{eqnarray*}
\sup_{t\leq T}\int_t^T e^{\alpha s}|Y_s|^{p-2}(Y_s) \cdot D_\mu H([Y_s])(Y_s) d K_s
&\leq & \int_0^T  e^{\alpha s} \frac{p}{p-1} |Y_s|^{p}+ \frac{1}{p}|D_\mu H([Y_s])(Y_s)|^p d K_s.
\end{eqnarray*}
Choosing $\epsilon$ small enough and {\pee then} $\alpha$ large enough gives
\begin{eqnarray}\label{uneqinterdepluspourmomentp}
& & e^{\alpha t} |Y_t|^p +   \frac{p(p-1)}{{\pee 4}} \int_t^T  e^{\alpha s} |Y_s|^{p-2} |Z_s|^2 ds\notag\\
& \leq &  | \xi|^p+   C_p \bigg\{ \int_t^T e^{\alpha s}  |Y_s|^{p} + |D_\mu H([Y_s])(Y_s)|^p d K_s+ {\pee \int_t^T  |f(s,0,0)|^p ds }\bigg\}\\\notag
&&-p\int_t^T  e^{\alpha s} |Y_s|^{p-2}(Y_s) \cdot \ Z_s d B_s.
\end{eqnarray}
Taking the expectation in the above inequality and using Gronwall's lemma applied to the continuous maps $s\mapsto K_s$ (see Lemma 4 in \cite{J64} or Theorem 17.1 in \cite{BS13}) eventually lead to
\begin{eqnarray*}
& &\sup_{t\leq T}\E\left[ e^{\alpha t} |Y_t|^p +   \frac{p(p-1)}{{\pee 4}} \int_t^T  e^{\alpha s} |Y_s|^{p-2} |Z_s|^2 ds\right]\notag\\
& \leq &C_{p,T} \bigg\{ \E[ | \xi|^p]+   \H_pK_T+  \E\left[\int_0^T |f(s,0,0)|^p ds\right]  \bigg\}.
\end{eqnarray*}
Coming back to \eqref{uneqinterdepluspourmomentp}, taking the supremum in time, the expectation and using then BDG' inequality and Young's inequality yield, together with the above estimate, to the result.
}

\end{proof}

}

\subsection{Interacting particle system constrained in mean field}\label{SEC_MF_FOR_BSDE}

Let us consider the following Skorokhod problem in mean field:
\begin{equation}\label{eq:backward:part}\left \lbrace\begin{array}{ll}
\displaystyle Y_t^i = \tilde \xi^i  + \int_t^T f(s,Y_s^i,Z_s^{i,i}) ds - \int_t^T \sum_{j=1}^N Z_s^{i,j} d B_s^j + \int_t^T D_\mu H(\mu_s^N)(Y_s^i) d K_s^N,\\
\displaystyle \forall t \in {\yin [0,T]}:\quad \mu_t^N = \frac{1}{N} \sum_{i=1}^N \delta_{Y_t^i},\quad H(\mu_t^N) \geq 0,\quad \int_0^T H(\mu_s^N) d K_s^N =0,\quad 1 \leq i \leq N,
\end{array}\right.
\end{equation}
where for each $i,j,$ $Z_s^{i,j}$ is a $n\times d$ matrix, $\{B^i\}_{1\leq i \leq N}$ are $N$ independent $d$-dimensional Brownian motions and $K^N$ is a continuous non decreasing process. The terminal conditions $\{\tilde \xi^i\}_{1 \leq i \leq N}$ are $\mathcal{F}_T^i$-measurable ($\{\mathcal{F}_t^i\}$ being the augmented natural filtration of $B^i$) independent r.v. having second order moment and satisfying $H\left( N^{-1} \sum_{i=1}^N \delta_{\tilde \xi^i}\right) \geq 0$.

\begin{rem}\quad
Making the analogy with classical mean field approximation of McKean-Vlasov processes, a natural choice for the family of terminal conditions $\{\tilde \xi^i\}$ is the family $\{\xi^i\}$ of copies of the r.v. $\xi$ in \eqref{eq:mainb}. Nevertheless, the condition $H([\xi]) \geq 0$ is not sufficient to ensure that \eqref{eq:backward:part} is indeed a solution of the Skorokhod problem in mean field since it does not imply $H\left( N^{-1} \sum_{i=1}^N \delta_{\xi^i}\right) \geq 0$. To overcome the problem, we may work only on the set $\Omega_N = \left\{H\left( N^{-1} \sum_{i=1}^N \delta_{\xi^i}\right) \geq -\eta_N\right\}$ which becomes, for a suitable choice of $\eta_N$, of full measure when $N\to + \infty$. This hence leads to obtain an asymptotic solution to the Skorokhod problem in mean field, but not for any $N$ {\color{black}(see the approach in paragraph \ref{subsec:A_mean_field_limit_cond})}. In this part, we decide to tackle the problem from a different point of view: we show how to construct, from the family of $\{\xi^i\}$, a family $\{\tilde \xi^i\}$ satisfying the Skorokhod condition and whose empirical measure tends to $[\xi]$, {\pee provided some additional integrability conditions on the data}. 

{\pe By ``some additional integrability conditions'', we mean that we assume fourth-order moment on the data $\xi$ and $\ld H$. In comparison with the assumptions needed in the forward case (see paragraph \ref{subsec:A_mean_field_limit_cond}), such conditions may appear to be the price to pay to solve the Skorokhod problem for each $N$. It seems to us that this is not the case, but something more related to the backward setting. For instance, still in comparison with the result obtained in the forward case, the set $\Omega_N$ is not independent of the $B^i$ and it is \emph{a priori} quite intricate to obtain higher order moment on the process $K^N$ without any additional assumptions on $\ld H$ (recalling the lack of integrability of $Z$).}
\end{rem}

\subsubsection*{Initialization, well posedness of the particle system constrained in mean field and estimates.} 
We first state the following Lemma which shows how to construct the $\{\tilde \xi^i\}$ from the $\{\xi^i\}$.

\begin{lemme}\label{lem:rateconvinitialcondition} Suppose that assumptions \A{A} hold. Given $N$ copies $\{\xi^i\}$ of $\xi$, there exists a family of random variable $\{\tilde \xi^i\}$ satisfying 
$$H\left(\frac{1}{N}\sum_{i=1}^N \delta_{\tilde \xi^i}\right) \geq 0,$$
and if we assume in addition that {\pe $\M_4([\xi]) + \H_4$ is finite} then there exists $C:= C(\A{A},\M_4([\xi]),\H_4) >0$ such that
\begin{eqnarray} 
\e\left[\frac{1}{N}\sum_{i=1}^N |\tilde \xi^i  - \xi^i|^2\right]  \leq  C\e^{1/2}\left[W_2^2(\bar \mu_T^N,[\xi])\right], \quad \bar \mu_T^N:=\frac{1}{N} \sum_{i=1}^N \delta_{\xi^i}.
\end{eqnarray}
\end{lemme}
\begin{proof}
To guarantee that the constraint is satisfied, the main idea consists in transporting the initial condition along the normal vector up to the set of constraint. Since we only assumed that the normal vector is non zero around the zero of $H$ we have to proceed in two steps : consider first the case where the empirical measure is not "too far" from the constraint set and when it is too far away.
{\pe
Let $\Omega_N = \{ H(\bar \mu_0^N) \geq \-\eta_N\}$ where $\eta_N \to 0$ and $\eta_N^{-2}\e\left[ W_2(\bar \mu_0^N,[\xi])\right] \to 0$ {\pee and let us push forward the empirical measure $\bar \mu_T^N$ along the flow  of  $d l_t^x =  \ld H(\mu_{l_t^x}^N)(l_t^x) dt,\, l_0^x = x,\, x \in \R^{nN}$, where $\mu_{l_t^x}^N=N^{-1}\sum_{i=1}^N \delta_{l_t^{i,x}}$ and $l_t^{i,x}$ denotes the $i^{{\rm th}}$ $n$-dimensional component of $l_t^x$. For $N$ large enough we have $-\eta_N \geq -\eta$ so that on $\Omega_N$ \eqref{eq:bilip_CV} is satisfied and it is therefore possible to find a positive $\bar{t} \leq \eta_N/\beta^2$ such that }
$H(\mu_T^N) \geq 0$ {\color{black} where $\mu_T^N := \big(\R^{nN} \ni x \mapsto l^x_{\bar t} \in \R^{nN}\big)  \sharp \bar \mu_T^N$}. We write $\{\zeta^i\}$ the family of associated r.v.


Let us now handle the case when we are in $\Omega_N^c$. From Remark \ref{point_inside_O} there exists $\tl$ such that $H(\tl) \geq 0$. We then set  $\tilde \xi^i = \zeta^i \mathbf{1}_{\Omega_N} + \TL \mathbf{1}_{\Omega_N^c}$. We have
\begin{eqnarray*}
\e\left[\frac{1}{N}\sum_{i=1}^N |\tilde \xi^i  - \xi^i|^2\right] \leq  M\frac{\eta_N^2}{\beta^4} + \E^{1/2}\left[\mathbf{1}_{\Omega_N^c}\right]\frac{1}{N}\sum_{i=1}^N \e^{1/2}\left[|\TL  - \xi^i|^4\right],
\end{eqnarray*}
thanks to Cauchy-Schwarz inequality. Since $H([\xi]) \geq 0$,
\begin{eqnarray*} 
\mathbb{P}(\Omega_N^c) \leq \mathbb{P}\left(|H(\bar \mu_T^N)-H([\xi])| \geq \eta_N \right) \leq \frac{\E\left[W_2^2(\bar \mu_T^N,[\xi])\right]}{\eta_N^2},
\end{eqnarray*}
we have, choosing $\eta_N^2 = \E^{1/2}\left[W_2^2(\bar \mu_T^N,[\xi])\right]$, that there exists a $C:=C\left(\A{A},\M_4([\xi]),\H_4\right)>0$ such that
\begin{eqnarray*} 
\e\left[\frac{1}{N}\sum_{i=1}^N |\tilde \xi^i  - \xi^i|^2\right]  \leq  C\e^{1/2}\left[W_2^2(\bar \mu_0^N,[\xi])\right].
\end{eqnarray*}
}
\end{proof}

Next, we have the following result.
\begin{prop}
Suppose that assumptions \A{A} hold and that the initial data $\{\tilde \xi^i\}$ in \eqref{eq:backward:part} are given by Lemma \ref{lem:rateconvinitialcondition}. Then,  the system \eqref{eq:backward:part}  is well posed.
\end{prop}
\begin{proof}
The proof of this result is not new and relies on classical results. Let us indeed remark that \eqref{eq:backward:part} reads as a classical reflected BSDE in $\R^{nN}$, with  normal reflexion in the constraint 
\begin{equation*}
{\mathcal K}=\left\{\bx = (x_1,\dots, x_N)\in (\R^n)^N, \; H\left(\frac{1}{N} \sum_{i=1}^N \delta_{x_i}\right)\geq 0\right\}. 
\end{equation*}
Indeed, let us recall that, if
\begin{equation*}
	H^N(\bx) = H\left(\frac{1}{N} \sum_{i=1}^N \delta_{x_i}\right),
\end{equation*}
 then (see \cite{CDLL})
\begin{equation*}
	\partial_{x_i} H^N(\bx) = \frac{1}{N}\, \ld H\left(\frac{1}{N} \sum_{i=1}^N \delta_{x_i}\right)(x_i).
\end{equation*}
Therefore the vector
$-\left(\ld H\left(\frac{1}{N} \sum_{i=1}^N \delta_{x_i}\right)(x_1),\dots, \ld H\left(\frac{1}{N} \sum_{i=1}^N \delta_{x_i}\right)(x_N)\right)$  is proportional to the outward normal to the set ${\mathcal K}$ at the point $(x_1,\dots, x_N)\in \partial {\mathcal K}$.
Existence and uniqueness of the solution are therefore immediate under our standing assumptions (cf. \cite{GPP:96}).\\ 
\end{proof}

\begin{lemme}[Moment estimates]\label{lem:momentestiKN:back}  
Under assumption \A{A}, for any $T>0$, there exists $C:=C(\A{A},T)>0$ such that
\begin{eqnarray*}
& & \E\left[\frac{1}{N} \sum_{i=1}^N|Y_t^i |^2  +   \frac{1}{N} \sum_{i=1}^N\int_0^T   \sum_{j=1}^N |Z_s^{i,j}|^2 ds\right]+\E[K_T^N]  \leq  C,
\end{eqnarray*}
and
\begin{equation*}
\sup_{N \geq 1} \e[(K_T^N)^2] \leq C.
\end{equation*}
\end{lemme}
\begin{proof}

{\pe Let $p\geq 2$, and let us \phrase\ (recalling that here, the filtration has also being enlarged to take into account the $N$ independent Brownian motions). Applying Itô's formula on $e^{\alpha t}|Y_t^i-\TL|^2$ we obtain, after summing over all the particles, that

 \begin{eqnarray*}
&&\frac{1}{N} \sum_{i=1}^N e^{\alpha t} |Y_t^i-\TL|^2 +   \frac{1}{N} \sum_{i=1}^N\int_t^T  e^{\alpha s}  \sum_{j=1}^N |Z_s^{i,j}|^2 ds \\
&= & {\color{black} \frac{1}{N} \sum_{i=1}^N  |\tilde \xi^i-\TL|^2} +  \frac{2}{N} \sum_{i=1}^N\int_t^T e^{\alpha s}  (Y_s^i-\TL) \cdot  f(s,Y_s^i,Z_s^{i,i}) ds\\
&&   -\frac{2}{N} \sum_{i=1}^N\int_t^T  e^{\alpha s} (Y_s^i-\TL) \cdot \sum_{j=1}^N Z_s^{i,j} d B_s^j  +\frac{2}{N} \sum_{i=1}^N \int_t^T e^{\alpha s} (Y_s^i-\TL) \cdot D_\mu H(\mu_s^N)(Y_s^i) d K_s^N\\
&& - \frac{\alpha}{N} \sum_{i=1}^N \int_t^T e^{\alpha s }|Y_s^i-\TL|^2 ds.\\
\end{eqnarray*}
From $L$-concavity and Skorokhod condition we have,
\begin{eqnarray}\label{eq:kpourpartconc1}
{\pee \tilde \E}\left[\frac{1}{N} \sum_{i=1}^N  \int_t^T e^{\alpha s} (Y_s^i -\TL)\cdot \ld H(\mu_s^N)(Y_s^i) dK_s^N\right] \leq -\frac{1}{N} \sum_{i=1}^N \int_t^T e^{\alpha s}H(\tl)dK_s^N,
\end{eqnarray}
while Lipschitz continuity assumption on $f$ together with Young's inequality yield\\
{\pee
\begin{eqnarray}
&& {\yin 2}\int_t^T e^{\alpha s} (Y_s^i -\TL)\cdot f(s,Y_s^i,Z_s^{i,i}) ds\notag\\ 
 &\leq& C \left\{ \int_t^T{\yin e^{\alpha s}} |f(s,\TL,0)|^2 ds + \int_t^T \left(1+\frac{1}{\epsilon} \right) e^{\alpha s}|Y_s^i -\TL|^{2}  + \epsilon e^{\alpha s}  |Z_s^{i,i}|^2ds\right\},\label{eq:estifpourK1}
\end{eqnarray}
for any $\epsilon>0$. Also, thanks to the construction in Lemma \ref{lem:rateconvinitialcondition} we have
\begin{eqnarray}
\frac 1N \sum_{i=1}^N |\tilde \xi^i - \TL|^2 \leq   |\xi^i|^2 + |\TL|^2 + \bar t M^2.
\end{eqnarray}
}

Choosing $\epsilon$ small enough and $\alpha$ large enough hence {\color{black} gives}
\begin{eqnarray}\label{esti:inter:pivot}
& &\frac{1}{N} \sum_{i=1}^N e^{\alpha t} {\pee \tilde \E}[|Y_t^i - \TL|^2] +   \frac{1}{{\yin 2}N} \sum_{i=1}^N\int_t^T  e^{\alpha s} \sum_{j=1}^N |Z_s^{i,j}|^2 ds+ 2H(\tl)(K_T^N-K_t^N) \notag\\
& \leq & \frac{e^{\alpha T}}{N} \sum_{i=1}^N {\pee \tilde \E}[ |\tilde \xi^i-\TL|^2]+ C{\pee \int_t^T {\yin e^{\alpha s}}{\pee \tilde \E}[|f(s,\TL,0)|^2] ds} -\frac{2}{N} \sum_{i=1}^N\int_t^T  e^{\alpha s} (Y_s^i-\TL) \cdot \sum_{j=1}^N Z_s^{i,j} d B_s^j.
\end{eqnarray}
Taking expectation on both side of the above estimate {\pee we can deduce} the first estimate. 
Let us now prove the second assertion. We come back to \eqref{ressertpourmoments} and obtain
\begin{eqnarray}
&& {\color{black} H( \mu^N_t)} - H( \mu^N_0) -  \frac{1}{N}\sum_{i=1}^N \int_0^t  D_\mu H( \mu^N_u)(Y_u^i)\cdot f(u,Y_u^i,Z_u^{i,i}) du\notag \\
&&  + \frac{1}{N}\sum_{i=1}^N \int_0^t  D_\mu H( \mu^N_u)( Y_u^i)\cdot  \sum_{j=1}^N Z_u^{i,j} dB^{{\yin j}}_u  \geq \frac{1}{N}\sum_{i=1}^N \int_0^t \left| D_\mu H ( \mu^N_u)( Y_u^i)\right|^2 dK_u^N.
\end{eqnarray}
Note that  $dK^N(\{s,\ s.t.\, H(\mu_s^N)>0\})=0$ so that by {\yin \eqref{eq:bilip_CV}} we obtain
\begin{eqnarray}
&&\beta^2 K_t^N \leq {\color{black} H( \mu^N_t)}- H( \mu^N_0) -  \frac{1}{N}\sum_{i=1}^N \int_0^t  D_\mu H( \mu^N_u)(Y_u^i)\cdot f(u,Y_u^i,Z_u^{i,i}) du\notag \\
&&  + \frac{1}{N}\sum_{i=1}^N \int_0^t  D_\mu H( \mu^N_u)( Y_u^i)\cdot  \sum_{j=1}^N Z_u^{i,j} dB^{{\yin j}}_u{\yin .} \label{interesti2}
\end{eqnarray}
{\color{black} On the one hand}
{\pee
\begin{eqnarray*}
&&\E\left[\sup_{t\leq T}\left(\frac{1}{N}\sum_{i=1}^N \int_0^t  D_\mu H( \mu^N_u)( Y_u^i)\cdot  \sum_{j=1}^N Z_u^{i,j} dB^{{\yin j}}_u\right)^2\right] \\
&\leq& \E\left[\frac{4}{N^2}\int_0^T\sum_{j=1}^N \left| \sum_{i=1}^N  |D_\mu H( \mu^N_u)( Y_u^i)|   |Z_u^{i,j}| \right|^2du\right]\\
&\leq& \E\left[\frac{4}{N^2}\int_0^T\sum_{j=1}^N \sum_{i=1}^N  \left\{|D_\mu H( \mu^N_u)( Y_u^i)|^2\right\}  \sum_{i=1}^N  \left\{ |Z_u^{i,j}|^2\right\} du\right]\\
&\leq& 4\H_2\E\left[\frac{1}{N}\sum_{i=1}^N\int_0^T \sum_{j=1}^N |Z_u^{i,j}|^2 du\right].
\end{eqnarray*}}
On the other hand, {\color{black} using the Lipschitz continuity} of $f$ and Cauchy-Schwarz{\yin 's} inequality,\\
{\pee
\begin{eqnarray*}
&&\left(\int_0^T  \frac{1}{N}\sum_{i=1}^N  D_\mu H( \mu^N_u)(Y_u^i)\cdot f(s,Y_s^i,Z_s^{i,i}) du\right)^2\\
&\leq& C(\H_2)\left\{\int_0^T |f(s,0,0)|^2ds + \int_0^T \frac{1}{N}\sum_{i=1}^N |Y_s^i|^2  + \frac{1}{N}\sum_{i=1}^N |Z_s^{i,i}|^2 ds\right\}.
\end{eqnarray*}

Hence, choosing $t=T$} 
, then squaring the expression and eventually taking the expectation lead, thanks to the two above estimates as well as the first assertion of this Lemma (using also \eqref{eq:DFL_CV}), to the result.
}
\end{proof}

\subsubsection*{The mean field limit} {\color{black}
Let $\{B^i\}_{1\leq i \leq N}$ be $N$ independent $d$-dimensional Brownian motions, $\{\xi^i\}_{1\leq i \leq N}$ be $N$ independent copies of $\xi$ and set $\{(\bar{Y}^i,\bar{Z}^i,K)\}_{1 \leq i \leq N}$ the solution of \eqref{eq:mainb} with data $\{(B^i,\xi^i)\}_{1\leq i \leq N}$ (note that $K$ does not depend on $i$ since it only depends on the law of the data, which are identically distributed thanks to Remark \ref{rem:uniquelaw}). Let $\{(Y^i,Z^i,K^N)\}_{1 \leq i \leq N}$ be the solution of \eqref{eq:backward:part} where the terminal data $\{\tilde \xi^i\}$ therein are built from the  $\{\xi^i\}_{1\leq i \leq N}$ thanks to Lemma \ref{lem:rateconvinitialcondition}. We aim at showing that the solution $\{(Y^i,Z^i,K^N)$ {\yin is close} to $\{(\bar{Y}^i,\bar{Z}^i,K)\}$.

Let us denote by $\Delta Y^i := Y^i - \bar Y^i$ and $\Delta Z^{i,j} := Z^{i,j} - \bar Z^{i,j}$ where $\bar Z^{i,j} = 0_{\R^{n\times d}}$ when $i\neq j$. We have:
}
\begin{thm}\label{th:estiparticleBSDE} 
Suppose that assumptions \A{A} hold. Then, there exist $C_T:=C(T,\A{A}) >0 $ such that
\begin{eqnarray*}
&&\e \left[\sup_{t \leq T} \frac{1}{N} \sum_{i=1}^N\Bigg\{ |\Delta Y_t^i|^2 +  \int_t^T  \sum_{j=1}^N|\Delta Z_s^{i,j}|^2 ds \Bigg\} \right] \\
&\leq & C_T  \left\{\e^{1/2} \left[\sup_{0 \leq s \leq T} W_2^2(\mu_s,\bar \mu_s^N)\right] + \E \left[\frac 1N \sum_{i=1}^N|\tilde \xi^i - \xi^i|^2\right]\right\},\\
&&\E\left[\sup_{s\leq t}\Big\{ |\Delta Y_s^i|^2 + \int_s^T  \sum_{j=1}^N|\Delta Z_u^{i,j}|^2 du \Big\}\right]\\
&\leq &  C_T \left\{\e^{1/4} \left[\sup_{t \leq T} W_2^2(\mu_s,\bar \mu_s^N)\right]+ \E^{1/2} \left[\frac 1N \sum_{i=1}^N|\tilde \xi^i - \xi^i|^2\right]\right\}.
\end{eqnarray*}
\end{thm}
{\color{black}

{\pe{\pee
\begin{lemme}\label{cor:corratebsdepart2} 
Suppose that assumptions \A{A} hold. Assume further that there exists $q>4$ such that $\M_q([\xi])$ is finite as well as {\pee $\H_q$ and $\int_0^T \E\left[|f(s,0,0)|^q \right]ds$}. Assume moreover that the data of the system are such that there exists $p\geq 4$ and $C_{p,Z}>0$ such that $\sup_{t \leq T}\E[|Z_t|^p] \le C_{p,Z}$. Then, there exists $C_T:=C\big(T,\A{A},\M_q([\xi]),\H_q,C_{p,Z}\big)>0 $ such that
\begin{equation*}
\left(\E \left[\frac 1N \sum_{i=1}^N|\tilde \xi^i - \xi^i|^2\right]\right)^2 + \e \left[\sup_{0 \leq s \leq T} W_2^2(\mu_s,\bar \mu_s^N)\right]\leq C_T \epsilon_N,
\end{equation*}
where $\epsilon_N$ is given in Lemma \ref{ConvPart} (see Remark \ref{remconvpart}).
\end{lemme}

\begin{rem}
The assumption on the integrability of $Z$ could seems strange at first sight, especially because this process is part of the solution and not an input of the problem. This assumption relies on the control done for the discretization of the process $Y$ to handle  the convergence of the supremum in time of the Wasserstein distance. For instance, such property is  {\color{black} satisfied} when the terminal condition $\xi$ and the driver $f$ are Malliavin differentiable (see also Lemma 4.4 in \cite{BH17} for further details).
\end{rem}

\begin{proof}[Proof of Lemma \ref{cor:corratebsdepart2}] Let us first emphasize that the result is straightforward for the first term in the l.h.s. of the above estimate: it follows from Lemma \ref{lem:rateconvinitialcondition} and Theorem 2 in \cite{FG15} (see also Theorem 5.8 and Remark 5.9 in \cite{CD17-1}). 

Let us now deal with the second term in the above l.h.s. To handle this part, we have to check that assumptions in Lemma \ref{ConvPart} are satisfied. Note first that, under our considered assumption, one can apply  {\color{black} Proposition \ref{lemmemomenY}}. It thus remains to check that \eqref{hypRacRu} hold after taking into account Remark \ref{remconvpart}. These are straightforward consequences of our standing assumptions provided the process $K$ is Lipschitz continuous. Let us now explain why this is latter fact is true.

Let $(Y,Z)$ be a given (part of) the unique solution of \eqref{eq:mainb}. It then holds that $\sup_{t \leq T} \E[|f(t,Y_t,Z_t)|^2] < + \infty$ so that one can come back to \emph{Step 1} of the proof of Theorem \ref{thm:existence_uniqueness} and build the process $K$ in such a way it is Lipschitz. By uniqueness, this proves that $K$ is Lipschitz. The result then follows.
\end{proof}
}

\begin{proof}[Proof of Theorem \ref{th:estiparticleBSDE}]
For $\alpha$ large enoug{\yin h} we have

\begin{eqnarray}\label{eq:itodeltaparticle}
&&e^{\alpha t} |\Delta Y_t^i|^2 + \frac{1}{2} \int_t^T e^{\alpha s} \sum_{j=1}^N|\Delta Z_s^{i,j}|^2 ds \\\notag
&\leq &  e^{\alpha T} |\tilde \xi^i  - \xi^i|^2  + 2 \int_t^T e^{\alpha s}\Delta Y_s^i \cdot D_\mu H(\mu_s^N)(Y_s^i) dK^N_s \\
&&- 2 \int_t^T e^{\alpha s} \Delta Y_s^i \cdot D_\mu H(\mu_s)(\bar Y_s^i) dK_s - 2 \int_t^T e^{\alpha s} \Delta Y_s^i \cdot \sum_{j=1}^N\Delta Z_s^{i,j} dB_s^j.\notag
\end{eqnarray}
So that, summing over $i$ leads to 

\begin{eqnarray}\label{eq:itodeltaparticle2}
&&\frac{1}{N}\sum_{i=1}^N \left\{ e^{\alpha t} |\Delta Y_t^i|^2 + \frac{1}{2} \int_t^T e^{\alpha s} \sum_{j=1}^N |\Delta Z_s^{i,j}|^2 ds \right\}\notag \\
&\leq &   \frac{{\yin e^{\alpha T}}}{N}\sum_{i=1}^N |\tilde \xi^i  - \xi^i|^2  + 2 \int_t^T e^{\alpha s}  \frac{1}{N}\sum_{i=1}^N \left\{ \Delta Y_s^i \cdot D_\mu H(\mu_s^N)(Y_s^i)\right\} dK^N_s \notag \\
&&- 2\int_t^T e^{\alpha s}  \frac{1}{N}\sum_{i=1}^N \left\{\Delta Y_s^i \cdot D_\mu H(\mu_s)(\bar Y_s^i) \right\} dK_s - 2 \frac{1}{N}\sum_{i=1}^N\int_t^T e^{\alpha s} \Delta Y_s^i \cdot \sum_{j=1}^N\Delta Z_s^{i,j} dB_s^j\notag\\
&=:& J_1^N + J_2^N(t,T) + J_3^N(t,T) + M_N(t,T). 
\end{eqnarray}

We first deal with  the second term in the r.h.s. of \eqref{eq:itodeltaparticle2}. We have
\begin{eqnarray}
J_2^N(t,T) = \int_t^Te^{\alpha s} \frac{1}{N} \sum_{i=1}^N \Delta Y_s^i \cdot D_\mu H(\mu_s^N)(Y_s^i) dK^N_s &\leq& \int_t^T e^{\alpha s} \big(H(\mu_s^N)-H(\bar \mu_s^N)\big) dK^N_s\label{eq:particles_back}\\
&\leq &  \int_t^T e^{\alpha s} \big(H(\mu_s)-H(\bar \mu_s^N)\big) dK^N_s\notag
\end{eqnarray}
from Skorokhod condition and since $H(\mu_s) \geq 0$ for all $0\leq s \leq T$. Thus,
\begin{eqnarray}
J_2^N(t,T) \leq  C e^{\alpha T}\sup_{t \leq s \leq T} W_2(\mu_s,\bar \mu_s^N) (K^N_T-K_t^N).\notag
\end{eqnarray}

Next we handle the third term in the right hand side of \eqref{eq:itodeltaparticle2} and split it into two parts
\begin{eqnarray*}
J_3^N(t,T) &=& - 2 \int_t^T e^{\alpha s} \frac{1}{N} \sum_{i=1}^N \Delta Y_s^i \cdot D_\mu H(\mu_s)(\bar Y_s^i) dK_s\\
&=& - 2 \int_t^T e^{\alpha s} \frac{1}{N} \sum_{i=1}^N \Delta Y_s^i \cdot \big(D_\mu H(\mu_s)(\bar Y_s^i)  - \big(D_\mu H(\bar \mu_s^N)(\bar Y_s^i)\big)dK_s \\
&& \quad - 2\int_t^T e^{\alpha s} \frac{1}{N} \sum_{i=1}^N \Delta Y_s^i \cdot D_\mu H(\bar \mu_s^N)(\bar Y_s^i)dK_s =: I_1^N(t,T) + I_2^N(t,T).
\end{eqnarray*}
On the one hand we have from Cauchy-Schwarz{\yin 's} inequality, \eqref{eq:UC} and Young{\yin 's} inequality{\yin 
\begin{eqnarray*}
I_1^N(t,T) &\leq &C\int_t^T e^{\alpha s} \bigg(\frac{1}{N} \sum_{i=1}^N |\Delta Y_s^i|^2\bigg)^{1/2} W_2(\bar \mu_s^N,\mu_s) d K_s\\
& \leq & C \int_t^T  \bigg(\frac{1}{N} \sum_{i=1}^N e^{\alpha s} |\Delta Y_s^i|^2\bigg)^{1/2} W_2(\bar \mu_s^N,\mu_s)  d K_s\\
& \leq & C \int_t^T  \left\{\frac{1}{N} \sum_{i=1}^N e^{\alpha s} |\Delta Y_s^i|^2  +  W_2^2(\bar \mu_s^N,\mu_s)\right\} d K_s\\
& \leq & C\int_t^T  \frac{1}{N} \sum_{i=1}^N e^{\alpha s} |\Delta Y_s^i|^2 dK_s  + C \sup_{t\leq s \leq T} W_2^2(\bar \mu_s^N,\mu_s)\, K_T.
\end{eqnarray*}
}
On the other hand we have from $L$-concavity and since for all $0\leq s \leq T$, $H(\mu_s^N) \geq 0$ and  $H(\mu_s) = 0$ $d K_s$-a.e.
\begin{eqnarray*}
I_2^N(t,T) &\leq &2\int_t^T e^{\alpha s} \big(H(\bar \mu_s^N) - H(\mu_s^N) \big) d K_s\\
& \leq & 2\int_t^T e^{\alpha s} \big(H(\bar \mu_s^N) - H(\mu_s) \big) d K_s\\
& \leq & 2 e^{\alpha T} \sup_{t \leq s \leq T}W_2(\bar \mu_s^N,\mu_s) K_T.
\end{eqnarray*}

Hence,
\begin{eqnarray*}
J_3^N(t,T)  & \leq & C \int_t^T  \frac{1}{N} \sum_{i=1}^N e^{\alpha s} |\Delta Y_s^i|^2 dK_s + C \sup_{t\leq s \leq T} \left\{W_2^2(\bar \mu_s^N,\mu_s) +  W_2(\bar \mu_s^N,\mu_s)\right\}.
\end{eqnarray*}

Bringing together the estimates on $J_i^N$, $1 \leq i \leq 3$ in \eqref{eq:itodeltaparticle2}, we obtain
\begin{eqnarray}\label{eq:estiparticlebackinter}
&& \frac{1}{N} \sum_{i=1}^N\Bigg\{e^{\alpha t} |\Delta Y_t^i|^2 + \frac{1}{2} \int_t^T e^{\alpha s} \sum_{j=1}^N|\Delta Z_s^{i,j}|^2 ds \Bigg\}\\
 &\leq &\frac{e^{\alpha T}}{N}\sum_{i=1}^N |\tilde \xi^i  - \xi^i|^2  + C e^{\alpha T}\sup_{t \leq s \leq T} W_2(\mu_s,\bar \mu_s^N) (K^N_T-K_t^N)\notag\\
 &&+C \int_t^T  \frac{1}{N} \sum_{i=1}^N e^{\alpha s} |\Delta Y_s^i|^2 dK_s + C \sup_{t\leq s \leq T} \left\{W_2^2(\bar \mu_s^N,\mu_s) +  W_2(\bar \mu_s^N,\mu_s)\right\} + M_N(t,T).\notag
\end{eqnarray}
Hence, taking the expectation in \eqref{eq:estiparticlebackinter}, using Cauchy-Schwarz's inequality, Gronwall's lemma and then using Lemmas \ref{lem:rateconvinitialcondition} and \ref{lem:momentestiKN:back}  lead to
\begin{eqnarray}\label{firststeppart}
&& \e\left[\frac{1}{N} \sum_{i=1}^N\Bigg\{e^{\alpha t} |\Delta Y_t^i|^2 + \int_t^T e^{\alpha s} \sum_{j=1}^N|\Delta Z_s^{i,j}|^2 ds \Bigg\} \right]\\
 &\leq &C_T \Bigg\{{\pe \e\left[\frac{1}{N}\sum_{i=1}^N |\tilde \xi^i  - \xi^i|^2\right]} + \e^{1/2} \left[\sup_{t \leq s \leq T} W_2^2(\mu_s,\bar \mu_s^N)\right]+ \E\left[ \sup_{t\leq s \leq T} \left\{W_2^2(\bar \mu_s^N,\mu_s) +  W_2(\bar \mu_s^N,\mu_s)\right\} \right]\Bigg\}\notag.
\end{eqnarray}

{\color{black}
Now we have from BDG and Young inequality
\begin{eqnarray}
\E[\sup_{t \leq T} M_{N}(t)] &\leq & C\E\left[\left(\frac{1}{N^2}\int_t^T\sum_{j=1}^N  e^{2\alpha s}\left( \sum_{i=1}^N {\pe |}\Delta Y_s^i{\pe |} |\Delta Z_s^{i,j}|\right)^2 ds\right)^{1/2}\right]\notag\\
&\leq& C\E\left[ \left(\sup_{t\leq T}\left\{\frac{1}{N}\sum_{i=1}^N\left\{e^{\alpha t} {\pe |}\Delta Y_s^i{\pe |}^{2}\right\}\right\} \frac{1}{N}\sum_{i=1}^N\left\{\int_t^T  e^{\alpha s} \sum_{j=1}^N |\Delta Z_s^{i,j}|^2 ds\right\}\right)^{1/2}\right]\notag\\
&\leq &C\varepsilon \E\left[\sup_{t\leq T} \frac{1}{N} \sum_{i=1}^N \left\{e^{\alpha t} {\pe |}\Delta Y_s^i{\pe |}^{2}\right\}\right] +  \frac {C^2}{4\varepsilon} \frac{1}{N} \sum_{i=1}^N \left\{ \E\left[\int_t^T  e^{\alpha s} \sum_{j=1}^N |\Delta Z_s^{i,j}|^2 ds\right]\right\},\label{estimarting}
\end{eqnarray}
for any $\varepsilon >0$. Hence, coming back to \eqref{eq:estiparticlebackinter} and taking first the supremum in time and then expectation we deduce  from the above estimate, Lemma \ref{lem:momentestiKN:back} and Gronwall's lemma that
\begin{eqnarray}
&& \e\left[\sup_{t\leq T}\frac{1}{N} \sum_{i=1}^N\Bigg\{e^{\alpha t} |\Delta Y_t^i|^2 + \frac{1}{2} \int_0^T e^{\alpha s} \sum_{j=1}^N|\Delta Z_s^{i,j}|^2 ds \Bigg\} \right]\label{encore_une_de_plus}\\
 &\leq &C_T \Bigg\{{\pe \e\left[\frac{1}{N}\sum_{i=1}^N |\tilde \xi^i  - \xi^i|^2\right]}+ \e^{1/2} \left[\sup_{t \leq T} W_2^2(\mu_s,\bar \mu_s^N)\right]+ \E\left[ \sup_{t\leq s \leq T} \left\{W_2^2(\bar \mu_s^N,\mu_s) +  W_2(\bar \mu_s^N,\mu_s)\right\} \right] \Bigg\}. \notag
\end{eqnarray}
Note that we also have from above results, thanks to the exchangeability of the $(Y^i,Z^i)$, that 
\begin{equation}
\e\left[e^{\alpha t} |\Delta Y_t^i|^2\right] + \frac{1}{2}\E\left[ \int_t^T e^{\alpha s}\sum_{j=1}^N|\Delta Z_s^{i,j}|^2 ds \right] \leq C_T\left\{{\pe \e\left[\frac{1}{N}\sum_{i=1}^N |\tilde \xi^i  - \xi^i|^2\right]} + \e^{1/2} \left[\sup_{0 \leq t \leq T} W_2^2(\mu_t,\bar \mu_t^N)\right]\right\}.
\end{equation}
Coming back to \eqref{eq:itodeltaparticle}, taking first the supremum then the expectation we get, thanks to BDG's inequality and \eqref{estimarting}, that

\begin{eqnarray}\label{eq:itodeltaparticle33}
&&\E\left[\sup_{t\leq T}\Big\{e^{\alpha t} |\Delta Y_t^i|^2 + \frac{1}{2} \int_t^T e^{\alpha s} \sum_{j=1}^N|\Delta Z_s^{i,j}|^2 du \Big\}\right]\\
&\leq &  \E\left[ |\tilde \xi^i  - \xi^i|^2\right]  + 2 \E\left[\int_0^T e^{\alpha s}|\Delta Y_s^i| | \ld H(\mu_s^N)(Y_s^i)|(dK^N_s + dK_s)\right] + \e^{1/2} \left[\sup_{t \leq T} W_2^2(\mu_t,\bar \mu_t^N)\right].\notag 
\end{eqnarray}
Since by exchangeability of the $(Y^i,Z^i)$, Cauchy-Schwarz's inequality and \eqref{boundHback}
\begin{eqnarray}\label{eq:itodeltaparticle343}
&&\E\left[\int_t^T e^{\alpha s}|\Delta Y_s^i| | \ld H(\mu_s^N)(Y_s^i)|(dK^N_s + dK_s)\right] \\&=& \frac{1}{N}\sum_{i=1}^N\E\left[\int_t^T e^{\alpha s}|\Delta Y_s^i| | \ld H(\mu_s^N)(Y_s^i)|(dK^N_s + dK_s)\right] \notag\\
&\leq& M^2 \E\left[\sup_{s \leq T}e^{\alpha s}\left(\frac{1}{N}\sum_{i=1}^N|\Delta Y_s^i|^2  \right)^{1/2} \big((K^N_T-K_t) + (K_T-K_t)\big)\right]\notag\\
&\leq& 2e^{\alpha T}M^2 \E^{1/2}\left[\sup_{s \leq T}\frac{1}{N}\sum_{i=1}^N|\Delta Y_s^i|^2 \right] \E^{1/2}\left[\big((K^N_T-K_t)^2 + (K_T-K_t)^2\big)\right],\notag
\end{eqnarray}
we obtain, thanks to estimate in  {\color{black} Proposition \ref{lemmemomenY}} and \eqref{encore_une_de_plus} that
\begin{eqnarray*}\label{eq:itodeltaparticle33czCD}
\E\left[\sup_{s\leq t}\Big\{e^{\alpha s} |\Delta Y_s^i|^2 + \frac{1}{2} \int_s^T e^{\alpha u} \sum_{j=1}^N|\Delta Z_u^{i,j}|^2 du \Big\}\right]&\leq & C\left\{ \e^{1/4} \left[\sup_{t \leq T} W_2^2(\mu_s,\bar \mu_s^N)\right]+{\pe \e^{1/2}\left[\frac{1}{N}\sum_{i=1}^N |\tilde \xi^i  - \xi^i|^2\right]}\right\} .
\end{eqnarray*}
}

\end{proof}

\subsection{Associated {\pe O}bstacle problem on {\color{black} the} Wasserstein space} 
As we did for the (forward) normally constrained in law SDE, we aim at connecting our backward system with a PDE. To do so, we consider the Markovian setup, {\color{black} where}  the final condition \eqref{eq:mainb} is of the form $\phi(X_T,[X_T])$, for some $\phi: \R^{n} \times \mathcal{P}^2(\R^n) \to \R$ and where $X$ is the solution of an SDE driven by Lipschitz coefficients. Namely we consider for $T>0$, for any r.v. $X_0$ having second order moment, independent of the Brownian motion and any $t$ in $[0,T]$, the Mean-Reflected Forward-Backward SDE (MR-FBSDE):
\begin{equation}\label{eq:MR-FBSDE}\left\lbrace\begin{array}{lll}
\displaystyle X_s^{t,X_0} = X_0 + \int_t^s b(r,X^{t,X_0}_r) \d r + \int_t^s \sigma(r,X^{t,X_0}_r) \d B_r,{\color{black}\ s \in [t,T],\quad X_s^{t,X_0} = X_0,\ s \in [0,t)} \\
\displaystyle Y_s^{t,X_0} = \phi(X_T^{t,X_0},[X_T^{t,X_0}]) + \int_s^T f(r,X_r^{t,X_0},Y_r^{t,X_0}) \d r - \int_s^T Z_r^{t,X_0} \d B_r\\
\displaystyle\quad \quad \quad \quad+ \int_s^T \ld H([Y_s^{t,X_0}])(Y_s^{t,X_0}) \d K_s^{t,X_0},\ s \in [t,T], \\
\displaystyle  \forall s \in [0,T],\quad H([Y_s^{t,X_0}]) \geq 0,\quad \int_0^T H([Y_s^{t,X_0}]) \d K_s^{t,[X_0]} = 0,\quad s \in [0,T],
\end{array}\right.
\end{equation}
where the superscript $(t,X_0)$ stands for the initial condition of the SDE associated with $X$ whose coefficients are supposed to be continuous in time and Lipschitz continuous in space (uniformly in time), where the coefficients of the backward component satisfy \A{A} (with $d=1$) and where $\phi : \R^n \times \mathcal{P}^2(\R^n) \to \R$ is a continuous and bounded function. Also, in this section we further assume that the following assumptions hold:
\begin{trivlist}
	\item[\A{A'1}] The functions $f: \rset_+\times \rset^n \times \rset \fl \rset $ is continuous and there exists a positive $C_f$ such that  for all $(s,x,y)$ in $[0,T] \times \R^n \times \R$ : $|f(s,x,y)| \leq C_f(1+|x| + |y|)$,
	\item[\A{A'2}] There exists $\beta>0$ such that for all $(x,\mu) \in \R^n \times \mathcal{P}^2(\R^n)$ we have $\beta \le \ld H(\mu)(x)$.
\end{trivlist}
In the following we say that assumptions \A{A'} are in force if assumptions \A{A'1}, \A{A'2} and \A{A} hold.\\

{\pe We start by recalling that, according to Remark \ref{rem:uniquelaw}, uniqueness in law holds for \eqref{eq:MR-FBSDE}.}
In this case, we aim at proving that there exists a decoupling field $u :[0,T] \times \R^n \times \mathcal P^2(\R^n)  \ni (t,x,\mu) \mapsto u(t,x,\mu) \in \R$ such that for any $t$ in $[0,T]$, any r.v. $X_0$ having second order moment and independent of the Brownian motion, we have for all $s$ in $[t,T]$ $u(s,X_s^{t,X_0},[X_s^{t,X_0}]) = Y_s^{t,X_0}$ a.s.. 

Especially we are going to prove that $u$ solves, in the viscosity sense, the following obstacle problem on the Wasserstein space:
\begin{equation}\label{eq:obstacle}\left\lbrace\begin{array}{lll}
\displaystyle \min  \bigg\{\big\{(\partial_t + {\pe\mathcal{L}})u(t,x,\mu) + f(t,x,u(t,x,\mu))\big\}; \\
\displaystyle \quad  H\big({\color{black} u(t,\cdot,\mu)\sharp \mu} \big)\bigg\}= 0,\quad \mbox{on}\quad [0,T)\times \R^n \times \mathcal{P}^2(\R^n),\\
u(T,\cdot,\cdot) = \phi,
\end{array}\right.
\end{equation}
{\color{black}
where $\mathcal{L}$ is given by, for all smooth $\varphi : [0,T]{\pe \times \R^n} \times \mathcal P^2(\R^n) \to \R${\pe
\begin{eqnarray*}
\mathcal{L}\varphi(t,x,\mu)&=&\frac12\int_{\R^n} {\rm Tr}\left[((\sigma\sigma^*)(t,y)) \partial_y \ld \varphi(\mu)(t,x,\mu)(y)\right]\mu(dy) +\int_{\R^n} \ld \varphi(\mu)(t,x,\mu)(y)\cdot b(t,y)\mu(dy)\\
&&+ \frac12{\rm Tr}\left[((\sigma\sigma^*)(t,x)) D_x^2 \varphi(t,x,\mu)\right] +D_x\varphi(t,x,\mu)\cdot b(t,x).
\end{eqnarray*}
}
Inspired from \cite{CD17-2}, we define a viscosity solution of \eqref{eq:obstacle} as follows.
\begin{df}\label{def:viscoObstacle}
We say that a continuous function $u : [0,T] \times \R^n \times \mathcal P^2(\R^n) \to \R$ is a viscosity solution of \eqref{eq:obstacle} if
\begin{enumerate}
\item[(i)] The function $u$ is jointly continuous {\color{black}and bounded};
\item[(ii)] for any $(t,x,\mu)$ in $[0,T]\times \R^n \times \mathcal{P}^2(\R^n)$, for any test functions  (see definition 11.18 of \cite{CD17-2} for the class of test functions) $\varphi : {\color{black}(0,T)}\times \R^n \times \mathcal{P}^2(\R^n) \to \R$ such that $u-\varphi$ have a {\color{black}global} minimum (resp. max) in $(t,x,\mu)$ we have
\begin{equation*}\left\lbrace\begin{array}{ll}
\displaystyle \min  \bigg\{\big\{(\partial_t + {\pe\mathcal{L}})\varphi(t,x,\mu) + f(t,x,\varphi(t,x,\mu))\big\}; \\
\displaystyle \quad  H\big({\color{black} u(t,\cdot,\mu)\sharp \mu} \big) \bigg\} \leq 0,\ (\rm{resp.} \geq 0),\quad [0,T)\times \R^n \times \mathcal{P}^2(\R^n);
\end{array}\right.
\end{equation*}
\item[(iii)] $\displaystyle u(T,x,\mu)= \phi(x,\mu)$ on $\R^{n} \times \mathcal{P}^2(\R^n)$.
\end{enumerate}
\end{df}

The result is the following :
\begin{thm}\label{THM:Obstacle} Suppose that assumption \A{A} holds. 
\begin{enumerate}
\item Then there exists a decoupling field $u:[0,T] \times \R^n \times \mathcal P^2(\R^n)  \ni (t,x,\mu) \mapsto u(t,x,\mu) \in \R$ such that for any $t$ in $[0,T]$, any r.v. $X_0$ having second order moment and independent of the Brownian motion, we have for all $s$ in $[t,T]$ $u(s,X_s^{t,X_0},[X_s^{t,X_0}]) = Y_s^{t,X_0}$ a.s..

\item Suppose in addition that assumption \A{A'} holds. Then, the mapping $u : [0,T] \times \R^n \times \mathcal P^2(\R^n)  \ni (t,x,\mu) \mapsto u(t,x,\mu) \in \R$ is a continuous function and solves,  in the sense of Definition \ref{def:viscoObstacle}, the obstacle problem on Wasserstein space \eqref{eq:obstacle}.

\end{enumerate}

\end{thm}

{\pee
\begin{rem}\label{rem:uniquevisoc}
In \eqref{eq:MR-FBSDE}, the driver $f$ does not depend on $Z$. This assumption again relies to the lack of comparison principle for FBSDE \eqref{eq:MR-FBSDE} (see \cite{BLP09}) which is required to tackle driver depending on the $Z$ argument when investigating {\color{black} the existence of a viscosity solution connected with FBSDEs}.
\end{rem}
}

\begin{proof} Let us first emphasize that strong uniqueness holds for \eqref{eq:MR-FBSDE} under our standing assumptions. Considering the equation with the reflection process $K^{t,X_0}$ as an entry, one can deduce from classical FBSDEs results that for any given initial $t$ in $[0,T]$ there exists a measurable function $v_{t,X_0}(t,\cdot): \R^n \to \R$, called a decoupling field, such that for any $s$ in $[t,T]$, $v_{t,X_0}(s,X_s^{t,X_0}) = Y_t^{t,X_0}$ a.s. (see also chapter 4 of \cite{CD17-1}). To make the connection between the classical decoupling field $v$ with a decoupling field  $u : [0,T] \times \R^n\times \mathcal{P}_2(\R^n) \to \R$, let us first introduce the following decoupled flow of \eqref{eq:MR-FBSDE}: for any $x \in \R^n$, we set on $[0,T]$:

\begin{equation}\label{eq:MR-FBSDE-decoupled}\left\lbrace\begin{array}{lll}
\displaystyle X_s^{t,x,X_0} = x + \int_t^s b(r,X^{t,x,X_0}_r) \d r + \int_t^s \sigma(r,X^{t,x,X_0}_r) \d B_r,\, s \geq t ,\text{ and } X_s^{t,x,X_0}=x,\, s <t,\\
\displaystyle Y_s^{t,x,X_0} = \phi(X_T^{t,x,X_0},[X_T^{t,X_0}]) + \int_s^T f(r,X_r^{t,x,X_0},Y_r^{t,x,X_0}) \d r - \int_s^T Z_r^{t,x,X_0} \d B_r\\
\displaystyle \quad + \int_s^T \ld H([Y_s^{t,X_0}])(Y_s^{t,x,X_0}) \d K_s^{t,[X_0]}.
\end{array}\right.
\end{equation}
Note that this equation is not of a MR-FBSDE and not of McKean-Vlasov type as well. Indeed the coefficients do not depend on the law of the solution of the above system but on the law of the solution of \eqref{eq:MR-FBSDE}. Under our current assumptions, since the processes $(X^{X_0}, Y^{X_0}, K^{X_0})$ are given, it is not hard to see that this system is well posed. Moreover, from weak uniqueness, the solution $(X^{t,x,X_0},Y^{t,x,X_0},Z^{t,x,X_0})$ only depends on $X_0$ through its law. Denoting this law by $\mu$, it is hence possible to write the solution $(X^{t,x,\mu},Y^{t,x,\mu},Z^{t,x,\mu})$ without specifying the choice of the lifted random variable $X_0$ with law $\mu$. Defining the mapping $u(t,x,\mu) : = Y_t^{t,x,\mu} : \R^+ \times \R^n \times \mathcal{P}_2(\R^n) \to \R$ we can show, arguing as in \cite{CCD14} (see proof of Proposition 2.2), that for any $[X_0]=\mu$, for all $s$ in $[t,T]$ : $v_{t,\mu}(s,x) = u(s,x,[X_s^{t,X_0}]), \, u(s,X_s^{t,x,[X_0]},[X_s^{t,[X_0]}]) = Y_s^{t,x,[X_0]}$ and $u(s,X_s^{t,X_0},[X_s^{t,X_0}]) = Y_s^{t,X_0}$ a.s..

We now have the following lemma whose proof is postponed at the end of the current section.
\begin{lemme}\label{lem:continuytiuback}
The decoupling field $u$ defined above is a continuous function from $[0,T] \times \R^n \times \mathcal{P}_2(\R^n) \to \R$.
\end{lemme}

It thus remains to check that this function solves, in the viscosity sense, \eqref{eq:obstacle}. {\color{black} Let $\varphi$ be a  test function and $(t,x,\mu)\in [0,T)\times \R^n\times \m P_2(\R^n)$ be such that $u-\varphi$ has a global minimum at $(t,x,\mu)$.} 
For any $s>t$ we have from the flow property that 
\begin{equation}
\E u(s,X_s^{t,x,\mu},[X_s^{t,\mu}]) = u(t,x,\mu) - \E \int_t^s f(r,X_{r}^{t,x,\mu},Y_r^{t,x,\mu})  \d r - \E\int_t^s \ld H([Y_r^{t,\mu}])(Y_r^{t,x,\mu})  \d K_r^{t,\mu},
\end{equation}
and from It\^{o}'s formula \eqref{I1} that:
\begin{equation}
\E \varphi(s,X_s^{t,x,\mu},[X_s^{t,\mu}]) = \varphi(t,x,\mu) + \E \int_t^s (\partial_t + {\pe\mathcal{L}}) \varphi (r,X_{r}^{t,x,\mu},[X_{r}^{t,\mu}]) \d r.
\end{equation}
{\color{black} Up to a translation of $\varphi$ we can assume that $u(t,x,\mu)=\varphi(t,x,\mu)$. Since $(t,x,\mu)$ is a global minimum we have
$$\E u(s,X_s^{t,x,\mu},[X_s^{t,\mu}]) - \E\varphi(s,X_s^{t,x,\mu},[X_s^{t,\mu}]) \geq 0,\quad s \in [t,T].$$
Hence we have:

\begin{eqnarray*}
\forall s\in [t,\tilde s],&& \E\int_t^s (\partial_t + {\pe\mathcal{L}})\varphi(r,X_r^{t,x,\mu},[X_r^{t,\mu}]) \d r \\
&& \quad +  \E \int_t^s f(r,X_{r}^{t,x,\mu},Y_r^{t,x,\mu})  \d r + \E\int_t^s \ld H([Y_r^{t,\mu}])(Y_r^{t,x,\mu})  \d K_r^{t,\mu} \leq 0.
\end{eqnarray*}

Assume now that 
$\mu$ is such that $H\big( {\color{black}  u(t,\cdot,\mu)\sharp \mu} \big)>0$ hence, by continuity, there exists $\tilde s>t$ such  for $s$ in $[t,\tilde s]$ we have $H([u(s,X_s^{t,\mu},[X_s^{t,\mu}])]) = H([Y_s^{t,X_0}]) >0$, $s\in [t,\tilde s]$ so that $\d K^{t,\mu}([t,s])=0,$ $s\in [t,\tilde s]$. Thus
\begin{eqnarray*}
&& \E\int_t^s (\partial_t + {\pe\mathcal{L}})\varphi(r,X_r^{t,x,\mu},[X_r^{t,\mu}])  \d r  +  \E \int_t^s f(r,X_{r}^{t,x,\mu},Y_r^{t,x,\mu})  \d r \leq 0.
\end{eqnarray*}
}

Dividing by $(s-t)$ and then  letting $s\to t$ we deduce
\begin{eqnarray}
(\partial_t + {\pe\mathcal{L}})\varphi(t,x,\mu) + f(t,x,\varphi(t,x,\mu))\leq 0,
\end{eqnarray}
and
$$\min\bigg\{(\partial_t + {\pe\mathcal{L}})\varphi(t,x,\mu)+ f(t,x,\varphi(t,x,\mu))\ ;\  H\big({\color{black} u(t,\cdot,\mu)\sharp \mu}\big)\bigg\} \leq 0.$$
If $\mu$ is now such that $H\big({\color{black}  u(t,\cdot,\mu)\sharp \mu}\big)=0$
then
$$\min\bigg\{(\partial_t +{\pe\mathcal{L}})\varphi(t,x,\mu)+ f(t,x,\varphi(t,x,\mu))\ ;\  H\big({\color{black}  u(t,\cdot,\mu)\sharp \mu} \big)\bigg\} \leq 0.$$

\end{proof}

\begin{proof}[Proof of Lemma \ref{lem:continuytiuback}]
Let $(t,x,\mu):=(t^0,x^0,\mu^0)$ and $(t^m,x^m,\mu^m)_{m }$ tends to $(t,x,\mu)$ as $m\to 0$. Let $X_0^m$ and $X_0$ be of law
$\mu^m$ and $\mu$ respectively. Recall that
\begin{equation}\label{eq:MR-FBSDE-decoupled-m}\left\lbrace\begin{array}{lll}
\displaystyle X_s^{t^m,x^m,X_0^m} &= x^m + \int_t^s b(r,X^{t^m,x^m,X_0^m}_r) \d r + \int_t^s \sigma(r,X^{t^m,x^m,X_0^m}_r) \d B_r, s\geq t \text{ and } \\
X_s^{t^m,x^m,X_0^m}&=x^m,\, s<t,\\
\displaystyle Y_s^{t^m,x^m,X^m_0} &= \phi(X_T^{t^m,x^m},[X_T^{t^m,X^m_0}]) + \int_s^T f(r,X_r^{t^m,x^m},Y_r^{t^m,x^m,X^m_0}) \d r - \int_s^T Z_r^{t^m,x^m,X^m_0} \d B_r\\
\displaystyle  &\qquad+ \int_s^T \ld H([Y_s^{t^m,X^m_0}])(Y_s^{t^m,x^m,X^m_0}) \d K_s^{t^m,[X^m_0]},\quad s \in [0,T].
\end{array}\right.
\end{equation}
It is straightforward to see that $X^{t^m,x^m,X_0^m} \to X^{t,x,X_0}$ in ${\mathcal S}^2$. Denoting by $(Y^{t^m,X_0^m},Z^{t^m,X_0^m})$ the solution of \eqref{eq:MR-FBSDE} with initial condition $(t^m,X_0^m)$, it is also easily to see (see proof of the uniqueness part of Theorem \ref{thm:existence_uniqueness}) that  $ Y^{t^m,X_0^m} \to Y^{t,X_0}$ in ${\mathcal S}^2$ and that $Z^{t^m,X_0^m} \to Z^{t,X_0}$ in $M^2$. The tricky part here consists in proving that the solution {\color{black} $Y_{t}^{t^m,x^m,X_0^m}$} of \eqref{eq:MR-FBSDE-decoupled-m} tends to the solution $Y_{t}^{t,x,X_0}$ of \eqref{eq:MR-FBSDE-decoupled}. Recall indeed that when considering the decoupled flow \eqref{eq:MR-FBSDE-decoupled} of \eqref{eq:MR-FBSDE} we are not considering solutions of MR-FBSDE since the process $K$ is only an input in these equations so that the Skorokhod condition is not satisfied anymore.\\

To prove that $Y_{t^m}^{t^m,x^m,X_0^m} \to Y_{t}^{t,x,X_0}$ we first notice that these quantities are deterministic so 
\begin{eqnarray}\label{eq:decoupe_y}
Y_{t^m}^{t^m,x^m,X_0^m} - Y_{t}^{t,x,X_0} &=& \E \left[ Y_{t^m}^{t^m,x^m,X_0^m} - Y_{t}^{t,x,X_0} \right]\notag\\
&\leq &  \left| \E \left[ Y_{t^m}^{t^m,x^m,X_0^m} - Y_{t}^{t^m,x^m,X_0^m}\right] \right| + \left|\E \left[Y_{t}^{t^m,x^m,X_0^m}  - Y_{t}^{t,x,X_0} \right]\right|.
\end{eqnarray}
We have
\begin{eqnarray*}
\left|\E \left[ Y_{t^m}^{t^m,x^m,X_0^m} - Y_{t}^{t^m,x^m,X_0^m} \right] \right| &\leq & 
\left|\E \int_{t^m}^t \left| f(r,X_r^{t^m,x^m,X_0^m},Y_r^{t^m,x^m,X_0^m}) \right|  dr \right| \\
&&+ \left|\E \int_{t^m}^t | \ld H([Y_r^{t^m,X_0^m}])(Y_r^{t^m,x^m,X_0^m}) |dK_r^m\right|.
\end{eqnarray*}
Note now that since $f$ satisfies the linear growth assumption in \A{A'} (so that $\sup_{t \leq T} \E[|f(t,X_t,Y_t)|^2]<+\infty$) we can deduce from the proof of \emph{Step 1} and Proposition \ref{prop:penalized} (see Remark \ref{constructKpart}) that the map $[0,T] \ni r\mapsto K^m_r$ is Lispchitz continuous (uniformly in $m$). Hence, there exists $k^m$ satisfying $\sup_{r \in [0,T]}k^m_r \leq c$ (with $c:=c(\A{\text{A'}})>0$ independent of $m$) such that for all $r$ in $[0,T]$ we have $dK_r^m = k_r^m dr$. Using this property, together with assumptions in \A{A} and \A{A'} it clear that the above contribution also tends to 0. 

It then remains to show that the second term in the right hand side of \eqref{eq:decoupe_y} also tends to 0. This last fact is a little bit more involved. We start by setting for all $s$ in $[0,T]$,
$$L^m_s=\int_0^s \ld H([Y_r^{t^m,X^m_0}])(Y_r^{t^m,x^m,X^m_0}) \d K_r^{t^m,X^m_0},$$
and
$$L_s=\int_0^s \ld H([Y_r^{t,X_0}])(Y_r^{t,x,X_0}) \d K_r^{t,X_0}.$$
From the  FBSDE \eqref{eq:MR-FBSDE}, we have, for all $s$ in $[0,T]$,
$$L^m_s=-Y_s^{t^m,X^m_0} +Y_{t^m}^{t^m,X^m_0} + \int_{t^m}^s f(r,X_r^{t^m,X^m_0},Y_r^{t^m,X^m_0}) \d r - \int_{t^m}^s Z_r^{t^m,x^m,X^m_0} \d B_r,$$
and
$$L_s=-Y_s^{t,X_0} +Y_{t}^{t,X_0} + \int_{t}^s f(r,X_r^{t,X_0},Y_r^{t,X_0}) \d r - \int_{t}^s Z_r^{t,X_0} \d B_r,$$
from which we deduce that for every $s$ in $[0,T]$:
\begin{equation}\label{eq:conv:de:L}
\lim_{m \to + \infty} L^m_s=L_s.
\end{equation}
Let us now rewrite the BSDE  \eqref{eq:MR-FBSDE-decoupled-m} in the following form:
\begin{eqnarray}\label{eq:MR-FBSDE-decoupled-m-L} Y_s^{t^m,x^m,X^m_0} &=& \phi(X_T^{t^m,x^m},[X_T^{t^m,X^m_0}]) + \int_s^T f(r,X_r^{t^m,x^m},Y_r^{t^m,x^m,X^m_0}) \d r - \int_s^T Z_r^{t^m,x^m,X^m_0} \d B_r\nonumber \\
& &\displaystyle \quad + \int_s^T \ld H([Y_s^{t^m,X^m_0}])(Y_s^{t^m,x^m,X^m_0}) \left(\ld H([Y_s^{t^m,X^m_0}])(Y_s^{t^m,x^m,X^m_0}) \right)^{-1}\d L^m_s.
\end{eqnarray}
Setting
\begin{eqnarray*}
A_s^m &=& \int_0^s\ld H([Y_s^{t^m,X^m_0}])(Y_s^{t,x,X_0}) \left(\ld H([Y_s^{t^m,X^m_0}])(Y_s^{t^m,x^m,X^m_0}) \right)^{-1}\d L^m_s\\
& &-\int_0^s\ld H([Y_s^{t,X_0}])(Y_s^{t,x,X_0}) \left(\ld H([Y_s^{t,X_0}])(Y_s^{t,x,X_0}) \right)^{-1}\d L_s,
\end{eqnarray*}
and
\begin{eqnarray*}
U^m=Y^{t^m,x^m,X^m_0}-Y^{t,x,X_0}+A^m,\quad V^m=Z^{t^m,x^m,X_0^m}-Z^{t,x,X_0},
 \end{eqnarray*}
we deduce  that $(U^m,V^m)$ satisfy the following BSDE,
\begin{equation*}
U_s^m=\eta^m+\int_s^T (\alpha_r^m U_r^m dr + \gamma_r^m U_r^m dL^m_r) -\int_s^T V_r^mdB_r + \int_s^T d F_r^m,\quad s\in [0,T],
\end{equation*}
where:
\begin{eqnarray*}
F_r^m = \int_0^r \left(\left(-\alpha_v^mA_v^m+\beta_v^m\right) dv -\gamma_v^mA_v^m dL^m_v\right),
\end{eqnarray*}
with
\begin{eqnarray*}
\eta^m&= &\phi(X_T^{t^m,x^m,X_0^m},[X_T^{t^m,X^m_0}]) - \phi(X_T^{t,x,X_0},[X_T^{t,X_0}]) \\
&&+\int_0^T\ld H([Y_s^{t^m,X^m_0}])(Y_s^{t,x,X_0}) \left(\ld H([Y_s^{t^m,X^m_0}])(Y_s^{t^m,x^m,X^m_0}) \right)^{-1}\d L^m_s\\
& &-\int_0^T\ld H([Y_s^{t,X_0}])(Y_s^{t,x,X_0}) \left(\ld H([Y_s^{t,x,X_0}])(Y_s^{t,x,X_0}) \right)^{-1}\d L_s,\\
\alpha^m&=& \frac{f(r,X_r^{t^m,x^m,X_0^m},Y_r^{t^m,x^m,X^m_0}) -f(r,X_r^{t^m,x^m,X_0^m},Y_r^{t,x,X_0})}{Y_s^{t^m,x^m,X^m_0}-Y_s^{t,x,X_0}}, \\
\beta^m&=&f(r,X_r^{t^m,x^m,X_0^m},Y_r^{t,x,X_0})- f(r,X_r^{t,x,X_0},Y_r^{t,x,X_0}),  \\
\gamma^m&=&\frac{ \left(\ld H([Y_s^{t^m,X^m_0}])(Y_s^{t^m,x^m,X^m_0})-\ld H([Y_s^{t^m,X^m_0}])(Y_s^{t,x,X_0})\right)}{Y_s^{t^m,x^m,X^m_0}-Y_s^{t,x,X_0}}\\
&& \times \left(\ld H([Y_s^{t^m,X^m_0}])(Y_s^{t^m,X^m_0}) \right)^{-1},\\
\end{eqnarray*}
using the convention $0/0=0$. The definition of $L^m$ and the Lipschitz property of $K^m$ yield
\begin{eqnarray}
 \gamma_r^m U_r^m dL^m_r = \frac{ \left(\ld H([Y_s^{t^m,X^m_0}])(Y_s^{t^m,x^m,X^m_0})-\ld H([Y_s^{t^m,X^m_0}])(Y_s^{t,x,X_0})\right)}{Y_s^{t^m,x^m,X^m_0}-Y_s^{t,x,X_0}} U_r^m k_r^m dr,
\end{eqnarray}
with $k_r^m \leq c$ where $c>0$ does not depend on $m$. Note also that $\alpha^m$, $\beta^m$ and $\gamma^m$ are bounded uniformly in $m$. We then obtain for every $1<p<2$:
$$\E |U_s^m|^p = \E \left[ |\E [U_s^m | \mathcal{F}_s] |^p\right] \leq C\left\{ \E |\eta^m|^p + \int_s^T \E |U_r^m|^p dr + \E \left[ \left( \int_0^T d|F^m|_r\right)^p \right] \right\}.$$
Hence, applying Gronwall's inequality, we obtain that there exists $c:=c(T)>0$ such that 
$$ \E |U_s^m|^p \leq c \left(\E |\eta^m|^p  \right)  +  \E \left[ \left( \int_0^T d|F^m|_r\right)^p \right].$$
Let us now prove that the right hand side of the above equation tends to 0.\\

Note that from \eqref{eq:conv:de:L} we obtain that $\eta^m \to 0$ and from Lipschitz property of $r \mapsto K^m_r$ we know that there exists $C:=C(\A{A'},T)>0$ such that $\E \left[ \int_0^T \left| \ld H([Y_s^{t^m,X_0^m}])(Y_s^{t,x,X_0}) dK_r^m \right|^2 \right] \leq C$ so that $\E\left[ |\eta^m|^p\right] \to 0$. 

We consider now the term $\E \left[ |F^m|_T^p\right]$. First note that from the Lispchitz property of $r\mapsto K^m_r$ and the definition of $\gamma^m$ and $L^m$:
\begin{eqnarray}
\gamma_r^m A_r^m d L_r^m &=& \frac{ \left(\ld H([Y_s^{t^m,X^m_0}])(Y_s^{t^m,x^m,X^m_0})-\ld H([Y_s^{t^m,X^m_0}])(Y_s^{t,x,X_0})\right)}{Y_s^{t^m,x^m,X^m_0}-Y_s^{t,x,X_0}}A_r^m k_r^mdr,
\end{eqnarray}
so that there exists $C:=C(\A{A},T)>0$ (that may change from line to line) such that
\begin{eqnarray*}
\E \left[ |F^m|_T^p\right] &\leq &C \left\{\E \left[ \int_0^T( |A_r^m|^p + |\beta_r^m|^p ) dr \right]  + \E \left|\int_0^T A_r^m dK_r^m\right|^p\right\}\\
& \leq & C \E \left[ \int_0^T( |A_r^m|^p + |\beta_r^m|^p)  dr \right].
\end{eqnarray*}
It is clear that $\E [|\beta_r^m|^p]\to 0$. Let us now investigate the convergence of $\E[|A_r^m|^p]$. We have from the definition of $A_s^m$ and $L^m_s$ that  
\begin{eqnarray}\label{eq:rewriteA}
A_s^m&=&\int_0^s  \ld H([Y_s^{t^m,X^m_0}])(Y_s^{t,x,X_0}) dK_r^m -\int_0^s  \ld H([Y_s^{t,X_0}])(Y_s^{t,x,X_0}) dK_r\notag\\
&=&\int_0^s \left(   \ld H([Y_s^{t^m,X^m_0}])(Y_s^{t,x,X_0})-\ld H([Y_s^{t,X_0}])(Y_s^{t,x,X_0}) \right) k_r^m dr \notag\\
&& + \int_0^s  \ld H([Y_s^{t,X_0}])(Y_s^{t,x,X_0}) \left(dK_r^m-dK_r \right).
\end{eqnarray}
The first term in the right hand side tends to 0 by continuity and uniform boundedness of $k_r^m$. To prove the convergence to 0 of the second term in the right hand side it is sufficient to prove, thanks to an approximation argument, that this holds when the integrand is any step function.   The problem hence reduces to prove the convergence of $K^m_s$ to $K_s$ for every $s$ in $[0,T]$. Recall 
$$K_s^m = \int_0^s [\ld H([Y_r^{t^m,X_0^m}])(Y_r^{t^m,x^m,X_0^m})]^{-1} d L_s^m,\quad K_s = \int_0^s [\ld H([Y_r^{t,X_0}])(Y_r^{t,x,X_0})]^{-1} d L_s.$$
We then have
\begin{eqnarray*}
K_s^m-K_s &=& \int_0^s  [\ld H([Y_r^{t^m,X_0^m}])(Y_r^{t^m,x^m,X_0^m})]^{-1} -  [\ld H([Y_r^{t,X_0}])(Y_r^{t,x,X_0})]^{-1} dL^m_r\\
&& + \int_0^s  [\ld H([Y_r^{t,X_0}])(Y_r^{t,x,X_0})]^{-1}d(L_s^m-L_s).
\end{eqnarray*}
On the one hand we have $r \mapsto  [\ld H([Y_r^{t^m,X_0^m}])(Y_r^{t^m,x^m,X_0^m})]^{-1}$ is continuous and we also have that for every step function $\varphi$, $ \int_0^s  \varphi(s)d(L_s^m-L_s) \to 0$ so that
$$ \int_0^s  [\ld H([Y_r^{t,X_0}])(Y_r^{t,x,X_0})]^{-1}d(L_s^m-L_s) \to 0.$$
On the other hand,
\begin{eqnarray*}
&&\E\sup_{s\in [0,T]} \left| \int_0^s  [\ld H([Y_r^{t^m,X_0^m}])(Y_r^{t^m,x^m,X_0^m})]^{-1} -  [\ld H([Y_r^{t,X_0}])(Y_r^{t,x,X_0})]^{-1} dL^m_r \right|\\
& \leq & \E\left[\int_0^T \left| [\ld H([Y_r^{t^m,X_0^m}])(Y_r^{t^m,x^m,X_0^m})]^{-1} -  [\ld H([Y_r^{t,X_0}])(Y_r^{t,x,X_0})]^{-1}\right| d|L^m|_r\right]\\
& \leq & C_T\int_0^T \E^{1/2} \left[\left| [\ld H([Y_r^{t^m,X_0^m}])(Y_r^{t^m,x^m,X_0^m})]^{-1} -  [\ld H([Y_r^{t,X_0}])(Y_r^{t,x,X_0})]^{-1}\right|^2\right] \\
&&\times \E^{1/2}\left[ \left|\ld H([Y_r^{t^m,X_0^m}])(Y_r^{t^m,x^m,X_0^m})\right|^2\right]d r\\
& \leq & C_T\int_0^T \E^{1/2} \left[\left| [\ld H([Y_r^{t^m,X_0^m}])(Y_r^{t^m,x^m,X_0^m})]^{-1} -  [\ld H([Y_r^{t,X_0}])(Y_r^{t,x,X_0})]^{-1}\right|^2\right] M dr,
\end{eqnarray*}
and it is clear that the last term in the above right hand side also tends to 0. 

For any subsequence $(m')$, there exists $(m")\subset (m')$ such that
for every $s$ in $[0,T]$, $ K_s^{m"} \to K_s$, and then $ A_s^{m"} \to 0$. Finally,  we deduce from our assumptions and the Lipschitz regularity of $K$ that the family $(A^{m^"})_{m^" \geq 0}$ is uniformly bounded in  $L^2([0,T]\times \Omega)$ so that it converges to $0$ in $L^p([0,T]\times \Omega)$ for any $1<p<2$.
This means that $(A^{m})_{m \geq 0}$ converges to $0$ in $L^p([0,T]\times \Omega)$ for any $1<p<2$.

 We  obtain that for all $s$ in $[0,T]$, $1<p<2$: $\E[ |U^m_s|^p]\to 0,$ which means that in $L^p$, for every $s$ in $[0,T]$,
$$Y_s^{t^m,x^m,X^m_0}\to Y_s^{t,x,X_0}.$$
Using this in \eqref{eq:decoupe_y}, this concludes the proof.
 
\end{proof}
%
%
%
%
%
%
%

\newpage
\end{document}